\documentclass[11pt,a4paper,oldfontcommands]{memoir}

\usepackage[T1]{fontenc}
\usepackage[ansinew]{inputenc}
\usepackage[brazil]{babel}
\usepackage[all]{xy}
\usepackage{color}
\usepackage[left]{showlabels}
\usepackage{enumerate}
\usepackage{indentfirst}
\usepackage{comment}
\usepackage{nicefrac}
\usepackage{pdfpages}

\OnehalfSpacing

\chapterstyle{madsen}   
\setsecheadstyle{\Large\bfseries\sffamily\raggedright}
\setsubsecheadstyle{\large\bfseries\sffamily\raggedright}
\setsubsubsecheadstyle{\bfseries\sffamily\raggedright}

\makepagestyle{plain}
\makeevenfoot{plain}{\thepage}{}{}
\makeoddfoot{plain}{}{}{\thepage}
\makeevenhead{plain}{}{}{}
\makeoddhead{plain}{}{}{}

\usepackage[total={210mm,297mm},left=20mm,right=20mm,bindingoffset=10mm,top=25mm,bottom=25mm]{geometry}

\usepackage[backref=page]{hyperref}  
                 \hypersetup{ pdfborder={0 0 0}, 
                              colorlinks=true, 
                              citecolor=blue,
                              linktoc=page,
                              pdfauthor={}, 
                              pdftitle={}}
\renewcommand{\ref}{\hyperref}

\usepackage{lmodern} 

\usepackage{graphicx} 

\usepackage{amsmath}
\usepackage{amsthm}
\usepackage{amsfonts}
\usepackage{multicol}

\usepackage{amssymb}
\usepackage{mathdots}

\usepackage{enumerate}
\usepackage{indentfirst}

\newcommand{\Ccal}{\mathcal{C}}
\newcommand{\Lcal}{\mathcal{L}}
\newcommand{\Ocal}{\mathcal{O}}
\newcommand{\Mcal}{\mathcal{M}}

\newcommand{\Hcal}{\mathcal{H}}
\newcommand{\Tcal}{\mathcal{T}}

\newcommand{\N}{\mathbb{N}}
\newcommand{\G}{\mathbb{G}}
\newcommand{\Z}{\mathbb{Z}}

\renewcommand{\P}{\mathbb{P}}
\newcommand{\A}{\mathbb{A}}
\newcommand{\Q}{\mathbb{Q}}
\newcommand{\C}{\mathbb{C}}

\DeclareMathOperator{\euler}{\chi_{top}} 

\newcommand{\csm}{c_{\mathrm{SM}}}
\newcommand{\cma}{c_{\mathrm{Ma}}}
\newcommand{\cf}{c_{\mathrm{F}}}
\newcommand{\milnor}{\Mcal}

\newcommand{\abs}[1]{\left\lvert#1\right\rvert}

\newcommand{\binomci}{\binom{r}{\left[\frac{i}{2}\right]}
\binom{r}{\left[\frac{i+1}{2}\right]}}

\DeclareMathOperator{\sing}{Sing}
\DeclareMathOperator{\grad}{grad}

\DeclareMathOperator{\codim}{codim}

\newtheorem{theorem}{Theorem}[chapter]

\newtheorem{thm}[theorem]{Teorema}
\newtheorem{teo}[theorem]{Teorema}
\newtheorem*{teo*}{Teorema}
\newtheorem{prop}[theorem]{Proposição}
\newtheorem{corol}[theorem]{Corolário}
\newtheorem{cor}[theorem]{Corolário}
\newtheorem{lem}[theorem]{Lema}
\newtheorem{lema}[theorem]{Lema}
\newtheorem{conjectura}[theorem]{Conjectura}
\newtheorem{afirm}[theorem]{Afirmação}

\theoremstyle{definition}
\newtheorem{defi}[theorem]{Definição}
\newtheorem{rmk}[theorem]{Observação}
\newtheorem{obs}[theorem]{Observação}
\newtheorem{exem}[theorem]{Exemplo}
\newtheorem{exemplo}[theorem]{Exemplo}

\DeclareMathOperator{\Image}{Im}

\definecolor{pink}{rgb}{0.89, 0.44, 0.48}
\definecolor{cornflowerblue}{rgb}{0.39, 0.58, 0.93}

\newcommand{\verm}[1]{{\color{red}#1}}

\newcommand{\Sec}{\operatorname{Sec}}
\newcommand{\Tan}{\operatorname{Tan}}
\newcommand{\cat}{\mathcal{C}}
\newcommand{\eddeg}{\operatorname{EDdeg}}
\newcommand{\geddeg}{\operatorname{gEDdeg}}

\newcommand{\Mdois}{\textsf{Macaulay2}}

\begin{document}

	%
	%
	
	\thispagestyle{empty}
	
	{
		\sffamily
		\centering
		\Large
		
		\includegraphics[scale=0.25]{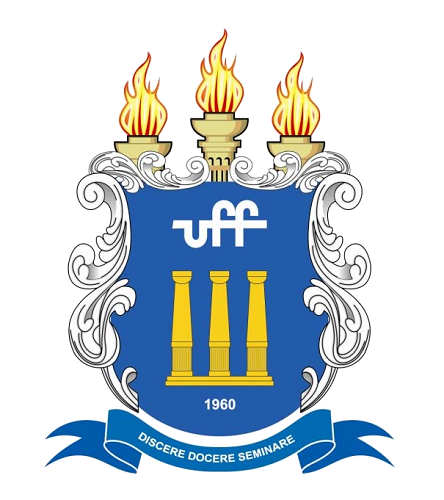}
		
		~\vspace{2cm}
		
		Universidade Federal Fluminense

		\vspace{\fill}
		
		{\huge
			Classes características e secantes de curvas racionais normais
			
		}

		\vspace{3.5cm}
		
		{\LARGE
			
			Jefferson Ribeiro Nogueira
			
		}
		
		\vspace{\fill}

		{\large Niterói }

		{\large  06 de abril de 2020}

	}	
	
	\clearpage
	
	%
	%
	
	\thispagestyle{empty}
	
	{
		
		\sffamily
		\centering
		\Large
		
		~\vspace{\fill}

		{\huge
			
			Classes características e secantes de curvas racionais normais
			
		}

		\vspace{3.5cm}

		{\LARGE
			
			Jefferson Ribeiro Nogueira
			
		}

		\vspace{3.5cm}
		
		{\normalsize
			\raggedleft	
			\begin{minipage}{210pt}
				Tese submetida ao Programa de Pós-Graduação
				em Matemática da Universidade Federal Fluminense
				como requisito parcial para a obtenção do grau de 
				Doutor em Matemática.
			\end{minipage}
			
		}

		\vspace{3.5cm}

		Orientador: Prof. Nivaldo Medeiros
		
		\vspace{\fill} 
		
		{\large Niterói }

		{\large 06 de abril de 2020}
		
	}	
	
	\clearpage

\thispagestyle{empty}

{
\centering

\includepdf{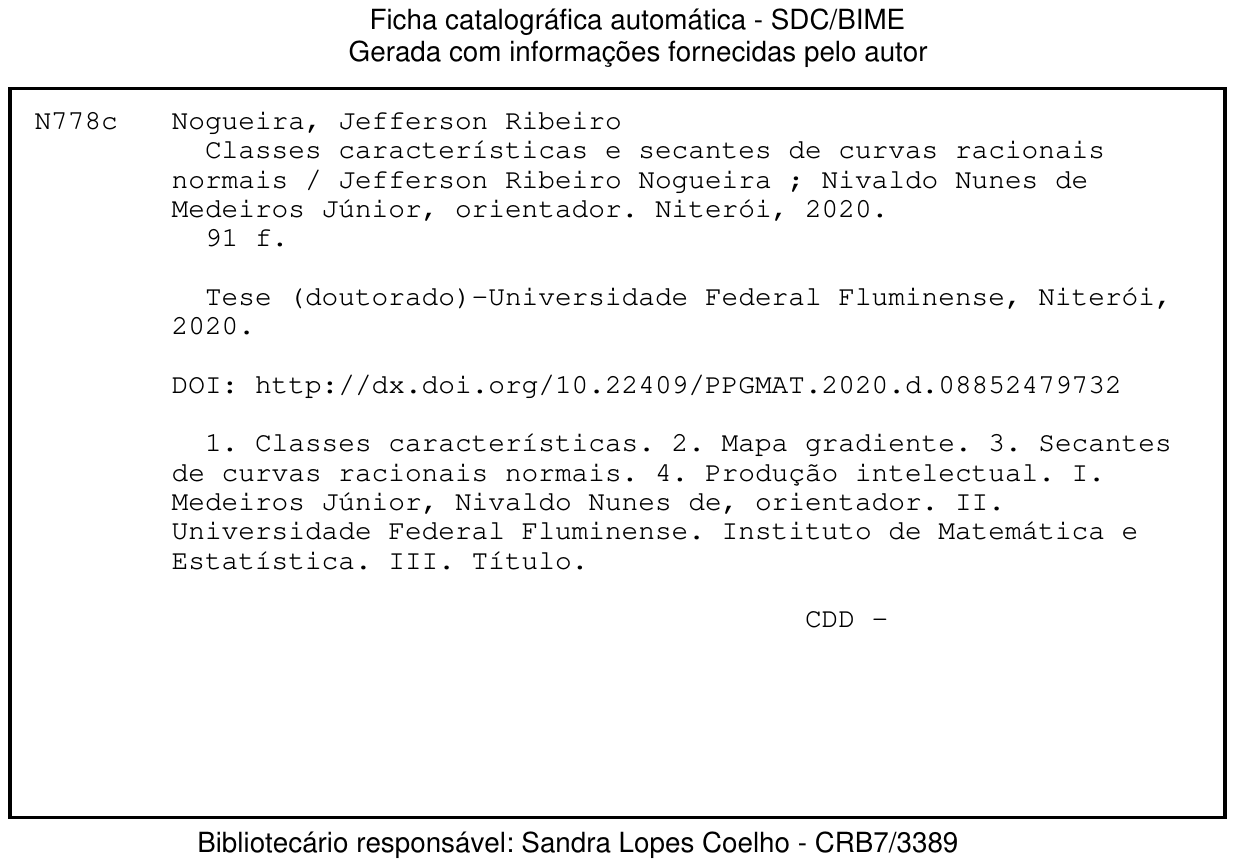}

}
	
	\clearpage

	\thispagestyle{empty}

	{
		\centering

		
	}
	\clearpage
	
	\thispagestyle{empty}
	
	{
		
		\sffamily\centering
		
		\textbf{Tese de Doutorado da Universidade Federal Fluminense}
		
		\vspace{1cm}
		
		por 
		
		\vspace{1cm}
		
		\textbf{Jefferson Ribeiro Nogueira}
		
		\vspace{1.5cm}
		
		apresentada ao Programa de Pós-Graduação em Matemática como requesito parcial para a obtenção do grau de
		
		\vspace{1.5cm}
		
		\textbf{Doutor em Matemática}
		
		\vspace{1.5cm}
		
		Título da tese:  Classes características e secantes de curvas racionais normais

		\hrulefill
		
		\begin{minipage}{0.8\textwidth}
			\centering
			\textbf{ Classes características e secantes de curvas racionais normais}
		\end{minipage}
		~\vspace{5pt}
		\hrule
		
		\vspace{1cm}
		
		\textit{Defendida publicamente em  06 de abril de 2020.}
		
		\vspace{1cm}
		
		{
			\raggedright
			Diante da banca examinadora composta por:
		}
		
		~\vspace{5pt}
		%

		\begin{tabular}{lll}
			Prof. Nivaldo Nunes de Medeiros Júnior& Universidade Federal Fluminense & Orientador\\
			Prof. Thiago Fassarella do Amaral& Universidade Federal Fluminense & Examinador\\
			Profa. Viviana Ferrer Cuadrado& Universidade Federal Fluminense & Examinadora\\
			Profa. Maral Mostafazadehfard& Universidade Federal do Rio de Janeiro & Examinadora\\
			Prof. Rodrigo José Gondim Neves& Universidade Federal Rural de Pernambuco & Examinador\\
			Prof. Marco Pacini& Universidade Federal Fluminense & Examinador\\
		\end{tabular}
		
		~\vspace{\fill}

		~\vspace{\fill}
		
	}
	
	\clearpage

	\thispagestyle{empty}
	
	{
		{\centering
			
			\textbf{DECLARAÇÃO DE CIÊNCIA E CONCORDÂNCIA DO(A) ORIENTADOR(A)}
		}

		~\vspace{2cm}
		
		Autor da Tese: Jefferson Ribeiro Nogueira
		
		Data da defesa: 06/04/2020
		
		\vspace{0.2cm}
		
		Orientador: Nivaldo Nunes de Medeiros Júnior

		\vspace{2cm}
		
		Para os devidos fins, declaro \textbf{estar ciente} do conteúdo desta \textbf{versão corrigida} elaborada em atenção às sugestões  dos membros da banca examinadora na sessão de defesa do trabalho, manifestando-me \textbf{favoravelmente} ao seu encaminhamento e publicação no \textbf{Repositório Institucional da UFF}.
		
		~\vspace{1cm}
		
		{
			\raggedright
			Niteroi, 12/01/2021.
		}
		~\vspace{1cm}
		
		\begin{center}
			\begin{minipage}{200pt}
				\centering 	
				\hrule
				
				Nivaldo Nunes de Medeiros Júnior
			\end{minipage}
		\end{center}

	}
	
	
	
	
	
	
	\clearpage
	
	\thispagestyle{empty}
	
	{
		
		\sffamily
		
		{\Large \centering	
			
			AGRADECIMENTOS
			
		}
		
		~\vspace{1cm}
		
		Aos meus amores, esposa Gláucia, e filhos Alyne e Arthur, por iluminarem meus dias, e por serem meus motivos em querer ser e estar cada vez melhor. O meu amor é de vocês!
		
		\smallskip
		
		Aos meus pais por tudo o que sempre foram e fizeram por mim.
		
		\smallskip
		
		Ao meu orientador Nivaldo Medeiros, sem o qual este trabalho não seria possível. Agradeço pelo início, meio, ... por cada recomeço e por não ter desistido. Serei sempre grato por todo apoio, compreensão, exemplo e toda a matemática.
		
		\smallskip
		
		Aos membros da banca pela participação, leitura e contribuições a este trabalho.
		
		\smallskip
		
		Aos matemáticos Alex Abreu, Giovanni Staglianò, Giuseppe Borrelli, Israel Vainsencher, Juliana Coelho, Maral Mostafazadehfard, Thiago Fassarella e Viviana Ferrer por toda matemática que compartilharam, e pelas contribuições diretas ou indiretas que deram a este trabalho.
		
		\smallskip
		
		A todos os meus familiares, amigos e amigas com quem convivi neste tempo e que, de diversas maneiras, contribuíram para este trabalho. Destaque especial aos amigos Jaqueline Siqueira, Luiz Viana, Reginaldo Demarque e Rômulo Rosa por cada momento nesse processo, vocês foram incríveis!
		
		\smallskip
		
		O presente trabalho foi realizado com apoio da Coordenação de Aperfeiçoamento de Pessoal de Nível Superior - Brasil (CAPES) - Código de Financiamento 001.
		
		\smallskip
		
		Esse trabalho foi apoiado com uma bolsa de doutorado da Fundação de Apoio à Pesquisa do Rio de Janeiro (FAPERJ).

	}
	
	\clearpage
	
	\thispagestyle{empty}
	
	{
		
		\sffamily
		
		{\Large	\centering
			
			RESUMO
			
		}
		
		~\vspace{1cm}
		
		Estudamos classes características de hipersuperfícies no espaço projetivo complexo, com ênfase nas secantes de curvas racionais normais.
		
		
		Para $\Sec_kC \subset \P^n$, a secante de $k$ pontos de uma curva
		racional normal $C\subset\P^n$, calculamos a série de Hilbert e a característica de Euler topológica. 
		
		Quando $n=2r$ e $k=r$, caso
		em que $\Sec_rC\subset\P^{2r}$ é uma hipersuperfície,
		mostramos que a dual $(\Sec_rC)^*$ é isomorfa a variedade
		de Veronese $\nu_2(\P^r)$, donde obtemos, para $\Sec_r C$, 
		fórmulas para a classe de Mather, o grau distância Euclidiana genérica, e seus graus polares. Mais ainda, apresentamos uma fórmula explícita para o grau topológico do mapa gradiente $\phi_r\colon \P^{2r}\dashrightarrow \P^{2r}$ associado à $\Sec_r C$, e como consequência obtemos uma resposta positiva para uma conjectura de M. Mostafazadehfard e A. Simis: \textit{para $r\geq 2$, a hipersuperfície $\Sec_r C\subset\P^{2r}$ não é homaloidal}. 
		
		A partir do cálculo em casos particulares somos levados a uma
		conjectura, a saber, fórmulas explícitas para os graus projetivos
		do mapa gradiente $\phi_r$ e para a classe de Schwartz-MacPherson $\csm(\Sec_rC)\in A_*\P^{2r}$, para todo $r$. Concluímos apresentando evidências que indicam a validade da nossa conjectura.
		\\
		

		\textbf{Palavras-chave:} Classes características. Mapas gradientes. Secantes de curvas racionais normais.
		
	}
	
	\clearpage
	
	\thispagestyle{empty}
	
	{
		
		\sffamily
		
		{\Large\centering
			
			ABSTRACT
			
		}
		
		~\vspace{1cm}
		
		We study characteristic classes of hypersurfaces in the complex projective space, with emphasis on secants to rational normal curves.
		
		For $\Sec_k C\subset\P^{n}$, the secant of $k$ points to a rational normal curve $C\subset \P^n$, we compute the Hilbert series and the topological Euler characteristic.
		
		For $n=2r$ and $k=r$, the case when $\Sec_r C\subset\P^{2r}$ is a hypersurface, we show that the dual $(\Sec_r C)^*$ is isomorphic to the Veronese variety $\nu_2(\P^r)$, from which we obtain, for $\Sec_r C$, formulas for the Mather class, the generic Euclidean distance degree and its polar degrees. Furthermore, we present an explicit formula for the topological degree of the gradient map $\phi_r \colon \P^{2r} \dashrightarrow \P^{2r} $ associated with $\Sec_r C$, and as a consequence we obtain an affirmative answer for a conjecture by M. Mostafazadehfard and A. Simis: \textit{for $r \geq 2$, the hypersurface $\Sec_r C\subset\P^{2r}$ is not homaloidal}.
		
		From computations in particular cases we are led to a conjecture, namely, explicit formulas for the projective degrees of the gradient map $\phi_r$ and the Schwartz-MacPherson class $\csm(\Sec_rC)\in A_*\P^{2r}$, for all $r$. We conclude by presenting evidence that indicates the validity of our conjecture.
		\\
		
		 Secants to rational normal curves.
		
	}

	%
	%
	
	

	%
	%
	
	\clearpage
	
	\listoftables
	
	\clearpage

\tableofcontents

%
%

\newtheorem{innercustomthm}{Teorema}
\newenvironment{neoteo}[1]
{\renewcommand\theinnercustomthm{#1}\innercustomthm}
{\endinnercustomthm}
\newtheorem{innercustomconj}{Conjectura}
\newenvironment{neoconj}[1]
{\renewcommand\theinnercustomconj{#1}\innercustomconj}
{\endinnercustomconj}
\newtheorem{innercustomcor}{Corolário}
\newenvironment{neocor}[1]
{\renewcommand\theinnercustomcor{#1}\innercustomcor}
{\endinnercustomcor}

\chapter*{Introdu\c{c}\~ao}
\addcontentsline{toc}{chapter}{Introdu\c{c}\~ao}
\markboth{INTRODUÇÃO}{INTRODUÇÃO}      

Para uma variedade alg\'ebrica n\~ao-singular $X$, a classe de Chern de $X$ \'e definida por $c(X) = c(TX)\cap [X]$, onde $TX$ \'e o fibrado tangente. Esta classe carrega importantes informa\c{c}\~oes geom\'etricas e topol\'ogicas da variedade; por exemplo, o grau desta classe \'e a caracter\'istica de Euler topol\'ogica de $X$ (Teorema de Poincar\'e-Hopf). \'E natural buscar por uma analogia para o caso em que $X$ \'e singular. Em 1965, Marie-H\'el\`ene Schwartz estendeu a no\c{c}\~ao de classe de Chern para variedades complexas anal\'iticas estudando campos vetoriais especiais nas singularidades de $X$ (\cite{Sch65a} e \cite{Sch65b}). De maneira independente, R. MacPherson em 1974 também definiu classes de Chern para variedades complexas algébricas \cite{Mac74}. Mais tarde foi demonstrado que de fato essas duas construções coincidem (\cite{BS81} e \cite{AB08}).

Na construção de MacPherson são definidas as classes de Schwartz-MacPherson e de Mather, as quais s\~ao elementos do grupo de Chow de uma variedade alg\'ebrica complexa. Essas defini\c{c}\~oes estendem a no\c{c}\~ao de classe de Chern para variedades singulares. A classe de Schwartz-MacPherson possui boas propriedades funtoriais, e vale que o grau desta classe coincide com a caracter\'istica de Euler da variedade.



Por outro lado, as variedades secantes constituem um assunto clássico em Geometria Algébrica (G. Castelnuovo, G. Scorza, F. Severi, A. Terracini). Além da riqueza geométrica inerente, as secantes são relevantes em diversos campos da matem\'atica e em aplicações, incluindo Combinatória, Teoria da Complexidade, Estatística e Física. Uma excelente apresentação de conexões, problemas e rica lista de referências é encontrada em \cite{4lectures}.   Recorde que
para $X\subset \P^n$, sua $k$-secante $\Sec_k X$ é o fecho
da união dos $(k-1)$-planos gerados por $k$ pontos gerais de $X$. 
\medskip

Nesta tese tratamos destes dois tópicos. Especificamente, 
abordamos o problema de determinar classes características de secantes de curvas racionais normais.  Este objetivo nos leva
a estudar aspectos bastante diversos 
da geometria dessas variedades.

 Aproveitamos 
para anunciar que trabalhamos sobre o corpo dos números complexos.

\bigskip


O ponto de partida é 
um resultado notável devido à P. Aluffi \cite[Theorem~1.4]{Alu99a} que nos permite 
calcular a classe de Schwartz-MacPherson de uma
\emph{hipersuperfície} $X\subset\P^n$
em termos do seu lugar singular $Y=\sing X$:
\begin{equation}
\label{eq-Aluffi-CSM}
\csm(X) = c(T\P^n)\cap \left(
s(X,\P^n) + c(\Ocal(X))^{-1} \left(
s(Y,\P^n)^\vee \otimes \Ocal(X)
\right)
\right) 
\qquad\in A_*\P^n.
\end{equation}
Seja qual for o significado dos símbolos nesta fórmula,
ela nos diz que ``basta'' calcular a classe
de Segre $s(Y,\P^n)$ do lugar singular.
Em princípio isto não ajuda muito. Citando P.~Aluffi \cite{Alu02}:
\emph{``%
	Segre classes are in general extremely hard to compute.
	Why? Because blow-ups are hard to compute.''
}%
Por outro lado, para hipersuperfícies,
o lugar singular de $X$ coincide com o lugar de base do
mapa racional 
$$\grad X\colon \P^n\dashrightarrow\P^n$$
 dado pelas derivadas
parciais do polinômio que define $X$, chamado \emph{o mapa gradiente} de $X$. Por sua vez
a classe de Segre do lugar de base 
está diretamente relacionada com os 
\emph{graus projetivos} (ou \emph{multigrau}) $d_0,\dotsc,d_n$ do mapa:
\textit{grosso modo}, aqui o $i$-ésimo grau projetivo é o grau da pré-imagem de um espaço linear geral de codimensão $i$ 
(veja a Definição~\ref{def-graus-projetivos} para o conceito preciso). A
conexão é que tanto a classe de Segre como estes graus podem
ser calculados via o gráfico do mapa.
Em resumo, podemos calcular a classe $\csm(X)$ a partir
dos graus projetivos do mapa gradiente (Teorema~\ref{Aluffi2.1}).
Somos assim
levados a estudar a geometria destes mapas.

Não é surpresa que mapas gradientes sejam também um tema clássico. 
Um problema interessante é a classifica\c{c}\~ao de hipersuperf\'icies com mapa gradiente de grau topol\'ogico fixado (\cite{Dol00}, \cite{FM12}, \cite{Huh14}), bem como a busca por f\'ormulas para os seus graus projetivos (\cite{DP03}, \cite{FP07}, \cite{Huh12}). Encontra-se aqui o importante
problema da classificação das hipersuperfícies
\emph{homaloidais}, aquelas cujo mapa gradiente é birracional.
Este é o contexto de uma das motivações iniciais deste trabalho,
e que explicamos a seguir.

\bigskip

Seja $C\subset\P^n$ uma curva racional normal. 
Como toda curva irredutível reduzida não-degenerada, $C$ é não-defeituosa \cite[Proposition~10.11, p.\thinspace{}372]{EH16},
o que aqui significa apenas que $\dim \Sec_k C = 2k-1$ sempre que $2k\leq n$. Repare que
só obtemos hipersuperfícies em espaços de dimensão par.
Suponha então $n=2r$ e considere o mapa gradiente
$\phi_r \colon \P^{2r} \dashrightarrow \P^{2r}$
associado à hipersuperfície $\Sec_r C\subset \P^{2r}$, que é dada pelo determinante 
de uma matriz (genérica) de Hankel 
$\Hcal_{r+1,r+1} = (x_{i+j})_{ij}$ (veja \eqref{eq-matrizhankel}).
 Seu ideal jacobiano
foi estudado em detalhe em \cite{Mo14} e \cite{MS}
do ponto de vista da Álgebra Comutativa. 
O mapa $\phi_r$
 é sempre dominante \cite[Proposition~3.1.1]{Mo14} e
 não é birracional para $r=2$ \cite[Proposition~3.2.1]{Mo14} (para $r=1$, é um isomorfismo linear de $\P^2$, nada emocionante). Mais ainda, 
 M. Mostafazadehfard e A. Simis propuseram a seguinte:

\begin{neoconj}{\cite{MS}}
\emph{Para $r\geq2$, o mapa $\phi_r$ não é birracional.}
\end{neoconj}

No Corolário~\ref{thm3} provamos que 
$\deg(\phi_r)=\frac{1}{r+1}\binom{2r}{r}$, os quais formam a famosa sequência dos \emph{números de Catalan}: 1, 1, 2, 5, 14, \dots \ mostrando consequentemente que a Conjectura de Mostafazadehfard e Simis é 
verdadeira. 

No intuito de determinar a classe $\csm(\Sec_rC) \in A_*\P^{2r}$, buscamos calcular os
outros graus projetivos destes mapas, o que conseguimos apenas
em casos particulares. Estes casos são suficientes para conjecturar uma fórmula para o caso geral, como detalharemos adiante.

\bigskip


Este trabalho est\'a dividido em quatro cap\'itulos. O primeiro
é essencialmente independente dos demais.

\medskip 

No Cap\'itulo~\ref{cap1} s\~ao apresentados preliminares sobre as classes caracter\'isticas das quais faremos uso no decorrer do texto, os graus projetivos de um mapa racional, e na sequ\^encia especializamos para o mapa gradiente de uma hipersuperf\'icie projetiva. O principal resultado do cap\'itulo descreve uma rela\c{c}\~ao entre as classes de Schwartz-MacPherson de um subconjunto localmente fechado de um espaço projetivo e de uma se\c{c}\~ao hiperplana gen\'erica:

\begin{neoteo}{\ref{thm1}}
\emph{Sejam $X \subset \mathbb{P}^n$ um conjunto localmente fechado e $H\subset\P^n$ um hiperplano geral. Ent\~ao:}
\[
\csm(X\cap H) = \csm(X) \cdot \dfrac{h}{1+h} \qquad\in A_* \mathbb{P}^{n}.
\]
\end{neoteo}

Obtivemos este resultado, de maneira independente, por volta da mesma época da publicação do preprint de Aluffi \cite{Alu13}; lá os resultados são mais gerais e valem sobre qualquer corpo algebricamente fechado de característica zero. Posteriormente percebemos 
que de fato, sobre os números complexos, 
o teorema já havia sido publicado em 2003 por
T. Ohmoto \cite[Theorem~1.1]{Ohm03}. 

As classes de Milnor (veja Definição~\ref{def-Milnor}) foram introduzidas por P. Aluffi em
\cite{Alu95} e generalizam os números de Milnor 
para singularidades isoladas bem como os números de 
Parusi\'{n}ski \cite{Pa88}. É em termos destas classes que,
com o auxílio do Teorema~\ref{thm1}, 
obtemos uma fórmula para o grau do mapa gradiente de  hipersuperf\'icies projetivas com singularidades arbitr\'arias:

\begin{neoteo}{\ref{thm2}}
\emph{Seja $X \subset \P^n$ uma hipersuperf\'icie de grau $k$. Ent\~ao
}\[
\deg \grad X = (k-1)^n - \mu(X) - \mu(X\cap H), 
\]
\emph{onde $H \subset \P^n$ \'e um hiperplano gen\'erico e $\mu(\cdot)$ denota o grau da classe de Milnor.
}
\end{neoteo}
\noindent{}Esta f\'ormula generaliza a conhecida express\~ao
para singularidades isoladas (Teorema~\ref{milnornum}).

\bigskip

No Cap\'itulo~\ref{cap2} revisamos as variedades secantes em geral, especializando para nosso caso de interesse, as secantes superiores de curvas racionais normais. Em seguida, a partir de uma fórmula devida a Abhyankar,  
deduzimos uma expressão particularmente simples 
para as séries de Hilbert de variedades determinantais gerais dadas por \emph{menores maximais} (Proposi\c{c}\~ao~\ref{prop:serieHilb}), da qual decorre uma expressão similar para a s\'erie de Hilbert para as secantes de uma curva racional normal qualquer (Corol\'ario~\ref{cor-hilb-secrnc}). Em seguida, via uma ação de grupos conveniente, calculamos 
a característica de Euler topológica:
\begin{neoteo}{\ref{teo-Euler_Char_Sec_k}}
Se $C\subset \P^n$ é uma curva racional normal,
então temos 
$\euler(\Sec_k C) = 2k$ sempre que $2k-1\leq n$.	
\end{neoteo}
%
%

Especializamos em seguida para o caso de hipersuperfícies, quando o espaço ambiente tem dimensão par. Tome $C\subset\P^{2r}$
uma curva racional normal.
Na Proposição~\ref{prop-projveronese}
mostramos que a variedade dual $(\Sec_r C)^*\subset (\P^{2r})^*$
é isomorfa, via uma projeção linear, à variedade de Veronese
$\nu_2(\P^r)$.
A partir daí obtemos uma fórmula para a classe de Mather destas hipersuperfícies:
\begin{neoteo}{\ref{teo-Mather_sec}}
	Dado $r\geq 1$, a classe de Mather da hipersuperfície
	$\Sec_rC\subset\P^{2r}$ é 
	\begin{equation*}
	\cma(\Sec_r C)=
	(1+h)^r\sum_{j=0}^{[\frac{r}{2}]} \binom{r+1}{2j+1}h^{2j+1}
	\quad \in A_*\P^{2r}.
	\end{equation*}
\end{neoteo}
Em particular, todos os coeficientes da classe de Mather são positivos. X. Zhang \cite[\S~7.6]{Zhang18} conjecturou que um resultado de positividade similar
deve valer para a classe de Mather de variedades determinantais genéricas.

Como consequência obtemos também o chamado
\emph{grau dist\^ancia Euclidiana gen\'erica} da
$r$-secante (Corolário~\ref{cor-geddeg}):
\[
\geddeg(\Sec_r C) = \dfrac{3^{r+1}-1}{2}.
\]

R. Piene descreve em \cite[Théorèm~3]{Pie88}, uma relação precisa entre os graus polares e as classes de Mather para uma variedade projetiva. A partir desta relação e do Teorema~\ref{teo-Mather_sec}, obtemos então expressões explícitas
\eqref{eq-graus-polares} para os graus polares das secantes $\Sec_rC\subset \P^{2r}$. 


\bigskip

Para o restante da introdução, denote por $C\subset\P^{2r}$
uma curva racional normal e seja
 $\phi_r\colon\P^{2r}\dashrightarrow \P^{2r}$
 mapa gradiente 
 associado à hipersuperfície $\Sec_rC \subset \P^{2r}$.
\medskip

No Cap\'itulo~\ref{cap-grau}, respondemos positivamente \`a Conjectura~\cite{MS} enunciada acima, calculando o grau do
mapa gradiente:
\begin{neocor}{\ref{thm3}}
\(
\displaystyle
	\deg(\phi_r) =
	 \frac{1}{r+1}\binom{2r}{r}.
	\)
\end{neocor}
Nossa abordagem consiste em fatorar o mapa 
$\phi_r$ por um certo mapa birracional $\psi_r$, seguido por uma proje\c{c}\~ao linear da Grassmanianna $\G(r-1, r+1)$:
\begin{equation*}
\xymatrix{ 
& \G(r-1, r+1) \ar@{->}[dr]^{\pi_r}  \\ \P^{2r}\ar@{-->}[ru]^{\psi_r} \ar@{-->}[rr]^{\phi_r}& & \mathbb{P}^{2r} }
\end{equation*}
O mapa $\psi_r$ \'e definido pelos menores maximais 
da matriz de Hankel $\Hcal_{r, r+2} = (x_{i+j})_{ij}$.
Nosso resultado principal aqui é o Teorema~\ref{thm2.9}, que nos diz que os graus projetivos dos mapas $\phi_r$ e $\psi_r$ coincidem. Segue daí a fórmula para $\deg(\phi_r)$. Uma outra consequência (Corol\'ario~\ref{cor-thm2.9}) é
que as classes de Segre de $\Sec_{r-1}C$ e de $\sing\Sec_r C$ coincidem. Isto é surpreendente uma vez que, embora estes
esquemas coincidam como conjuntos
(Proposi\c{c}\~ao~\ref{corol_props_rnc2}), este último em geral não é reduzido.

Finalizamos o capítulo calculando alguns dos outros graus projetivos (Proposi\c{c}\~ao~\ref{prop-graus}), o que servirá de suporte para a parte final da tese.


\bigskip

O assunto do Cap\'itulo~\ref{cap-conjecturas} é uma
conjectura sobre os graus projetivos  $d_0(\phi_r),\dotsc,d_{2r}(\phi_r)$ do mapa gradiente e as classes
$\csm(\Sec_rC)$. É conveniente considerar o aberto complementar $U_r := \P^{2r}\setminus \Sec_r C$.
Para $i=0,1,\dotsc,2r$, denote por 
$c_i(U_r)$ o coeficiente de $[\P^i]$ 
na classe $\csm({U_r})$, ou seja,
\[
\csm(U_r) = c_0(U_r)[\P^0] + c_1(U_r)[\P^1] + \dotsb + c_{2r}(U_r)[\P^{2r}] \qquad \in A_*\P^{2r}.
\] 
Então:
\begin{neoconj}{A}
\label{cA}
\begin{equation*}
\label{eq-cA}
c_i(U_r)=\binom{r}{\left[\frac{i}{2}\right]}
\binom{r}{\left[\frac{i+1}{2}\right]}
\qquad\text{e}\qquad
d_i(\phi_r) = \sum_{k=0}^{i} \frac{1}{k+1}\binom{k}{i-k}\binom{r}{k}\binom{2k}{k}.
\end{equation*}
\end{neoconj}

Este capítulo conta a história de como chegamos a expressões tão precisas.
De fato a Conjectura~\ref{cA} vem como reformulação de uma outra que fizemos inicialmente, e que descrevemos agora. 

Na Proposição~\ref{prop-graus} calculamos 
explicitamente as funções $d_i(\phi_r)$ para
$i\leq 4$.
Em cada um destes casos
vimos que $d_i(\phi_r)$
é polinomial em $r$, de grau $i$.
Entre estas,
a expressão para $d_4(\phi_r)$ é a mais intrigante.
No seu cálculo, necessitamos do gênero
de uma certa curva não-singular, obtido
via a série de Hilbert calculada no Capítulo~\ref{cap2}. Nossa ideia inicial (e infrutífera) era estimar o gênero  
via a desigualdade de Castelnuovo, ou seja,
o Teorema de Riemann-Roch. Dessa 
heurística ficou a ideia fixa \emph{comportamento 
	polinomial}. Isto, juntamente com os resultados empíricos obtidos via \Mdois{} (veja \ref{subsecao-r<=5} no Apêndice), nos conduziu à seguinte:
\begin{neoconj}{A'}
\label{cB}
\emph{Para cada $i\geq 0$ fixado, tem-se:}
\begin{enumerate}[{\rm(a)}]
\item \emph{A fun\c{c}\~ao $c_i(U_r)$ \'e polinomial em $r$, de grau $\leq i$.}
\item\emph{ A fun\c{c}\~ao $d_i(\phi_r)$ \'e polinomial em $r$, de grau $\leq i$.}
\end{enumerate}
\end{neoconj}
Supondo a validade da Conjectura~\ref{cB}, obtemos no Teorema~\ref{algoritmo} um algoritmo que calcula simultaneamente
as funções $c_i$ e $d_i$ para todo $i\geq 0$, caracterizando-as de modo único.
Utilizando os resultados produzidos pelo algoritmo, 
consideramos então as \emph{funções geradoras} associadas,
isto é, séries formais com coeficientes dados pelos valores destas
funções; 
especificamente, olhamos para $p_i(x):=\sum_{r\geq0} d_i(r)x^r$. Estas, por sua vez, se escrevem como funções racionais
$p_i(x)=a_i(x)/b_i(x)$. E aqui ocorre uma conexão inesperada: em todos os exemplos que calculamos, 
os numeradores destas funções racionais, após pequenos ajustes, coincidem com 
os polinômios de \emph{Kazhdan-Lusztig} de 
(matróides associados a) grafos completos tripartidos $K_{1,1,i}$ obtidos por 
 K. Gedeon em
um artigo recente (\cite{Gedeon}). Estes polinômios constituem um invariante importante em Combinatória. Cabe men\-cio\-nar também a relação com os \emph{caminhos de Dyck com semicomprimento e número de subidas longas fixados}; veja a seção~\ref{secao-funcaogeradora} para detalhes.

É a partir desta conexão que obtemos as expressões que aparecem na Conjectura~\ref{cA}. Mostramos no Teorema~\ref{teo-formulas}
que de fato estas são as fórmulas para as sequências produzidas pelo algoritmo do
Teo\-rema~\ref{algoritmo}, demonstrando assim que a Conjectura~\ref{cA} e 
a Conjectura~\ref{cB} são equivalentes.

Apresentamos no Teorema~\ref{teo-equivalencias} uma longa lista
de reformulações equivalentes da nossa  Conjectura \ref{cA}/\ref{cB}, incluindo expressões para as funções geradoras envolvidas.
No Corolário~\ref{cor-conjectura} mostramos que a eventual validade da nossa conjectura tem como consequência a positividade dos coeficientes das classes $\csm(U_r)$ e $\csm(\Sec_rC)$ em $A_*\P^{2r}$. Este
é um tema relevante: X. Zhang \cite[\S~7.6]{Zhang18} conjecturou 
algo similar para variedades determinantais genéricas; J. Huh \cite{Huh16} provou a efetividade de $\csm$ para células de Schubert em Grassmannianas, que havia sido conjecturado por P. Aluffi e L. Mihalcea; e poste\-rior\-mente, em \cite{AMSS17}, os autores estenderam o resultado para células de Schubert em certas variedades bandeira.

Na Proposição~\ref{conj-A'} destacamos ainda uma outra reformulação que acreditamos ser uma estratégia promissora para a demonstração da nossa conjectura. Por fim, concluímos o capítulo apresentando uma lista de evidências que apontam para a validade da conjectura.

\bigskip

O Apêndice~\ref{apA} trata de combinatória, contém demonstrações de identidades binomiais necessárias no Capítulo~\ref{cap-conjecturas}. Finalmente, o Apêndice~\ref{apB} contém diversas porções de código  para o \Mdois{}, com cálculos de exemplos e uma implementação do algoritmo do Teorema~\ref{algoritmo}.

%
%

%
%
%
\chapter*{Nota\c{c}\~ao e conven\c{c}\~oes}
\addcontentsline{toc}{chapter}{Nota\c{c}\~ao e conven\c{c}\~oes}
\markboth{CONVENÇÕES}{}

Trabalhamos sobre o corpo dos n\'umeros complexos.

\bigskip
%
%

Denotamos a caracter\'istica de Euler topol\'ogica de uma variedade $X$ por $\euler(X)$.

\bigskip

Para um mergulho fechado $i\colon X\hookrightarrow M$ e um ciclo $\alpha\in A_*X$, sempre que poss\'ivel omitimos o pushforward indicando o grupo de Chow do contra-dom\'inio, ou  seja, no lugar de $i_*\alpha$ escrevemos $\alpha \in A_*M$.

\bigskip

Na maioria dos resultados e exemplos, as classes que calculamos pertencem ao grupo de Chow de um espa\c{c}o projetivo $\P^n$. Salvo exce\c{c}\~ao, denotamos a classe hiperplana por $h=c_1(\Ocal(1))\cap [\P^n]$. Observe que $A_*\P^n \cong \Z[h]/(h^{n+1})$.

\bigskip

Dizemos que uma variedade $X\subset\P^n$ \'e \emph{n\~ao-degenerada} quando n\~ao est\'a contida em um hiperplano.

\bigskip

Para uma curva racional normal $C \subset \P^n$ fica impl\'{\i}cito que $C$ tem grau $n$ e portanto \'e n\~ao-degenerada.

\bigskip

Dado um mapa racional $\phi\colon\P^n\dashrightarrow\P^N$, denotamos por $\deg(\phi)$ o grau topol\'ogico do mapa, isto \'e, a cardinalidade da pr\'e-imagem de um ponto geral de $\phi(\P^n$). Quando o mapa  n\~ao \'e genericamente finito, definimos $\deg(\phi)=0$.

\bigskip

Dado $d\in \Z$, definimos os \emph{polin\^omios binomiais} em $\Q[x]$ da seguinte maneira:
\[
\binom{x}{d} :=
\begin{cases}
\frac{x(x-1)(x-2)\dotsb(x-d+1)}{d!}, & \text{se $d>0$}\\
1, & \text{se $d=0$}\\
0, & \text{se $d<0$}.
\end{cases}
\]
Este \'e um polin\^omio de grau $d$ sempre que $d\geq 0$.
Finalmente, se $a$ \'e um inteiro, ent\~ao $\binom{a}{d}$ denota o 
polin\^omio binomial avaliado em $a$ e 
em particular se $0\leq a < d$, ent\~ao $\binom{a}{d}=0$.

%
%
%
%
\chapter{Classes características}\label{cap1}

Neste capítulo apresentamos os conceitos básicos sobre classes características que serão necessários para os capítulos seguintes. Iniciamos o capítulo apresentando as diversas classes de que faremos uso: Schwartz-MacPherson, Mather, Segre, Fulton e de Milnor. No caso particular de hipersuperfícies projetivas, mostramos que existem relações entre estas classes e o mapa gradiente dado pela equação que define a hipersuperfície.

As referências básicas para Teoria de Interseção são \cite{Ful84} e \cite{EH16}.

\section{Preliminares} 

Em \cite{Mac74}, R. MacPherson definiu classes de Chern para variedades singulares, as quais concordam com a noção de classe de Chern para variedades não-singulares. Entre outras, encontram-se as classes de Schwartz-MacPherson, de Fulton e de Mather. Destacamos as duas primeiras uma vez que a diferença entre elas, a menos de sinal, define a classe de Milnor, que guarda informação sobre o lugar singular da variedade.

Começamos descrevendo a classe de Schwartz-MacPherson ($\csm$). Como referências indicamos 
\cite[Example~19.1.7, p.\thinspace{}376]{Ful84},
\cite{Alu99a}, \cite{Alu03}, \cite{Zhang18} e particularmente
\cite{Alu02}.

\medskip

Seja $X$ uma variedade algébrica. 


Uma \textbf{função construtível em $X$} é uma soma finita $\sum m_{V}1_{V}$ sobre todas as subvariedades fechadas irredutíveis $V$ de $X$, os coeficientes $m_V$ são números inteiros e
\[
1_{V}(x) := \begin{cases} 1, &\text{se $x \in V$ }\\ 0, &\text{se $x \in X \setminus V$} \end{cases}.
\]

Denotemos por $\Ccal(X)$ o conjunto das funções construtíveis em $X$. Este é um grupo abeliano com a soma de funções, e é livremente gerado pelas funções
\[
\{1_{V};\  V\subset X \text{ é uma subvariedade fechada irredutível}\}.
\]
Dado um morfismo 
$f \colon X \longrightarrow Y$, definimos o 
\emph{pushforward}
\[
f_* \colon \Ccal(X) \longrightarrow \Ccal(Y)
\]
como segue. Para cada subvariedade fechada irredutível $V\subset X$ e $p \in Y$
\[
f_*(1_{V})(p) := \euler (f^{-1}(p) \cap V)
\]
e estendemos linearmente: $f_*(\sum m_V 1_V) = \sum m_V f_*(1_V)$.
Temos assim um funtor (\cite[Proposition~1]{Mac74} ou
\cite[Lemma~2.3]{Alu13})
da categoria das variedades algébricas com mapas regulares para a categoria dos grupos abelianos.

Por outro lado, temos o \textbf{funtor de Chow} $A_{*}$ (veja por exemplo \cite[Chapter 1]{Ful84}), que associa cada variedade $X$ ao grupo abeliano (chamado o grupo dos \emph{ciclos} de $X$)
\[
A_{*}X = \bigoplus_{k} A_{k}X,
\]
onde, para cada $k$, $A_{k}X$ é o grupo abeliano livre gerado 
pelas subvariedades fechadas irredutíveis $k$-dimensionais em $X$, módulo equivalência racional: se $\alpha \in A_{k}X$, então $\alpha = \sum_{i}m_{i}[V_{i}]$ onde ${m_i} \in \Z$ e $V_{i} \subset X$ tem dimensão $k$. Dizemos que $\alpha$ é um \emph{$k$-ciclo} de $X$.
Aqui, para cada um morfismo próprio $f \colon X \longrightarrow Y$, temos o morfismo \emph{pushforward}
\[
f_* \colon A_*X \longrightarrow A_*Y
\]
induzido pela definição $f_*[V] := \deg(V/f(V))[f(V)]$, onde
$V \subset X$ é uma subvariedade irredutível e 
\[
\deg(V/f(V)) = \begin{cases} [R(V): R(f(V))], &\text{se } \dim f(V) = \dim V \\ 0, &\text{se } \dim f(V) < \dim V \end{cases},
\]
sendo $[R(V): R(f(V))]$ o grau da extensão de corpos de funções racionais (veja \cite[Theorem~1.4, p.\thinspace{}11]{Ful84}). 
Observe também que \cite[p.\thinspace13]{Ful84}
%
\[
\displaystyle \int \alpha = \int f_* (\alpha)
\]
para todo ciclo $\alpha \in A_*(X)$.

\medskip

Em \cite[Theorem~1]{Mac74}, R. MacPherson provou, o que havia sido conjecturado por Deligne e Grothendieck, que para 
variedades complexas compactas existe uma transformação natural
\[
\Ccal \rightsquigarrow A_{*}
\]
resultado que foi estendido em \cite{Kennedy} para variedades completas sobre qualquer corpo algebricamente fechado de característica zero.
Isto significa que para cada variedade completa $X$ 
existe um homomorfismo de grupos abelianos $c_*(X)\colon \Ccal(X)\longrightarrow A_*(X)$ (que denotamos simplesmente por $c_*$) tal que, para cada morfismo próprio
$f \colon X \longrightarrow Y$ entre variedades completas tem-se
\[
c_*(f_*(\varphi)) = f_*(c_*(\varphi))
\]
para cada função construtível $\varphi\in \Ccal(X)$. Em outras palavras, vale a comutativade do diagrama
\[
\xymatrix {\Ccal(X) \ar@{->}[d]_{c_*}
\ar@{->}[r]^{f_*} & \Ccal(Y)
 \ar@{->}[d]^{c_*}\\
A_*(X) \ar@{->}[r]_{f_*} &
A_*(Y)}
\]

\begin{defi}
Dada uma variedade completa $X$ possivelmente singular, definimos a \textbf{classe de Schwartz-MacPherson de $X$} como
\[
\csm(X) = c_{*}(1_{X}) \quad \in A_{*}(X).
\]
Mais ainda, se $V$ é uma subvariedade de $X$ (fechada ou não),
por abuso de linguagem definimos $\csm(V)=c_*(1_V) \in A_*(X)$.
\end{defi}

\begin{rmk}\cite[Proposition~3.1]{Alu03} 
\begin{enumerate}
\item
O grau da classe de Schwartz-MacPherson é a característica de Euler da variedade:
\[
\displaystyle \int \csm(X) = \euler(X).
\]

De fato, consideremos o mapa constante $\kappa \colon X \longrightarrow \{\mathrm{ponto}\}$. Então:
\[
\kappa_*(\csm(X)) = \euler(X)\cdot 1_{\{\mathrm{ponto}\}},
\]
pois, pela relação de funtorialidade,
\[
\begin{split}
\kappa_*(\csm(X)) &= \kappa_*(c_*(1_{X})) = c_*(\kappa_*(1_{X})) \\
&= c_*(\euler(X)\cdot 1_{\{\mathrm{ponto}\}}) \\
&= \euler(X)\cdot c_*(1_{\{\mathrm{ponto}\}}) \\
&= \euler(X)\cdot [\mathrm{ponto}].
\end{split}
\]
Logo, tomando graus, obtém-se o afirmado.

\item
A classe de Schwartz-MacPherson satisfaz o princípio de inclusão-exclusão:
se $V_1, V_2$ são subvariedades de uma variedade completa $X$, então
\[
\csm(V_{1}\cap V_{2}) = \csm(V_1) + \csm(V_2) - \csm(V_{1}\cup V_{2}) \quad \in A_*X.
\]
\item
Quando $X$ é não-singular, $\csm(X) = c(TX)\cap [X]$.
\end{enumerate}
\hfill{$\Box$}
\end{rmk}

\begin{exem} \cite[1.9.2]{Alu10}
Sejam $X$ uma variedade não-singular, $V$ uma subvariedade não-singular de $X$ de codimensão $d$ e $\pi \colon \tilde{X} \longrightarrow X$ o mapa do blow-up de $X$ ao longo de $V$. Da identificação do divisor excepcional com o projetivizado
do fibrado normal de $V$ em $X$ vem que a característica de Euler das fibras é
\[
\euler (\pi^{-1}(q)) := \begin{cases} 1, &\text{se $q \in X\setminus V$ }\\ d, &\text{se $q \in V$} \end{cases},
\]
(note que $\euler(\P^{d-1})=d$, veja o
Exemplo~\ref{Euler_Char_P^n}) e portanto
\[
\pi_* (1_{\tilde{X}}) = 1_X + (d-1)1_V.
\]
Da relação de funtorialidade de MacPherson obtemos
\[
\begin{split}
\pi_* (\csm(\tilde{X}))  &= \pi_*(c_*(1_{\tilde{X}})) \\
 &= c_*(\pi_*(1_{\tilde{X}})) \\
&= c_*(1_X+(d-1)\cdot 1_V)\\
&= c_*(1_X)+(d-1)\cdot c_*(1_V)\\
&= \csm(X) +(d-1)\cdot \csm(V).
\end{split}
\]
Em particular, tomando graus:
\[
\euler(\tilde{X}) = \euler(X)+(d-1)\cdot \euler(V).
\]
\hfill{$\Box$}
\end{exem}

\begin{defi} \cite[\S~2]{Mac74}
Seja $X$ uma subvariedade de uma variedade não-singular $M$, e seja $G(\dim X, TM)$ o fibrado Grassmanniana \cite[\S~9.6, p.\thinspace{}346]{EH16}. 
A correspondência $x \mapsto T_x X$ nos dá um mapa racional
\[
X \dashrightarrow G(\dim X, TM).
\]
O fecho $\tilde{X}$ da imagem é chamado o \textbf{blow-up de Nash de $X$} e vem equipado com um
mapa natural $\upsilon \colon \tilde{X}\to X$.
Definimos a \textbf{classe de Chern-Mather de $X$} como 
\[
\cma(X) = \upsilon_*(c(\Tcal)\cap[\tilde{X}]) \quad \in A_*X
\]
onde 
$\Tcal$ é o fibrado tautológico da Grassmanniana.
\end{defi}

\begin{rmk}\cite[p.\thinspace{}3993-3994]{Alu99a}
\begin{enumerate}
\item
A definição acima independe da escolha de $M$.
\item
Quando $X$ é não-singular, tem-se $\cma(X) = c(TX)\cap [X]$.
\item
Em geral, as classes $\cma(X)$ e $\csm(X)$ são distintas.
\end{enumerate}
Para um exemplo onde estas classes diferem, 
tome $X\subset\P^4$ a secante da quártica racional
normal. Veremos no Exemplo~\ref{classgrafsec} que
\[
\csm(X) = 3h+6h^2+8h^3+4h^4 \quad \in A_*\P^4,
\]
enquanto que pelo Exemplo~\ref{cma_secmax_rnc}
\[
\cma(X) = 3h+6h^2+4h^3+2h^4 \quad \in A_*\P^4.
\]
\hfill{$\Box$}
\end{rmk}

\begin{defi}
Seja $X$ um subesquema fechado próprio de uma variedade não-singular $M$, e seja $\tilde{M}$ o blow-up de $M$ ao longo de $X$, com divisor excepcional $\tilde{X}$ e projeção $\eta \colon \tilde{X} \longrightarrow X$. A \textbf{classe de Segre de $X$ em $M$} é definida por
\[
s(X, M) = \sum_{k\geq 1} (-1)^{k-1} \eta_*(\tilde{X}^k) \quad \in A_*X,
\]
onde $\tilde{X}^k$ é a auto-interseção de $\tilde{X}$.
\end{defi}

\begin{rmk}\cite[p.\thinspace{}454]{EH16}
Se $X$ é não-singular, tem-se $s(X, M) = s(N_{X/M})\cap[X]$, onde $s(N_{X/M})$ é a classe de Segre (isto é, o inverso da classe de Chern) do fibrado normal de $X$ em $M$.
\hfill{$\Box$}
\end{rmk}

\begin{exem}\label{segre_rnc_4} (Classe de Segre da curva racional normal em $\P^4$): Seja $C \subset \P^4$ a quártica racional normal. Então
\[
s(C, \P^4) = 4h^3 - 18h^4 \quad \in A_*\P^4.
\]

De fato, como $C$ é não-singular, temos que $s(C, \P^4) = s(N_{C/\P^4})\cap [C]$.
Por outro lado, da sequência exata:
\[
0\longrightarrow TC \longrightarrow T\P^4|_{C} \longrightarrow N_{C/\P^4}\longrightarrow 0
\]
segue, via a fórmula de Whitney, que 
\[
c(TC)\cdot c(N_{C/\P^4}) = c(T\P^4|_{C}),
\]
e como $c(T\P^4|_{C})=(1+h)^5$ e $c(TC)\cap[C] = \deg(C)h^3+\euler(C)h^4 = 4h^3+2h^4 \in A_*\P^4$, tem-se
\[
\begin{split}
s(N_{C/\P^4})\cap[C] &= c(T\P^4)^{-1}\cdot c(TC)\cap[C] \\
&= ((1+h)^5)^{-1}\cdot (4h^3+2h^4) \\
&= (1-5h+o(h^2))\cdot (4h^3+2h^4) \\
&= 4h^3-18h^4 \quad \in A_*\P^4.
\end{split}
\]
\hfill{$\Box$}
\end{exem}


\begin{defi}
Seja $X$ um esquema que pode ser mergulhado como um subesquema fechado em uma variedade não-singular $M$. Definimos a \textbf{classe de Chern-Fulton de $X$} como 
\[
\cf(X) = c(TM|_{X}) \cap s(X, M) \quad \in A_{*}(X).
\]
\end{defi}

\begin{rmk}\cite[Example~4.2.6(a), p.\thinspace{}77]{Ful84}
\begin{enumerate}[{\rm(a)}]
\item
A definição acima independe da escolha do mergulho de $X$.
\item
Se $X$ é localmente de interseção completa, então
\[
\cf(X) 
= c(TX)\cap [X],
\]
onde $TX = TM|_{X} - N_{X/M}$ é o fibrado tangente virtual de $X$. 
\item
Se $X$ é não-singular, tem-se $\cf(X) = c(TX)\cap [X]$.
\end{enumerate} 
\hfill{$\Box$}
\end{rmk}

\begin{exem}\label{csm_exem} 
Consideremos as curvas planas $C = Z(x_0^3+x_1^3+x_2^3)$ (cúbica de Fermat) e $D = Z(x_0x_1(x_0+x_1))$ (três retas por um ponto). Então $\euler(C)=0$ e $\euler(D)=4$, portanto
\[
\csm(C) = 3h \quad e \quad \csm(D) = 3h+4h^2 \quad \in A_*\P^2.
\]

Por outro lado, para as classes de Fulton:
\[
\cf(C) = 3h = \cf(D) \quad \in A_*\P^2.
\]
\hfill{$\Box$}
\end{exem}

\begin{rmk}\label{cfInv}
Vimos no exemplo acima duas curvas planas de mesmo grau para as quais as classes de Schwartz-MacPherson são distintas, mas as classes de Fulton coincidem. Esta ``insensibilidade'' da classe
de Fulton é um fato geral. Precisamente:
\emph{duas hipersuperfícies de mesmo grau em $\P^n$ têm a mesma classe de Fulton.}

 Com efeito, como uma hipersuperfície $X\hookrightarrow \P^n$ de grau $d$
é localmente de interseção completa,  
sua classe de Segre é dada 
pelo inverso da classe de Chern \cite[p.\thinspace{}454]{EH16} e portanto igual a
$\frac{[X]}{1+[X]}=\frac{dh}{1+dh} = dh-d^2h^2+d^3h^3-\cdots \in A_*\P^n$. A afirmação
segue daí pela fórmula de projeção. 
\hfill{$\Box$}
\end{rmk}

\begin{exem}\label{classtangent}
Seja $Y \subset \P^n$ a hipersuperfície não-singular de grau $d$. Vamos calcular a classe do fibrado tangente de $Y$.

De fato, a partir da sequência exata
\[
0\longrightarrow TY \longrightarrow T\P^n|_{Y} \longrightarrow N_{Y/\P^n}\longrightarrow 0
\]
e da fórmula de adjunção (\cite[Examples~3.2.11 e 3.2.12, p.\thinspace{}59]{Ful84}), tem-se
\[
\begin{split}
c(TY)\cap[Y] &= c(T\P^n|_Y)c(N_{Y/\P^n})^{-1}\cap[Y] \\
&= (1+h)^{n+1}(1+dh)^{-1}\cap[dh] \\
&= (1+h)^{n+1}(dh-(dh)^2+(dh)^3-\cdots) \quad \in A_*\P^n.
\end{split}
\]
\hfill{$\Box$}
\end{exem}



\begin{defi}
\label{def-Milnor}
Definimos a \textbf{classe de Milnor} de $X$ como 
\[
\milnor(X) = (-1)^{\dim(X)}\big(\cf(X) - \csm(X)\big) \quad \in A_* X.
\]
\end{defi}

O grau da classe de Milnor, $\mu(X) := \displaystyle \int \milnor(X)$, é chamado \textbf{número de Parusi\'nski} de $X$. No caso em que $X$ tem apenas singularidades isoladas, ele coincide com a soma dos números de Milnor.

Para uma variedade não-singular, a classe de Milnor é nula.

\begin{exem}\label{classdiscrim}
Seja $X \subset \P^5$ a hipersuperfície discriminante que parametriza as cônicas planas singulares. Então $X$ é uma cúbica, dada pelo lugar dos zeros do determinante da matriz simétrica $3\times 3$ (genérica).

Para o cálculo de $\cf(X)$, segue da Observação~\ref{cfInv} que é suficiente calcular a classe de Fulton de uma hipersuperfície cúbica não-singular $Y\subset \P^5$ e, neste caso, $\cf(Y) = c(TY)\cap [Y]$, e pelo Exemplo~\ref{classtangent}:
\[
\cf(X) = 3h + 9h^2 + 18h^3 + 6h^4 + 27h^5 \quad\in A_*\P^5.
\]

Por outro lado, de \cite[Example~4.5]{Alu03}, a classe de Schwartz-MacPherson de $X$ é:
\[
\csm(X) = 3h + 9h^2 + 14h^3 + 12h^4 + 6h^5
 \quad\in A_*\P^5
\]
e logo,  por definição,
\[
\milnor(X) = 4h^3 - 6h^4 + 21h^5 
\quad\in A_*\P^5.
\]
\hfill{$\Box$}
\end{exem}

\section{Graus projetivos de mapas racionais}

\begin{defi}
\label{def-graus-projetivos}
Sejam $M\subset \P^n$ uma variedade projetiva de dimensão $s$ e $\phi \colon M \dashrightarrow \P^N$ uma aplicação racional. Os \textbf{graus projetivos} ou o \textbf{multigrau}%
\footnote{Seguimos a nomenclatura em \cite[Example~19.4]{Ha92}. Aproveitamos para avisar que nossa escolha de índices difere daquela ali estabelecida: o que aqui indicamos por `$d_s$' lá se indica `$d_0$', etc.}
de $\phi$ são os coeficientes  $d_0, d_1, \dotsc, d_s$, da classe do fecho do gráfico $\Gamma$ do mapa $\phi$ em $\P^n\times \P^N$:
\[
[\Gamma] =\displaystyle \sum_{i=0}^{s} d_i\cdot [\P^{s-i}\times \P^i],
\]
portanto
\[
d_i = \deg (\phi^{-1}(\P^{N-i})\cap \P^{n-(s-i)}),
\]
onde $\P^{n-(s-i)}$ e $\P^{N-i}$ são subespaços lineares de $\P^n$ e $\P^N$, respectivamente.
\end{defi}

Note que:
\begin{enumerate}
\item
$d_s = 0$ se, e somente se, a fibra geral de $\phi$ tem dimensão positiva.
\item
Se a fibra geral consiste de $k$ pontos reduzidos, então $d_s = k\cdot \deg \overline{\Image (\phi)}$.
\item
$d_0 = \deg M$, e $d_i = 0$ se, e somente se, $i> \dim \Image(\phi)$.
\item 
Observe que o $i$-ésimo grau projetivo de $\phi$ coincide com o $i$-ésimo grau projetivo da restrição de $\phi$ à interseção de $M$ com um $\P^i\subset\P^n$ geral, ou seja: $d_i(\phi)=d_i(\phi|_{M\cap \P^i})$.
\end{enumerate}

\bigskip

A partir de agora nos restringiremos ao caso em que $M = \P^n$.

\smallskip

Seja $Y$ um subesquema fechado de $\P^n$ dado por um ideal homogêneo $I = (f_{0}, \dotsc, f_{N})$, e suponhamos que os geradores de $I$ têm o mesmo grau. 
Consideremos o mapa racional associado ao ideal $I$
\[
\begin{array}{cccl}
\varphi \ : & \! \P^n & \! \dashrightarrow & \! \P^N \\
& \! p & \! \longmapsto & \! \big( f_{0}(p) : \dots : f_{N}(p) \big).
\end{array}
\]
cujo lugar de base é o esquema $Y$. 
O resultado básico é que (o pushforward de) a classe de Segre do lugar de base pode ser calculada em termos dos graus projetivos $d_0,\dotsc,d_n$ deste mapa e vice-versa. Denotamos por $h$ a classe hiperplana em $A_*\P^n$ e escrevemos $s(Y, \P^n) = \sum_{j=0}^{n}s_{j}h^j \in A_*\P^n$.

\begin{prop} [{\cite[Proposition~3.1]{Alu03}}]
\label{Aluffisegre} 
Mantenha a notação acima e seja $r$ o grau comum dos polinômios que definem o ideal $I$. Então:
\[
s(Y, \P^n) = 1 - c(\Ocal(r h))^{-1} \cap \left(\sum_{i=0}^{n} \dfrac{d_i h^i}{c(\Ocal(r h))^i} \right) 
\qquad \in  A_*\P^n.
\]
Expandindo, obtemos $s_0=0$ e, para $l\in\{1,\dotsc,n\}$,
\begin{equation}
\label{eq-segre-djs}
s_l = - \sum_{i+j=l}(-1)^{j}\binom{l}{i}d_i r^j
\end{equation}
e invertendo estas relações,
\begin{equation}
\label{eq-djs-segre}
d_k = r^k - \sum_{j=1}^{n}\binom{k}{j}s_{j}r^{k-j}.
\end{equation}
para $k=0,\dotsc,n$.
\end{prop}

\section{Mapas gradientes e classes CSM de hipersuperfícies}

O objetivo desta seção é estudar o mapa gradiente associado a uma hipersuperfície projetiva, apresentando exemplos, propriedades, além da importante relação entre os graus projetivos do mapa e a classe de Schwartz-MacPherson da hipersuperfície.

Seja $X$ uma hipersuperfície de grau $k$ em $\P^n$, definida pelos zeros de um polinômio homogêneo $f\in \C[x_0, \dotsc , x_n]$.

\begin{defi}
O \textbf{mapa gradiente (ou polar)} de $X$ (ou de $f$) é a aplicação racional
\[
\begin{array}{cccl}
\grad X \ : & \! \P^n & \! \dashrightarrow & \! \P^{n} \\
& \! p & \! \longmapsto & \! \left(f_{x_{0}}(p):\dots : f_{x_{n}}(p)\right),
\end{array}
\]
onde $f_{x_i} = \dfrac{\partial f}{\partial x_i}$, para $i=0, \dotsc, n$. 
\end{defi}
O mapa gradiente é portanto uma extensão do mapa de Gauss da hipersuperfície para o espaço ambiente. O lugar de base deste mapa é o esquema dos pontos singulares de $X$.

O grau topológico do mapa gradiente de $X$ (ou de $f$) é a cardinalidade de uma fibra genérica (e $0$ se o mapa não for dominante), e este grau coincide com o $n$-ésimo grau projetivo do mapa $\grad X$.
Dizemos que $X$ é \textbf{homaloidal} se o seu mapa gradiente é birracional.

Observamos que, para o cálculo do grau topológico do mapa gradiente de $f$, podemos supor $f$ é reduzido
(\cite[Corollary~2]{DP03}). 

\begin{rmk} Uma justificativa para o estudo de mapas gradiente está no fato que certas propriedades do mapa refletem informações sobre a geometria da hipersuperfície.
\begin{enumerate}
\item
A hipersuperfície $X = Z(f)$ é não-singular se, e somente se, o mapa gradiente associado $\grad X$ é um morfismo.
\item Note que a hessiana de $f$ coincide com o jacobiano do mapa gradiente e logo 
$X$ tem Hessiano não-nulo se, e somente se, $\grad X$ é um mapa dominante. 
\end{enumerate}
\hfill{$\Box$}
\end{rmk}

Vejamos alguns exemplos simples de hipersuperfícies {homaloidais}.

\begin{exem}
\begin{enumerate}
\item
Para $f = x_{0}^{2}+\dotsb +x_{n}^{2}$, $X = Z(f)$ é uma quádrica não-singular em $\P^n$ cujo mapa gradiente é a aplicação identidade.
\item
Se $f = x_0\dotsb x_n$, então $X=Z(f)$ é a união dos $n+1$ hiperplanos coordenados em $\P^n$, e o mapa gradiente de $X$ é a transformação de Cremona padrão
\[
(x_0: \cdots: x_n) \longmapsto (x_1\dotsb x_n: \dots : x_0\dotsb x_{n-1}),
\]
a qual é uma aplicação birracional com $(\grad X)^{-1} = \grad X$. 
\item
Fixado $n \geq 1$, a hipersuperfície $X\subset \P^{n^{2}-1}$ dada pelo determinante de matrizes $n \times n$ é homaloidal. 

De fato, a menos de uma mudança de coordenadas, o mapa polar do determinante é dado por $A \mapsto A^{-1}$ (regra de Cramer), sendo portanto birracional.
\hfill{$\Box$}
\end{enumerate}
\end{exem}

O próximo teorema relaciona (o {pushforward} de) 
a classe de Schwartz-MacPherson com os graus projetivos do mapa gradiente de uma dada
hipersuperfície, e é fundamental para as aplicações que faremos.

\begin{thm} [{\cite[Theorem~2.1]{Alu03}}]
\label{Aluffi2.1} 
Dada $X$ uma hipersuperfície em $\P^n$, tem-se 
\[
\csm(X) = (1+h)^{n+1} - \sum_{j=0}^{n}d_j(-h)^j(1+h)^{n-j} \qquad \in A_*\P^n
\]
onde os $d_j$'s são os graus projetivos do mapa gradiente de $X$.
\end{thm}

\begin{exem}\label{classgrafsec}
(Mapa gradiente da secante da quártica racional normal). Sejam $C \subset \P^4$ a curva racional normal e $X$ a variedade secante de $C$ (fecho do lugar das retas por dois pontos gerais de $C$). Veremos no próximo capítulo (Proposição~\ref{secant_x_hankel}) que $X$ é uma hipersuperfície cúbica de $\P^4$, dada pelos zeros do polinômio
\[
F = \det
\begin{pmatrix}
x_0 & x_1 & x_2 \\
x_1 & x_2 & x_3 \\
x_2 & x_3 & x_4
\end{pmatrix},
\]
e que $\sing X = C$ (veja também a Seção \ref{secao-codigos}), e portanto $C$ é o lugar de base do mapa gradiente de $X$. 

Logo, uma vez conhecida a classe de Segre de $C$ (Exemplo~\ref{segre_rnc_4}), obtemos da Proposição~\ref{Aluffisegre} os graus projetivos do mapa $\grad X$:
\[
(d_0, d_1, d_2, d_3, d_4) = (1, 2, 4, 4, 2)
\]
e, consequentemente, do Teorema~\ref{Aluffi2.1} obtemos a classe $\csm$ da variedade secante de $C$:
\[
\csm(X) = 3h+6h^2+8h^3+4h^4 \qquad \in A_*\P^4.
\]
\hfill{$\Box$}
\end{exem}


\section{Resultados sobre seções hiperplanas}

O objetivo desta seção é a obtenção de uma fórmula para o cálculo do grau do mapa gradiente para hipersuperfícies com singularidades arbitrárias, a qual é dada em termos do seu grau e das classes de Milnor da hipersuperfície e de uma seção hiperplana geral. Este resultado, por sua vez decorrerá da expressão da classe de Schwartz-MacPherson de uma seção hiperplana genérica.

\subsection{CSM de uma seção hiperplana}

Como consequência do Teorema~\ref{Aluffi2.1} e do lema seguinte, conseguimos expressar a classe de Schwartz-MacPherson de uma seção hiperplana genérica de uma hipersuperfície projetiva em termos da classe de Schwartz-MacPherson da hipersuperfície e da classe de Segre do hiperplano.

\begin{lem}\label{lemdegreesection}
Sejam $X\subset \P^n$ uma hipersuperfície e $X' = X\cap H$ uma seção hiperplana genérica. Sejam $d_0,\dotsc,d_n$ e $d'_0, \dots, d'_{n-1}$ os graus projetivos dos mapas
$\grad X$ e $\grad X'$, respectivamente. Então $d_i = d'_i$, para todo $i=0, \dotsc, n-1$. 
\end{lem}

\begin{proof}
Veja \cite[Theorem~2.1]{DS15}.
\end{proof}

O resultado a seguir, crucial para a obtenção do Teorema~\ref{thm2}, foi provado independentemente
em  \cite[Proposition~2.6]{Alu13}. Posteriormente percebemos que de fato já havia sido publicado antes em \cite[Theorem~1.1]{Ohm03}. A ideia central na prova de Aluffi, que se encontra escondida no lema acima, é o cálculo da classe de Segre de uma seção linear geral do lugar singular (veja \cite[(2)]{Alu13}).

\begin{thm} \label{thm1}
Sejam $X \subset  \mathbb{P}^n$ um conjunto localmente fechado e $H\subset\P^n$ um hiperplano geral. Denote por $h$
a classe hiperplana em $\P^n$. Então:
\[
\csm(X\cap H) = \csm(X) \cdot \dfrac{h}{1+h} \quad\in A_* \mathbb{P}^{n}.
\]
\end{thm}

\begin{proof} 

Iniciamos reduzindo ao caso em que $X$ é uma hipersuperfície. Como $X$ pode ser escrito como um conjunto diferença de conjuntos fechados, pela aditividade das classes $\csm$, podemos supor que o próprio $X$ é um conjunto fechado. Agora, como todo conjunto fechado pode ser escrito como uma interseção de hipersuperfícies podemos, pelo princípio de inclusão-exclusão, supor que $X$ é uma hipersuperfície.

Portanto, pelo Teorema~\ref{Aluffi2.1},
\[
\csm(X) = (1+h)^{n+1} - \displaystyle \sum_{j=0}^{n} d_j(-h)^j(1+h)^{n-j}
\quad \in A_*\P^n
\]
e, tomando o pushforward via a inclusão $H=\P^{n-1}\hookrightarrow \P^n$ e com a mesma notação
do Lema~\ref{lemdegreesection},
\[
\csm(X\cap H) = h \cdot\left( (1+h)^{n} - \displaystyle \sum_{j=0}^{n-1} d'_j(-h)^j(1+h)^{n-1-j} \right) 
\quad \in A_*\P^n.
\]
Então, usando o fato de que 
$h^{n+1} = 0$ e o Lema~\ref{lemdegreesection}, obtemos
\[
\begin{split}
\csm(X)\cdot \dfrac{h}{1+h} 
&= \Big((1+h)^{n+1} - \displaystyle \sum_{j=0}^{n} d_j(-h)^j(1+h)^{n-j}\Big)\cdot \dfrac{h}{1+h} \\
&= h\cdot \left( (1+h)^{n} - \displaystyle \sum_{j=0}^{n-1} d_j(-h)^j(1+h)^{n-1-j} - (-h)^n(1+h)^{-1}d_n \right) \\
&=h\cdot\left( (1+h)^{n} - \displaystyle \sum_{j=0}^{n-1} d'_j(-h)^j(1+h)^{n-1-j} \right) \\
&= \csm(X\cap H)
\quad \in A_*\P^n.
\end{split}
\]
\end{proof}

Da aplicação sucessiva do Teorema~\ref{thm1} decorre:

\begin{cor}\label{csm_hip_sections}
Para um conjunto localmente fechado $X\subset \P^n$ e hiperplanos gerais $H_1, \dots, H_r \subset \P^n$, tem-se:
\[
\csm(X\cap H_1 \cap \dots \cap H_r) = \csm(X) \cdot \dfrac{h^r}{(1+h)^r} \quad \in A_* \mathbb{P}^{n},
\]
onde $h$ é classe hiperplana em $\P^n$.
\end{cor}

\begin{rmk} \label{CSMsecao}
Seja $X\subset \P^n$ um conjunto localmente fechado. 
O Teorema~\ref{thm1} nos diz que uma seção hiperplana carrega muita informação da classe original.
Com efeito, escrevendo
\(
\csm(X) = a_1 h + a_2 h^2 + \dotsb + a_n h^n \in A_*\P^n
\)
e
\[
\csm(X\cap H) = b_1 h + \dotsb +  b_{n-1} h^{n-1} \quad \in A_*\P^{n-1},
\]
temos que cada coeficiente $b_k$ é uma soma alternada dos $a_i$'s; precisamente:
\begin{equation}\label{eq:CSMsecao}
b_k = \sum_{i=1}^k (-1)^{k-i} a_i.
\end{equation}

De fato, as características de Euler de
sucessivas seções hiperplanas gerais 
e os coe\-fi\-cientes da
classe $\csm$ carregam exatamente a mesma informação.
Isto é uma conta simples e já foi feita 
em \cite{Alu13}. 
Para hiperplanos gerais $H_1, \dots, H_r \subset \P^n$ ($1\leq r < n$), sejam
\[
\chi_X(t) := \displaystyle \sum_{r\geq 0} (-1)^r \euler(X\cap H_1 \cap \dots \cap H_r)t^r
\qquad \text{e} \qquad
\gamma_X(t) := \displaystyle \sum_{r= 0}^{n-1} a_{r+1} t^{r}.
\]
Do Corolário~\ref{csm_hip_sections} segue que
\[
\euler(X\cap H_1 \cap \dots \cap H_r) = \displaystyle \int \csm(X\cap H_1 \cap \dots \cap H_r) = \displaystyle \int \csm(X) \cdot \dfrac{h^r}{(1+h)^r}.
\]

Então
\cite[Theorem~1.1]{Alu13}
nos dá as seguintes relações:
\[
\chi_X(t) = \mathcal{I}(\gamma_X) \qquad \text{e} \qquad \gamma_X(t) = \mathcal{I}(\chi_X),
\]
onde $\mathcal{I}$ é a involução $p(t) \mapsto \dfrac{tp(-t-1)+p(0)}{t+1}$.

\hfill{$\Box$}
\end{rmk}

Naturalmente a hipótese ``hiperplano geral'' no Teorema~\ref{thm1} é essencial para a validade do resultado.

\begin{exem}

Seja $X=Z(x_0^3+x_1^3+x_2^3+x_3^3)\subset \P^3$ a cúbica de Fermat e $H = Z(x_2+x_3)$ o plano tangente a $X$ no ponto $p=(0: 0: 1: -1)$. Então $X\cap H$ é uma união de três retas passando por $p$ e, pelo que vimos no Exemplo~\ref{csm_exem}:
\[
\csm(X \cap H) = 3h^2+4h^3 
\quad\in A_*\P^3
\]

Por outro lado, da sequência de Euler e da fórmula de adjunção, obtemos
\[
c(TX) = c(T\P^3|_X)c(N_{X/\P^3})^{-1} = (1+h)^4\cdot(1+3h)^{-1},
\]
portanto
\[
\csm(X) = c(TX)\cap [X] = (1+h+3h^2-5h^3)\cdot (3h) = 3h+3h^2+9h^3 
\quad\in A_*\P^3
\]
e consequentemente
\[
\csm(X)\cdot \frac{h}{1+h} = (3h + 3h^2 + 9h^3)\cdot(h - h^2 +h^3) = 3h^2 
\quad\in A_*\P^3.
\]
\hfill{$\Box$}
\end{exem}

\subsection{Fórmula para o grau do mapa gradiente}

Nosso próximo objetivo é apresentar uma fórmula para o grau do mapa gradiente
em termos da classe de Milnor. Uma tal expressão já é conhecida para
o caso de singularidades isoladas, a saber:
\begin{thm}
\label{milnornum}
Seja $X \subset \P^n$ uma hipersuperfície de grau $k$ com apenas singularidades isoladas. Então
$$\deg \grad X = (k-1)^{n} - \displaystyle \sum_p \mu_{p}(X)$$
onde $\mu_p(X)$ é o número de Milnor.
\end{thm}

\begin{proof}
Veja \cite{Ful84}, \cite{Dim01} ou \cite[Proposition~2.3]{FM12}.
\end{proof}

Como consequência imediata da fórmula acima obtem-se:

\begin{cor}
O mapa gradiente de uma hipersuperfície não-singular em $\mathbb{P}^n$ de grau $k$ tem grau topológico $(k-1)^n$. Em particular, se $k \geq 3$ então a hipersuperfície não é homaloidal (isto é, o mapa gradiente não é birracional).
\hfill{$\Box$}
\end{cor}

O corolário nos diz que uma hipersuperfície homaloidal com grau $k \geq 3$ é neces\-sa\-ria\-mente singular e é, portanto, uma exceção no sentido que as variedades não-singulares formam um aberto no espaço que as parametriza.

Para a generalização do Teorema~\ref{milnornum}, nossa contribuição do Capítulo, o substituto da soma dos números de Milnor são os graus das classes de Milnor da hipersuperfície e de uma seção hiperplana genérica. 

\begin{thm}
\label{thm2}
Seja $X \subset \P^n$ uma hipersuperfície de grau $k$. Então
\[
\deg \grad X = (k-1)^n - \mu(X) - \mu(X\cap H),  
\]
onde $H \subset \P^n$ é um hiperplano genérico.
\end{thm}

\begin{proof}
Sejam $d_0, d_1, \dots, d_n$ os graus projetivos do mapa $\grad X$.
A fórmula do enunciado seguirá da operação de duas relações auxiliares que enun\-ciamos e provamos abaixo.
\begin{afirm}
\label{af1}
$\displaystyle \int \csm(X) - \displaystyle \int \csm(X\cap H)  = 1-(-1)^n d_n$.
\end{afirm}
De fato, pelos Teoremas~\ref{thm1} e \ref{Aluffi2.1},
\[
\csm(X\cap H) = \csm(X)\cdot \frac{h}{1+h},
\]
e
\[
\csm(X) = (1+h)^{n+1} - \sum_{j=0}^{n}d_j(-h)^j(1+h)^{n-j},
\]
portanto
\[
\begin{split}
\int \csm(X)- \int \csm(X\cap H) &= \int \csm(X)\Big(1-\frac{h}{1+h}\Big) \\
&= \int \csm(X)\frac{1}{1+h} \\
&= \int \Big((1+h)^{n} - \sum_{j=0}^{n}d_j(-h)^{j}(1+h)^{n-j-1}\Big)\\ 
&= 1 - (-1)^{n}d_{n}.
\end{split}
\]
\begin{afirm}
$\displaystyle \int \cf(X) - \displaystyle \int \cf(X\cap H)  = 1-(-1)^n(k-1)^n$.
\end{afirm}
De fato, para hipersuperfícies a classe de Fulton depende apenas do grau (Observação~\ref{cfInv}) e logo podemos supor $X$ não-singular, donde $\cf(X) = \csm(X) = c(TX)\cap[X]$. Neste caso o mapa gradiente é um morfismo e portanto pelo teorema de Bézout seu grau é igual a $(k-1)^n$. O desejado segue então da Afirmação~\ref{af1}.

\medskip

O resultado do enunciado do Teorema seguirá da definição de classe de Milnor e das duas afirmações demonstradas:
\[
\begin{split}
\mu(X) + \mu(X\cap H) &= \int \milnor(X) + \int \milnor(X\cap H) \\
&= (-1)^{n-1}\Big(\int \cf(X) - \int \csm(X)\Big) \\
&+ (-1)^{n-2}\Big(\int \cf(X \cap H) - \int \csm(X\cap H)\Big) \\
&= (-1)^{n-1}\Big(\int \cf(X) - \int \cf(X\cap H)\Big) \\
&+ (-1)^{n-1}\Big(\int \csm(X) - \int \csm(X\cap H)\Big) \\
&= (k-1)^n - d_n.
\end{split}
\]
\end{proof}



\begin{rmk}
\label{obs-DimcaPapadima}
Na demonstração do teorema acima, a Afirmação~\ref{af1} apresenta uma outra prova da fórmula de A. Dimca e S. Papadima sobre o grau topológico do mapa gradiente em termos das características de Euler da hipersuperfície e de uma seção hiperplana genérica: \emph{se $X \subset \mathbb{P}^{n}$ é uma hipersuperfície e $H$ é um hiperplano genérico, então}
\[
\deg(\grad X) = (-1)^{n}(1 - \euler(X\setminus X\cap H)).
\]
Esta fórmula foi obtida originalmente em \cite[Theorem 1]{DP03}, utilizando métodos topológicos.
\hfill{$\Box$}
\end{rmk}

\begin{rmk}
Note que, o termo $\mu(X) + \mu(X\cap H)$ na fórmula do Teorema acima é a soma alternada dos coeficientes da classe de Milnor da hipersuperfície, ou equivalentemente, o grau $\displaystyle \int \dfrac{\Mcal(X)}{1+h}$.
\hfill{$\Box$}
\end{rmk}

No exemplo a seguir calculamos de maneira direta o grau do mapa gradiente da secante da curva racional normal em $\P^4$, e como consequência do Teorema~\ref{milnornum}, reobtemos os demais graus projetivos deste mapa. De fato, no Teorema~\ref{thm3} apresentaremos uma fórmula para o grau do mapa gradiente para todas as hipersuperfícies secantes de curvas racionais normais em espaços de dimensão par.

\begin{exem}\label{ex.GdeRNC}
Revisitando o Exemplo~\ref{classgrafsec}: sejam $C \subset \P^4$ a curva quártica racional normal e $X \subset \P^4$ a variedade secante de $C$, a qual é uma hipersuperfície cúbica com $\sing X = C$.
 
Vamos ilustrar como utilizar o Teorema~\ref{thm2} para o cálculo do grau do mapa $\grad X$.

Seguindo o argumento empregado no Exemplo~\ref{classdiscrim}, obtemos
\[
\cf(X) = 3h+6h^2+12h^3-6h^4 \quad \in A_*\P^4,
\]
e, pelo Exemplo~\ref{classgrafsec},
\[
\csm(X) = 3h+6h^2+8h^3+4h^4 \quad \in A_*\P^4,
\]
portanto
\[
\milnor(X) = (-1)^3\left(\cf(X) - \csm(X)\right) = -4h^3 + 10h^4 \quad \in A_*\P^4,
\]
e logo $\mu(X)=10$.

Agora, se $H$ é um hiperplano genérico de $\P^4$, este intersecta $\sing X$ em $4$ pontos com multiplicidade $1$ em cada, portanto $X\cap H$ é uma superfície cúbica de $\P^3$ com $4$ nós ordinários (cúbica de Cayley), e daí concluímos que $\mu(X\cap H)=4$.

Logo, pelo Teorema~\ref{thm2},
\[
d_4 = \deg \grad X = (3-1)^4 - 10 - 4 = 2.
\]

Por fim, vamos calcular os demais graus projetivos do mapa $\grad X$.

Pelo Lema~\ref{lemdegreesection}, $d_3(X) = d_3(X\cap H_1)$, onde $H_1\subset \P^4$ é um hiperplano genérico. Pelo Teorema~\ref{milnornum},
\[
d_3 = (3-1)^3 - \displaystyle \sum \mu_{p} = 8 - 4 = 4.
\]
Pelo mesmo argumento, obtemos $d_2 = \deg\grad(X\cap H_1\cap H_2)$, onde $H_2 \subset \P^4$ é outro hiperplano genérico. Como $X\cap H_1\cap H_2$ é uma curva plana cúbica suave, segue-se que seu mapa gradiente tem grau $(3-1)^2 = 4$, logo $d_2 = 4$. Analogamente, $d_1 = (3-1)^1 = 2$, e $d_0 = 1$.
\hfill{$\Box$}
\end{exem}

\begin{exem}\label{ex.thm2}
Calculamos agora o grau topológico do mapa gradiente da hipersuperfície $X\subset\P^5$ do Exemplo~\ref{classdiscrim}, o discriminante das cônicas de $\P^2$. Utilizamos os dados do mesmo, a saber:
\[
\cf(X) = 3h + 9h^2 + 18h^3 + 6h^4 + 27h^5 \quad\in A_*\P^5,
\]
e 
\[
\csm(X) = 3h + 9h^2 + 14h^3 + 12h^4 + 6h^5 \quad\in A_*\P^5.
\]

Para um hiperplano genérico $H\subset \P^5$, consideremos a inclusão $X\cap H\hookrightarrow H$. Calculemos as expressões das classes para esta seção hiperplana. Da Observação~\ref{cfInv} e do cálculo da classe do tangente de uma cúbica suave, que fizemos no Exemplo~\ref{ex.GdeRNC}, tem-se:
\[
\cf(X\cap H) = 3h + 6h^2 + 12h^3 - 6h^4 \quad \in A_*H.
\]
Por outro lado, do Teorema~\ref{thm1},
\[
\csm(X\cap H) = 3h + 6h^2 + 8h^3 + 4h^4 \quad \in A_*H.
\]
Portanto
\[
\milnor(X) = 4h^3 - 6h^4 + 21h^5 \quad e \quad \milnor(X\cap H) =  -4h^4 + 10h^5 \quad \in A_*\P^5,
\]
e do Teorema~\ref{thm2} obtemos o grau topológico do mapa gradiente de $X$:
\[
\deg \grad X = (3-1)^5 - 21 - 10 = 1.
\]
Logo, $X$ é uma hipersuperfície homaloidal.
\hfill{$\Box$}
\end{exem}

%
%
%
%
\chapter{Secantes de curvas racionais normais}\label{cap2}

\section{Secantes em geral e resultados básicos}


Iniciamos o capítulo expondo definições, observações e resultados conhecidos sobre variedade secantes, dos quais faremos uso posteriormente. Como referências básicas, consulte \cite{Ha92}, \cite{Ru06} e \cite{EH16}.

\begin{defi}
Seja $X\subset \P^n$ uma variedade irredutível não-degenerada, e seja
\[
S_{X}^{0} = \{((x_1, x_2), z);\  z\in \langle x_1, x_2\rangle\} \subset (X\times X\setminus \Delta_X) \times \P^n.
\]
Denotemos por $S_X$ o fecho de $S_{X}^{0}$ em $X\times X\times \P^n$. Esta é uma variedade projetiva irredutível de dimensão $2\dim X +1$, chamada \textbf{variedade secante abstrata para $X$}.

Temos os seguintes mapas de projeção de $S_X$ para $X\times X$ e para $\P^n$:
\[
\xymatrix{& S_X \ar[ld]_{p_1} \ar[dr]^{p_2} & \\
X\times X & & \P^n}
\]
A \textbf{variedade secante de $X$}, denotada $\Sec_2 X$, é a imagem de $S_X$ pela projeção $p_2$:
\[
\Sec_2 X = p_2(S_X) = \overline{\bigcup_{x_1\neq x_2, x_i\in X} \langle x_1, x_2\rangle}.
\] 
Por vezes denotamos $\Sec_2 X$ simplesmente por $\Sec X$. Esta é uma variedade irredutível de dimensão $\leq 2\dim X +1$.
\end{defi}

Agora, para $k\geq 2$ número inteiro, generalizando a construção acima, definimos a \textbf{variedade $k$-secante de $X$} como sendo $\Sec_{k}X = p_2(S^{k}_X)$, onde $S^{k}_X$ é o fecho em $\underbrace{X\times \dots \times X}_{k}\times \P^n$ de
\[
\{((x_0, \dots, x_{k-1}), u); \ \dim \langle x_0, \dots, x_{k-1}\rangle = k-1 \text{ e } u\in \langle x_0, \dots, x_{k-1}\rangle \}.
\]

Em outras palavras, $\Sec_{k}X$ é o fecho do conjunto dos pontos que vivem em subespaços lineares de dimensão ($k-1$) gerados por coleções de $k$ pontos gerais em $X$.

Tome $k_0 = \min \{k\in \mathbb{N};\; \Sec_k X = \P^n\}$.
Então, para cada $0 \leq k \leq k_0$, tem-se que $\Sec_k X$ é uma variedade projetiva irredutível e 
\[
X \subsetneq \Sec_2 X \subsetneq \dots \subsetneq \Sec_k X \subsetneq \dots \subsetneq \Sec_{k_0}X = \P^n.
\]

Nosso interesse reside nas variedades secantes 
de curvas racionais normais, em especial as de grau par.

\smallskip
Uma \emph{curva racional normal} $C\subset \P^d$ (de grau $d$) é uma curva
projetivamente equivalente à imagem do mergulho de 
Veronese $\nu_d\colon \P^1\to\P^d$. Por simplicidade, assumiremos que
$C=\nu_d(\P^1)$ sem maiores comentários. 
Recordamos que  esta curva é de fato 
uma variedade determinantal, realizada como o 
lugar dos pontos $(p_0: \cdots:p_d)\in \P^d$ tais que a matriz
\[
\begin{pmatrix}
p_0 & p_1 & \dots & p_{d-2} & p_{d-1} \\
p_1 & p_2 & \dots & p_{d-1} & p_d
\end{pmatrix}
\]
tem posto $1$.

É curioso que podemos descrever a curva racional normal 
também utilizando outras matrizes. Uma matriz $(x_{ij})$ de 
dimensões $m\times n$ é chamada uma
\emph{matriz de Hankel} (genérica) se $x_{ij}=x_{i+j}$
e a denotamos
\begin{equation}
\label{eq-matrizhankel}
\Hcal_{m,n}(x)=
\begin{pmatrix}
x_0 & x_1 & x_2 & \dots & x_{n-2} & x_{n-1} \\
x_1 & x_2 & \dots & & x_{n-1} & x_{n} \\
\vdots & \vdots & & \vdots & \vdots & x_{m+n-3}\\
x_{m-1} & \dots & \dots & & x_{m-n-3}& x_{m+n-2}
\end{pmatrix}.
\end{equation}
Então,
para qualquer que seja $1< k < d-1$, a curva $C\subset\P^d$ 
é dada pelo lugar dos pontos $p\in \P^d$ tais que 
os $2$-menores de $\Hcal_{k,d+2-k}(p)$ sejam nulos.

Isso se generaliza: a $k$-secante de $C$ é dada pelos zeros de ($k+1$)-menores de \emph{qualquer}
matriz de Hankel nas variáveis $x_0,\dotsc,x_d$ desde que as dimensões da matriz sejam adequadas. 
Mais ainda, tais igualdades valem para os ideais gerados pelos menores,
fato 
conhecido como o \emph{truque de Gruson-Peskine}.
Para enunciá-lo, precisamos de alguma notação:
dada uma matriz $A$ com entradas em
$\C[x_0,\dotsc,x_d]$, seja $I_k(A)$ o
ideal gerado pelos $k$-menores de $A$.
 
\begin{prop}
\label{secant_x_hankel}  Tome inteiros
$m,n,m',n'$ tais que $m+n=m'+n'$. Então, para todo
$k\geq 1$ satisfazendo 
$k\leq\min\{m,n\}$ e $k\leq\min\{m',n'\}$, tem-se
$I_k(\Hcal_{m,n}) = I_k(\Hcal_{m' ,n' })$.

Geometricamente:
seja $C\subset \P^d$ a curva racional normal. Tome $1\leq a\leq d-1$.
Então, para $1\leq k \leq \min\{a, d-a\}$, o ideal
da secante $\Sec_k C\subset\P^d$ é dado pelos ($k+1$)-menores da matriz
de Hankel
\[
\Hcal_{a+1,d-a+1} = 
\begin{pmatrix}
x_0 & x_1 & \dots & x_{d-a} \\
x_1 & x_2 & \dots & x_{d-a +1} \\
\vdots & \vdots & & \vdots \\
x_{a} & x_{a+1} & \dots & x_d
\end{pmatrix}.
\]
\end{prop}
\medskip

\begin{proof}
Para a versão original, veja \cite[Lemma~2.3]{Gruson-Peskine}. Uma prova
independente foi dada por Watanabe em 1985,
em um manuscrito que precede a publicação \cite[Theorem~1]{Watanabe}, em uma versão um pouco mais geral. Consulte também o livro-texto de J.
Harris \cite[Proposition~9.7, p.\thinspace{}103]{Ha92}.
\end{proof}

Em particular, cada secante pode ser obtida como o lugar dos zeros
dos \emph{menores maximais} de uma matriz de Hankel conveniente. 
Por exemplo, o ideal da secante $\Sec_3 C\subset\P^{2r}$ é gerado 
pelos  $4$-menores
da matriz
\[
\mathcal{H}_{4, 2r-2} = \begin{pmatrix}
x_0 & x_1 & \dots & x_{2r-3} \\
x_1 & x_2 & \dots & x_{2r-2} \\
x_2 & x_3 & \dots & x_{2r-1} \\
x_3 & x_4 & \dots & x_{2r}
\end{pmatrix}.
\]

Como pode-se esperar, secantes de curvas racionais normais foram muito estudadas, tanto do ponto de vista algébrico como geométrico.
Abaixo listamos algumas de suas propriedades importantes. Recorde que uma variedade projetiva $X\subset\P^n$ é 
\emph{aritmeticamente Cohen-Macaulay} se seu anel 
de coordenadas homogêneas é Cohen-Macaulay. 

\begin{prop}\label{props_rnc2}
Seja $C\subset \P^{n}$ uma curva racional normal. 
Para $1\leq k \leq [n/2]$, a secante $\Sec_k C$ é uma variedade irredutível, 
reduzida, projetivamente normal e aritmeticamente Cohen-Macaulay.
Seu ideal é dado na Proposição~\ref{secant_x_hankel}.
Mais ainda:
\begin{enumerate}[{\rm(a)}]
\item As secantes tem as dimensões esperadas:
$\dim \Sec_k C = 2k-1$. Em particular temos uma cadeia de inclusões estritas
\[
C \subsetneq \Sec_2 C \subsetneq  \Sec_3 C\subsetneq \cdots \subsetneq \Sec_{[n/2]} C \subsetneq \P^{n}.
\]
\item Quanto aos graus:
$\deg(\Sec_k C) = \binom{n-k+1}{k}$.
\item Em geral, o lugar singular de $\Sec_k C$ não é um esquema reduzido.
Entretanto temos $(\sing \Sec_{k}C)_{red} = \Sec_{k-1}C$ sempre que $2\leq k\leq [n/2]$.
\item Suponha $n=2r$. Sejam $J_r$ o ideal jacobiano do determinante da matriz $\Hcal_{r+1,r+1}$ e $I_r$ o ideal 
gerado pelos menores maximais de $\Hcal_{r,r+2}$. Então $\sqrt{J_r}=I_r$.
\end{enumerate}
\end{prop}

\begin{proof}
As demonstrações podem ser encontradas em diversas referências, por
vezes realizadas com métodos diferentes. Para (a)-(c) indicamos \cite[Theorem~1.56]{Iarrobino-Kanev},
\cite[Proposition~6]{Watanabe}, também \cite[Chapter 10]{EH16}
e \cite[Proposition~4.3]{eisenbud-linear}. Especificamente para (d), consulte
\cite[Proposition~3.3.7]{Mo14} ou \cite[Proposition~3.12]{MS}. 
Finalmente, para um exemplo
onde $J_r$ não é radical, consulte 
\ref{exemplo-nao-radical} no Apêndice.
\end{proof}

Destacamos os casos que necessitaremos adiante.

\begin{corol}
\label{corol_props_rnc2}
Tome $n=2r$. Então:
\begin{enumerate}[{\rm(a)}]
\item $\Sec_r C \subset\P^{2r}$ é uma hipersuperfície, de grau $r+1$.
\item $\Sec_{r-1} C$ tem codimensão $3$ em $\P^{2r}$ e seu grau é $\binom{r+2}{3}$.
\item $(\sing \Sec_{r}C)_{red} = \Sec_{r-1}C$.
\end{enumerate}
\end{corol}

\section{Séries de Hilbert}

Seja $M_k(m,n) \subset\P^{mn-1}$ a \emph{variedade determinantal} (genérica), 
dada pelo lugar dos zeros do ideal 
dos $k$-menores de uma matriz $m\times n$ geral.
Geometricamente, $M_2(m,n)$ é a imagem do mergulho de Segre $\P^{m-1}\times\P^{n-1} \hookrightarrow \P^{mn-1}$.
Aos menores de ordem superior correspondem as secantes
da variedade de Segre, isto é, $M_k(m,n)=\Sec_{k-1}(M_2(m,n))$ (\cite[Example~9.2, p.\thinspace{}99]{Ha92}). Por fim,
$M_n(n,n) \subset \P^{n^2-1}$ é uma hipersuperfície, dada
pelo zeros do determinante.

Dado um subesquema $X\subset\P^{n}$
denotamos por $h_X(t)$ e $P_X(t)$ a sua série e polinômio de Hilbert,
respectivamente. Tome ainda $P_i(t):=\binom{t+i}{i} \in \Q[t]$ para $i\geq 0$.
Observe que $P_0(t),\dotsc, P_n(t)$ formam uma
base para o espaço dos polinômios de grau $\leq n$.

\smallskip
Existe uma expressão determinantal maravilhosa, devida a Abhyankar (\cite{Abhyankar}),
para a série de Hilbert das variedades $M_k(m,n)$,
para quaisquer $m,n,k$.
Nossas referências são \cite[Theorem~6.9]{bruns-conca}
ou \cite[Corollary~1]{conca-herzog} (neste último há um pequeno erro de tipografia,
falta o termo $t^\ell$ na expressão do determinante):
\begin{equation}
\label{eq:Abhyankar0}
h_{M_{k+1}(m,n)}(t)=\frac{ \det 
\left( \sum_\ell \binom{m-i}{\ell}\binom{n-j}{\ell}
t^\ell \right)_{i,j=1,\dots,k}}{ t^{\binom{k}{2}}(1-t)^{k(m+n-k)}}.
\end{equation}
De fato, utilizaremos uma versão levemente diferente desta 
fórmula, provada no lema
logo após o Corolário~1 em \cite{conca-herzog}:
\begin{equation}
\label{eq:Abhyankar}
h_{M_{k+1}(m,n)}(t)=\frac{ \det 
\left( \sum_\ell \binom{m-i}{\ell}\binom{n-j}{\ell+i-j}
t^\ell \right)_{i,j=1,\dots,k}}{(1-t)^{k(m+n-k)}}.
\end{equation}
Embora não sejam amistosas à primeira vista, estas expressões
se tornam extremamente simples
no caso de \emph{menores maximais}, caso que nos será útil 
em aplicações às secantes de curvas racionais normais.

A ideia é que, afortunadamente, as expressões do numerador da série
de Hilbert são, por assim dizer, ``aditivas para menores maximais'' 
quando a diferença entre o número de linhas e colunas é constante.
Tudo não passa de um cálculo elementar com determinantes, algo
trabalhoso de encontrar porém simples de descrever.
\begin{prop}\label{prop:serieHilb}
Sejam $c,k$ inteiros com $c\geq 0$ e $k\geq 1$.
Então a série de Hilbert da variedade determinantal geral
$M_{k+1}(k+1+c,k+1)$ é da forma $q_{k,c}(t)/(1-t)^{k(k+2+c)}$, onde
\begin{equation}
\label{eq:qk}
q_{k,c}(t) = \sum_{j=0}^k 
\textstyle \binom{c+j}{j}t^j 
=
1 + \binom{c+1}{1}t + \binom{c+2}{2}t^2 + \dotsb + \binom{c+k}{k}t^k.
\end{equation}
\end{prop}
\begin{proof} É muito curiosa. 
Aplicaremos a fórmula em \eqref{eq:Abhyankar} para $m=k+1+c$ e $n=k+1$.
Quando $k=1$, é imediato verificar que \eqref{eq:qk} vale para 
qualquer valor de $c$.

Fixamos $c\geq 0$ e fazemos indução em $k$.  
Aplicando \eqref{eq:Abhyankar}, desejamos então calcular o
determinante da matriz
\[
A_k := \begin{pmatrix}
\sum_{\ell\geq 0} \binom{k+1+c-i}{\ell}\binom{k+1-j}{\ell+i-j}t^\ell
\end{pmatrix}_{i,j=1,\dotsc,k}
\]
A observação fundamental é que 
a submatriz $(k-1)\times (k-1)$ inferior à direita, obtida
omitindo-se a primeira linha e a primeira coluna, é exatamente
a matriz $A_{k-1}$. Tendo isto em mente, mostramos que
tomando uma redução conveniente das 
colunas o determinante $\abs{A_k}$ se reduz a
$\abs{A_{k-1}} + \binom{c+k}{c}t^{k}$ e indutivamente
obtemos a fórmula desejada.

Indicamos como proceder no caso em que $c=2$ e $k=3$. Nosso objetivo é mostrar que
\[
\abs{A_3} = \begin{vmatrix}
 10t^3+30t^2+15t+1   & 10t^3+20t^2+5t   & 10t^3+10t^2 \\
    6t^2+12t+3         & 6t^2+8t+1       & 6t^2+4t    \\
     3t+3              & 3t+2           & 3t+1     
\end{vmatrix}
\]
vale $1+3t+6t^2+10t^3$.

Focando na primeira linha,
utilizamos a segunda coluna para anular o termo linear da primeira coluna;
em seguida, a terceira coluna para anular o termo quadrático da primeira coluna; e
assim por diante. Obtemos
\[
\abs{A_3} = \begin{vmatrix}
 10t^3+1   & 10t^3+20t^2+5t   & 10t^3+10t^2 \\
    6t^2         & 6t^2+8t+1       & 6t^2+4t    \\
     3t              & 3t+2           & 3t+1     
\end{vmatrix}
\]
que separamos como:
\[
\begin{vmatrix}
 1   & 10t^3+20t^2+5t   & 10t^3+10t^2 \\
  0         & 6t^2+8t+1       & 6t^2+4t    \\
  0              & 3t+2           & 3t+1     
\end{vmatrix}
+
\begin{vmatrix}
 10t^3   & 10t^3+20t^2+5t   & 10t^3+10t^2 \\
    6t^2         & 6t^2+8t+1       & 6t^2+4t    \\
     3t              & 3t+2           & 3t+1     
\end{vmatrix} 
.
\]
Observe que o primeiro determinante é exatamente $\abs{A_2}$, que por indução 
vale $1+3t+6t^2$. Uma situação similar ocorre no caso geral.

Para calcular o segundo determinante, 
subtraímos a primeira coluna nas restantes, eliminando
os termos líderes. Note que a última coluna fica com ``$1$'' na diagonal:
\[
\begin{vmatrix}
 10t^3   & 20t^2+5t   & 10t^2 \\
    6t^2         & 8t+1       & 4t    \\
     3t              & 2           & 1     
\end{vmatrix} .
\]
Agora, utilizando a última coluna, anulamos todas as entradas da última
linha abaixo da diagonal. A penúltima coluna fica com um ``$1$'' na diagonal.
Utilizamos a penúltima coluna para anular todas as entradas
da penúltima linha abaixo da diagonal, e assim por diante.
No final obtemos uma matriz triangular superior da forma
\[
\begin{vmatrix}
 10t^3   & *  & *  \\
            & 1      &   *  \\
                 &            & 1     
\end{vmatrix}
\]
e portanto $\abs{A_3} = \abs{A_2}+10t^3$, como queríamos.

\smallskip
O caso geral é feito do mesmo modo. 
%
%
\end{proof}

Durante a revisão final deste trabalho, nos foi indicado que o resultado da Proposição~\ref{prop:serieHilb} pode ser obtido
a partir do trabalho seminal de Eagon e Northcott
\cite{EN62a} sobre resoluções de ideais de matrizes
e que de fato a fórmula \eqref{eq:qk} já se encontra,
(em um escopo bem mais geral) em \cite{EN62b}, muito embora em um
formato diferente. Nossa fórmula é portanto apenas uma simplificação daquela obtida por Eagon e Northcott. Chegamos a estas referências graças a M. Mostafazadehfard, a quem agradecemos.

\begin{proof}[Outra demonstração da Proposição~\ref{prop:serieHilb}]
Denote $s=k+1$ e seja $c\geq 0$. Então a série de Hilbert de $M_s(s,s+c)$ 
encontra-se calculada na última fórmula da seção~3
de \cite{EN62b} (tomando-se $p=1$ e $r=s+c$):
\begin{equation}
\label{eq:h0}
h_{M_s(s,s+c)}(t) = \frac{F_{s,c}(t)}{(1-t)^{s(s+c)}} 
\end{equation}
onde 
\begin{equation}
\label{eq:h1}
F_{s,c}(t) = 1-s\binom{s+c}{c}\sum_{q=0}^c(-1)^q\binom{c}{q}\frac{t^{q+s}}{q+s}.
\end{equation}
Comparando os expoentes dos denominadores de \eqref{eq:h0} e do enunciado da Proposição~\ref{prop:serieHilb}, vem que devemos provar:
\begin{equation}
\label{eq:h2}
(1-t)^{c+1}q_{s,c}(t) = F_{s,c}(t)
\end{equation}
onde $q_{s,c}(t)=1+\binom{c+1}{1}t + \dotsb + \binom{c+s-1}{s-1}t^{s-1}$ é a expressão que aparece em \eqref{eq:qk}. 
\smallskip

Faremos indução em $s$. Para $s=1$, devemos mostrar
\[
(1-t)^{c+1} = 1-(c+1)\sum_{q=0}^c (-1)^q  \binom{c}{q}\frac{t^{q+1}}{q+1}
\]
que segue diretamente do fato de que 
$(c+1)\binom{c}{q}=(q+1)\binom{c+1}{q+1}$. Agora, supondo que 
\eqref{eq:h2} vale para $s\geq 1$ dado, devemos então
mostrar que 
\[
(1-t)^{c+1}q_{s+1,c}(t)=F_{s+1,c}(t).
\]
Da expressão de $q_{s,c}(t)$ e da hipótese de indução vem que o que desejamos provar é
\[
\textstyle
F_{s+1,c}(t)-F_{s,c}(t) = \binom{c+s}{s}t^s(1-t)^{c+1}
\]
que por sua vez se segue do fato de que 
\[
\textstyle
(s+1)\binom{s+1+c}{c}\binom{c}{q}+s\binom{s+c}{c}\binom{c}{q+1}=
\binom{s+c}{c}(q+s+1)
\quad \qquad \text{para $q=0,1,\dotsc,c$}.
\]
\end{proof}

\begin{rmk}
	O vetor formado pelos coeficientes do polinômio
$q_{k,c}(t)$ em \eqref{eq:qk} é por
vezes chamado de \textbf{$h$-vetor} 
(de $M_{k+1}(k+1+c,k+1)$) que é portanto
\[
\left(
\textstyle
1,\binom{c+1}{1}, \binom{c+2}{2}, \dotsc, \binom{c+k}{k}  \right).
\] 
Este é um importante invariante em Álgebra Comutativa;
por exemplo, dele se obtém exatamente a partir de qual índice a
função e o polinômio de Hilbert passam a coincidir (\cite[Proposition~4.1.12, p.\thinspace{152}]{BH93}).
\hfill{$\Box$}
\end{rmk}

\smallskip

Passamos agora a calcular a série de Hilbert das secantes das curvas
racionais normais. A ideia é simples: geometricamente,
as secantes de curvas racionais normais são obtidas
a partir de variedades determinantais gerais via certas seções hiperplanas.

Necessitamos de um resultado auxiliar.

\begin{lem}
	\label{lema-h1}
	Sejam $X,H\subset\P^n$, com $H=V(\ell)$ um
	hiperplano tal que $\ell$ não é um divisor de zero
	em $\Ocal_X$. Então 
	\begin{equation}
	\label{eq:hf}
	h_{X\cap H}(t) = h_X(t) (1-t).
	\end{equation}
	Em particular, se $P_X(t)=a_0 P_0(t) + \dotsb+ a_d P_d(t)$,
	então $P_{X\cap H}(t)=a_1P_{0}(t) + \dotsb +a_d P_{d-1}(t)$.
\end{lem}
\begin{proof}
	Como $\ell$ não é um divisor de zero, 
	a multiplicação por $\ell$ nos dá um mapa
	injetivo $\Ocal_X(-1)\to\Ocal_X$, donde obtemos uma sequência exata
	\[
	0 \to \Ocal_{X}(t-1) \to \Ocal_X(t) \to \Ocal_{X\cap H}(t) \to 0
	\]
	para todo $t$, obtendo \eqref{eq:hf}. 
	Além disso, 
	\(
	P_{X\cap H}(t) = P_X(t) - P_X(t-1).
	\)
	Para a última afirmação do enunciado, basta observar que
	\(
	P_i(t) - P_{i}(t-1)=P_{i-1}(t).
	\)
\end{proof}

Finalmente obtemos a fórmula desejada.

\begin{corol}\label{cor-hilb-secrnc}
Seja $C\subset\P^{n}$ uma curva racional normal. Então,
para cada $k=1,\dotsc,[n/2]$:
\[
h_{\Sec_k C}(t) = \frac{
 \sum_{j=0}^k 
\textstyle \binom{n-2k+j} {j}t^j
}
{(1-t)^{2k}}
=
\frac{1 + \binom{n-2k+1}{1}t + \binom{n-2k+2}{2}t^2 + \dotsb + \binom{n-k}{k}t^k}
{(1-t)^{2k}}
\quad.
\]  	
\end{corol}
\begin{proof}	
	A secante $\Sec_k C \subset\P^{n}$ é dada pelos
	menores maximais de uma matriz de Hankel $\Hcal_{k+1,n-k+1}$,
	que por sua vez é obtida a partir de uma matriz genérica $(x_{i,j})$ 
	identificando-se
	as entradas das \emph{anti-diagonais}
	(aquelas cuja soma dos índices da linha e coluna são iguais),
	ou seja,
	impondo-se
	$x_{i,j}-x_{u,v}=0$ sempre que $i+j=u+v$. Ocorre que estas 
	equações formam uma sequência regular no anel de $M_{k+1}(k+1,n-k+1)$
	(veja por exemplo \cite[Proposition 2]{Watanabe})
	e portanto podemos deduzir a série de 
	Hilbert de $\Sec_{k} C$ a partir da série de Hilbert de $M_{k+1}(k+1,n-k+1)$
	via o Lema~\ref{lema-h1}.
O resultado agora é uma consequência direta da Proposição~\ref{prop:serieHilb}.
\end{proof}

Faremos agora uma aplicação
do Corolário~\ref{cor-hilb-secrnc} em uma situação
específica, que nos será útil na demonstração da Proposição~\ref{prop-graus}. 

\begin{exem}\label{ex_seccional_genus}
Seja $C\subset\P^{2r}$ uma curva racional normal e tome
um $\P^4\subset \P^{2r}$ geral. Vem
da Proposição~\ref{props_rnc2} que $\Sec_{r-1}C \subset\P^{2r}$ tem codimensão 3 e seu
lugar singular tem codimensão 5 em $\P^{2r}$. Segue daí pelo Teorema de Bertini que 
$Y_r:= (\Sec_{r-1}C)\cap \P^4$ é uma curva não-singular. Desejamos
calcular o gênero desta curva.
\smallskip

Dada uma variedade projetiva $X$ de dimensão $d$, escreva sua série de Hilbert na forma $q(t)/(1-t)^{d+1}$, onde $q(t)\in \Z[t]$ satisfaz $q(1)\neq 0$. Se $P_X(t)=\sum_{i=0}^d a_{d-i} P_i(t)$ é o polinômio de Hilbert de $X$, então seus os coeficientes são dados por
\begin{equation}
\label{eq-coefsHilbPol}
a_{d-i}=\frac{(-1)^{d-i}q^{(i)}(1)}{i!}
\end{equation}
para $i=0,1,\dotsc,d$ (\cite[Proposition~4.1.9, p.\thinspace{}151]{BH93}). Por outro lado, se $Y\subset X$ é uma
curva obtida por seções hiperplanas genéricas de $X$,
vem então do Lema~\ref{lema-h1} e das identidades em  \eqref{eq-coefsHilbPol}
que o polinômio de Hilbert de $Y$ 
é 
\[
P_Y(t)=a_dP_0(t)+a_{d-1}P_1(t) = a_d + a_{d-1}(t+1) = q(1)-q'(1)(t+1). 
\]
Recorde que o gênero aritmético da curva $Y$ é dado por 
$p_a(Y)=-(p_Y(0)-1)$ (veja por exemplo \cite[Ex.~7.2, p.\thinspace{}54]{Har77}) e logo 
$p_a(Y)=q'(1)-q(1)+1$.
\smallskip

Especializamos para o caso em que $X=\Sec_{r-1} C$.
Como $Y_r$ é não-singular, seu gênero aritmético coincide com 
o seu gênero geométrico $g(Y_r)$. 
Do Corolário~\ref{cor-hilb-secrnc}:
\[
h_{\Sec_{r-1}C}(t)=
 \frac
{\textstyle 1+\binom{3}{2}t+\binom{4}{2}t^2+\dotsb+\binom{r+1}{2} t^{r-1}}
{(1-t)^{2(r-1)}}
\]
e denotando o numerador por $q_r(t)$, temos:
\begin{align*}
g(Y_r) & = q_r'(1)-q_r(1)+1 \\
&= \sum_{j=0}^{r+1}(j-2)\binom{j}{2} - \sum_{j=0}^{r+1}\binom{j}{2} +1 \\
&= \sum_{j=0}^{r+1}(j+1)\binom{j}{2} - 4\sum_{j=0}^{r+1}\binom{j}{2} +1 \\
&= 3\binom{r+3}{4}-4\binom{r+2}{3}+1 \\
&= \frac{1}{24}(r-1)(r-2)(3r^2+11r+12),
\end{align*}
onde a penúltima igualdade decorre das identidades 
\[
\sum_{j=0}^{n} (j+1)\binom{j}{2} = 3\binom{n+2}{4}
\qquad\text{e}\qquad
\sum_{j=0}^{n} \binom{j}{2} = \binom{n+1}{3}
\]
(compare coeficientes em $\frac{(1+x)^{n+1}-1}{(1+x)-1}$ e $(1+x)^n+(1+x)^{n-1}+\dotsb+1$, derive, etc).
\hfill{$\Box$}
\end{exem}

\section{Característica de Euler}

Nesta seção faremos o cálculo da característica de Euler topológica de secantes quaisquer de curvas racionais normais via pontos fixos de uma ação conveniente. O ponto
de partida é:

\begin{prop}\label{C*_action}
{\rm(\cite[Corollary~2]{Bi73})}
Se o grupo multiplicativo $G :=\C^{*}$ age em uma variedade algébrica $X$, e $X^G$ denota o lugar dos pontos fixos de $X$ pela ação, então $\euler (X) = \euler (X^G)$.
\end{prop}

\begin{exem}\label{Euler_Char_P^n}
Como aplicação, calculemos a característica de Euler de $\P^n$.
Considere a ação:
\begin{equation}
\label{eq:acao}
\begin{array}{cccl}
& \! \C^*\times \P^n & \! \longrightarrow & \! \P^n \\
& \! (t,x) & \! \longmapsto & \! t\cdot x= (x_0:t x_1:\cdots :t^n x_n).
\end{array}
\end{equation}
Esta ação tem $n+1$ pontos fixos, os pontos coordenados $e_0, e_1, \dots, e_n$. 
De fato, estes são os únicos: nenhum ponto em $\P^n$ com pelo menos duas 
coordenadas não-nulas fica fixado.
Portanto, pela Proposição~\ref{C*_action}, obtemos $\euler(\P^n) = \euler((\P^n)^{\C^{*}}) = n+1$.
\hfill{$\Box$}
\end{exem}

\begin{exem} 
Calculamos agora a característica de Euler da secante da
quártica racional normal
$C\subset\P^4$.
Afirmamos que, sob a mesma ação do exemplo anterior em $\P^4$,
a secante 
$\Sec_2 C$ fica invariante.
De fato, observe que $\Sec_2 C$ é dada pelo lugar dos zeros do determinante da matriz 
$3\times 3$ de Hankel 
\[
\Hcal(x) = 
\begin{pmatrix}
x_0 & x_1 & x_2 \\
x_1 & x_2 & x_3 \\
x_2 & x_3 & x_4
\end{pmatrix},
\]
e que, por sua vez, 
\[
\det \Hcal(t\cdot x)=t^6\det \Hcal(x).
\]
Podemos assim aplicar a Proposição~\ref{C*_action} à nossa variedade $\Sec_2 C$.
Dos 5 pontos coordenados, apenas $e_2$ não pertence a secante,
uma vez que
\[
\det \Hcal(e_2)=
\begin{vmatrix}
0 & 0 & 1 \\
0 & 1 & 0 \\
1 & 0 & 0
\end{vmatrix} \neq 0.
\]
Logo, $\euler(\Sec_2 C) = \euler((\Sec_2 C)^{\C^{*}}) = 5-1 = 4$.
\hfill{$\Box$}
\end{exem}

O cálculo deste exemplo é um caso particular do seguinte resultado:

\begin{teo}
\label{teo-Euler_Char_Sec_k}
Seja $C \subset \P^n$ uma curva racional normal. 
Então a secante $\Sec_k C \subset \P^n$ fica invariante
pela ação $x\mapsto t\cdot x$ dada em \eqref{eq:acao}.
Mais ainda, vale que $\euler(\Sec_k C) = 2k$ sempre que $2k-1<n$.
\end{teo}

\begin{proof}
Vejamos a primeira parte. Suponha que $2k-1<n$. 
Da Proposição~ \ref{secant_x_hankel} temos que
que $\Sec_k C$ é o lugar dos zeros dos $k+1$-menores da 
matriz de Hankel $\mathcal{H}_{k+1, n-k+1}(x) = (x_{i+j})$.

Dado um subconjunto $I\subset\{0,1,\dotsc,n\}$ com $k+1$ elementos, denote por $p_I$ o 
menor correspondente da matriz de Hankel acima.
Então, via a ação de \eqref{eq:acao}, 
\[
p_I(t\cdot x) = t^{|I|}\cdot p_I (x)
\]
onde $\abs{I}$ é a soma dos elementos de $I$.
Portanto, se $x \in \Sec_k C$ tem-se $t\cdot x \in \Sec_k C$ para todo $t\in \C^*$.

Agora a segunda parte do enunciado.  Aplicando a Proposição~\ref{C*_action}
obtemos 
\[
\euler(\Sec_k C)=\euler((\Sec_k C)^{\C^*}) = \# \mathcal{F}
\]
onde $\mathcal{F}$ é o conjunto dos pontos de $\Sec_k C$ fixos pela ação.

Vimos que o conjunto dos pontos de $\P^n$ fixos pela ação é $\{e_0, e_1, \dots, e_n\}$.

Vamos mostrar que $\mathcal{F} = \{e_0, \dots, e_{k-1}, e_{n-k+1}, \dots, e_n\}$.

De fato, para cada $i\in \{0, \dots, k-1, n-k+1, \dots, n\}$, ao avaliar a matriz $\mathcal{H}_{k+1, n-k+1}$ em $e_i$, esta possuirá ao menos $n-2k+1$ colunas nulas, portanto cada menor maximal de $\mathcal{H}_{k+1, n-k+1}$ se anula nesses vetores.

Por outro lado, os pontos $e_k, e_{k+1}, \dots, e_{n-k}$ não anulam, respectivamente, os menores maximais $p_{12\dots(k+1)}, p_{23\dots(k+2)}, \dots, p_{(n-2k+1)\dots(n-k+1)}$.

Logo, $\euler(\Sec_k C) = \# \mathcal{F} = (n+1)-(n-2k+1) = 2k$.
\end{proof}

\section{Dualidade}

Estabelecemos agora a notação que utilizaremos no restante desta seção.

Fixamos $r\geq 1$.
Seja $C=\nu_{2r}(\P^1) \subset \P^{2r}$ a curva racional normal.
Denotamos por $V$ a imagem de $\P^r$ pelo mergulho 
de Veronese $\nu_2\colon \P^r\hookrightarrow \P^N$ onde
$N={\binom{r+2}{2}-1}$. Vemos $\P^N$ como a projetivização
do espaço das matrizes simétricas de tamanho $r+1$.

É fato conhecido que o ideal de $V$ é
dado pelos $2$-menores de uma matriz simétrica, digamos
$(y_{ij})$ com $i,j=0,\dotsc,r$.
Mais ainda \cite[Examples~1.3.6 (5), p.\thinspace{}26]{FOV99}, o ideal de cada secante $\Sec_k V$ é dada
pelos $k+1$-menores de uma tal matriz. 
Em particular $\Sec_r V$ é uma hipersuperfície, 
dada pelos zeros do determinante.

Por outro lado, do truque de Gruson-Peskine também obtemos os ideais das secantes de $C$ como menores
de uma matriz quadrada, agora de Hankel 
(Proposição~\ref{secant_x_hankel}),
de mesmas dimensões.
Como consequência, cada $k$-secante de $C$
é obtida da $k$-secante de $V$ por uma \textit{mesma} seção linear $L\cong\P^{2r}$, 
isto é, $\Sec_k C = L\cap \Sec_k V$ para todo $k$.
O subespaço linear $L$ é obtido
identificando-se as \emph{anti-diagonais}
(entradas cuja soma dos índices da linha e coluna são iguais)
da matriz simétrica; precisamente,
$L$ é dado pela interseção dos hiperplanos
\begin{equation}
\label{eqs-L}
y_{ij}-y_{uv}=0, \qquad 
\text{com }  
\quad
\begin{cases}
i\leq j, \quad (i,j)\neq(u,v) \\
i+j=u+v\\
v=u  \text{ ou } v=u+1
\end{cases}
\qquad .
\end{equation}
Observe que há exatamente
$\binom{r}{2}$ tais hiperplanos e eles são independentes.

Outro resultado conhecido (veja por exemplo
\cite[Example~5.6]{DH+16})
é que tomando duais na cadeia 
\[
V \subsetneq \Sec_2 V \subsetneq  \cdots \subsetneq \Sec_{r-1} V \subsetneq \Sec_r V \subsetneq \P^N
\]
e identificando $\P^N$ com $(\P^N)^*$,
obtemos a exatamente mesma cadeia, porém com índices revertidos
\[
(\Sec_r V)^*
\subsetneq
(\Sec_{r-1} V)^*
\subsetneq  \cdots \subsetneq
(\Sec_2 V)^*  
\subsetneq 
(V)^* 
\subsetneq 
(\P^N)^* 
\]
isto é, temos que para $k=1,\dotsc,r$ vale
\begin{equation}
\label{eq-dualSecV}
(\Sec_k V)^* \cong \Sec_{r-k+1} V. 
\end{equation}
Em particular, a dual da hipersuperfície $\Sec_r V$ é isomorfa à variedade
de Veronese $V$.

Por fim, denotaremos por $\pi_{L^*}\colon (\P^N)^* \dashrightarrow(\P^{2r})^*$
a projeção centrada \emph{no dual} $L^*$ de $L$. 

\medskip

Começamos com um belo resultado, que 
relaciona dualidade, projeções e seções li\-nea\-res. 
O enunciado abaixo, que faz uso implícito 
da bidualidade (\cite{Kle86}, característica zero),
está em um formato não muito usual, porém mais adequado aos nossos propósitos.

\begin{teo}
[{\cite[Proposition~4.1, p.\thinspace{}31]{GKZ}}]
\label{teo-dualproj}
Sejam $X,L\subset \P^n$ uma subvariedade projetiva e um subespaço linear
 tais que $X^*\cap L^* = \emptyset$ e  $\dim X^* < n - \dim L^*$.
Seja $\pi_{L^*}\colon (\P^n)^* \dashrightarrow (\P^n)^*$ a
projeção centrada no dual de $L$. Então vale a inclusão 
\[
(X\cap L)^* \subseteq \pi_{L^*}(X^*).
\]
Mais ainda, se $\pi_{L^*}(X^*)\cong X^*$, então 
$(X\cap L)^* = \pi_{L^*}(X^*)$.
\end{teo}

\begin{prop}
\label{prop-projveronese}
A dual da hipersuperfície $\Sec_r C \subset \P^{2r}$ é uma projeção isomorfa da variedade de Veronese 
$V=\nu_2(\P^r)$.
\end{prop}
\begin{proof}
A ideia é aplicar o Teorema~\ref{teo-dualproj}
para $X=\Sec_r V$ e para $L$ descrito por \eqref{eqs-L},
uma vez que daí $X^*=(\Sec_r V)^*=V$ por \eqref{eq-dualSecV}
e que $X\cap L=(\Sec_r V)\cap L=\Sec_r C$.
Para tanto, mostraremos $\pi_{L^*}|_V$ define uma imersão fechada de $V\hookrightarrow (\P^{2r})^*$.
Segue daí que $(\Sec_r C)^*=\pi_{L^*}(V)$. 

Com esse intuito, 
basta mostrar que o sistema linear que define $\pi_{L^*}|_V$ separa pontos 
e separa tangentes. Isso ocorre se, e somente se, o centro de projeção $L^*$
não intersecta a va\-rie\-dade das secantes $\Sec V$ e tampouco intersecta a variedade
das tangentes $\Tan V$: veja \cite[Proposition~3.4, p.\thinspace{}309]{Har77} 
ou \cite[Proposition~1.2.8, p.\thinspace{}12]{Ru06}.
Por outro lado, como $V$ é suave, vale que $\Tan V\subseteq \Sec V$ (e.g. \cite[p.\thinspace{}12]{Ru06})
e portanto é suficiente mostrar que $L^*$ não intersecta $\Sec V$,
o que faremos no lema a seguir.
\end{proof}

\begin{lem}
\label{lema-centroproj1}
Com a notação acima, $L^*\cap \Sec V=\emptyset$.
\end{lem}
\begin{proof}
	Denote por $e_{ij} \in \P^{N}$ a matriz cujas entradas $(i,j)$ e $(j,i)$ valem $1$
	e todas as outras entradas são nulas. Então
	$L^*$ é gerado pelos pontos da forma $e_{ij} - e_{uv}$, de forma que 
	os índices $i,j,u,v$ satisfazem as condições descritas em \eqref{eqs-L}.
	
Suponhamos que $(q_{ij}) \in \P^N$ seja um ponto de $L^*\cap \Sec V$. 
Recorde que estamos identificando $q_{ij}=q_{ji}$. É instrutivo acompanhar
o argumento em um caso concreto, o que faremos no Exemplo~\ref{exemplo-r=4}.

Provaremos que $q_{ij}=0$ para todo $i,j$ e portanto tal ponto não existe. Procedemos por indução sobre a soma $s=i+j$ dos índices, onde $s$ indexa as anti-diagonais. Para $s=0,1$, note que $y_{0,0}$ e $y_{0,1}$ não aparecem em nenhuma das equações 
\eqref{eqs-L} e logo $q_{0,0}=q_{0,1}=0$;
suponhamos $s\geq 2$ e que
todas as entradas das anti-diagonais $\leq s-1$ são nulas.

Quando $s$ é par, a anti-diagonal $s$ da matriz $(q_{ij})$ tem a forma
\[
(c_1,c_2,\dotsc,c_m,-\sum_{i=1}^m c_i, c_m,\dotsc, c_2, c_1)
\]
e para $s$ ímpar
\[
(c_1,c_2,\dotsc,c_m,-\sum_{i=1}^m c_i, -\sum_{i=1}^m c_i, c_m,\dotsc, c_2, c_1)
\]
onde $m=[(s+1)/2]$. Por indução, todas as entradas superiores à esquerda 
da anti-diagonal $s$ são nulas.

Por outro lado, como $\Sec V$ é dada pelo anulamento dos $3$-menores, 
escolhendo adequa\-damente submatrizes $3\times 3$, obtemos relações
\[
\textstyle
c_j^2\sum_k c_k=0, \quad\forall j  \qquad\text{e}\qquad c_i^2c_j=0, \quad \forall i\neq j
\]
das quais se segue que os coeficientes $c_i$ devem satisfazer 
\[
\begin{cases}
\displaystyle
\sum_{k=1}^m c_k =0,\\
 c_ic_j=0, & \quad \forall i\neq j
\end{cases}
\]
e a única solução possível é $c_1=\dotsb=c_m=0$.

\end{proof}


\begin{exem}
\label{exemplo-r=4}
Ilustremos a situação com detalhes em um caso significativo, digamos $r=4$.
Aqui,
$C\subset \Sec_4 C\subset \P^8$ e $V=\nu_2 (\P^4)\subset \P^{14}$.
A secante $\Sec V$ é
dada pelos zeros dos menores de ordem $3$ da
matriz simétrica
\[
\begin{pmatrix}
y_{0,0} & y_{0,1} & y_{0,2} & y_{0,3} & y_{0,4} \\
y_{0,1} & y_{1,1} & y_{1,2} & y_{1,3} & y_{1,4} \\
y_{0,2} & y_{1,2} & y_{2,2} & y_{2,3} & y_{2,4} \\
y_{0,3} & y_{1,3} & y_{2,3} & y_{3,3} & y_{3,4} \\
y_{0,4} & y_{1,4} & y_{2,4} & y_{3,4} & y_{4,4} \\
\end{pmatrix}.
\]
A seção linear $L$ é obtida impondo-se que todas as entradas
nas anti-diagonais sejam iguais, o que nos fornece
seis equações, descritas em \eqref{eqs-L}:
\[
y_{0,2}-y_{1,1}, \quad\ 
y_{0,3}-y_{1,2}, \quad\ 
y_{0,4}-y_{2,2},\ y_{1,3}-y_{2,2},\quad\ 
y_{1,4}-y_{2,3}, \quad\ 
y_{2,4}-y_{3,3}.
\]
Denote por $e_{ij} \in \P^{14}$ a matriz cujas entradas $(i,j)$ e $(j,i)$ valem $1$
e todas as outras entradas são nulas. O dual $L^*$ é gerado pelos seis pontos $p_1,\dotsc,p_6$:
\[
e_{0,2}-e_{1,1}, \quad\ 
e_{0,3}-e_{1,2}, \quad\ 
e_{0,4}-e_{2,2},\ e_{1,3}-e_{2,2},\quad\ 
e_{1,4}-e_{2,3}, \quad\ 
e_{2,4}-e_{3,3}.
\]
Suponha que $q=ap_1+bp_2+c_1p_3+c_2p_4+dp_5+ep_6$ é um ponto em $L^*\cap\Sec V$.
Então todos os menores de ordem $3$ da matriz
\[
q=
\begin{pmatrix}
0 & 0 & a & b & c_1 \\
0 & -a & -b & c_2 & d \\
a & -b & -c_1-c_2 & -d & e \\
b & c_2 & -d & -e & 0 \\
c_1 & d & e & 0 & 0 \\
\end{pmatrix}
\]
se anulam. Imediatamente obtemos $a=0$ e, a fortiori, $b=0$. Procedemos para a próxima
anti-diagonal: escolhendo menores $3\times 3$ adequados,
deduzimos $c_1c_2=0$ e $c_i(c_1+c_2)=0$ para $i=1,2$ e daí que
$c_1c_2=c_1+c_2=0$, donde $c_1=c_2=0$. Também 
é imediato que $d=e=0$. Concluímos portanto que $L^*\cap\Sec V=\emptyset$,
como previsto no Lema~\ref{lema-centroproj1}.
\hfill{$\Box$}
\end{exem}

\begin{obs}
Suponha que $C\subset\P^n$ é uma curva racional normal em um espaço projetivo qualquer, seja $n$ par ou não. Então as duais das variedades $\Sec_k C$ são conhecidas, para qualquer que seja $k$ e são dadas pelo lugar de $n$ pontos de $\P^1$ com certas multiplicidades $\Delta(1^{n-2k},2^k)$ (\cite[Proposition~3.1]{LS16}).
\hfill{$\Box$}
\end{obs}

\section{Classe de Mather}

Em \cite{Alu18}, Aluffi mostra como obter a classe de Mather de uma variedade projetiva a partir da classe de Mather da sua variedade dual. Consequentemente obtemos uma fórmula para a classe de Mather
da $r$-secante de uma curva racional normal $C\subset\P^{2r}$.
\begin{teo}
	\label{teo-Mather_sec}
	Dado $r\geq 1$, a classe de Mather da hipersuperfície
	$\Sec_rC\subset\P^{2r}$ é 
	\begin{equation}
		\label{eq:cmaSecr}
		\cma(\Sec_r C)=
		(1+h)^r\sum_{j=0}^{[\frac{r}{2}]} \binom{r+1}{2j+1}h^{2j+1}
		\quad \in A_*\P^{2r}.
	\end{equation}
	Em particular, todos os coeficientes da classe de Mather são positivos.
\end{teo}
\begin{proof}
	Como vimos na Proposição~\ref{prop-projveronese},
	a dual $(\Sec_r C)^*$  \'e isomorfa a $V:=\nu_2(\P^r)$ via a projeção linear $\pi_{L^*}$. Em particular
	\[
	\cma((\Sec_rC)^*) = \cma(\pi_{L^*}(V)) = \cma(V) \quad\in A_*(\P^{2r})^*.
	\]
	Como $V$ é não-singular, sua classe de Mather
	coincide com a sua classe de Chern, cujo cálculo é simples (omitimos o pushforward por $\pi_{L^*}$)
	\begin{align}
		c(TV)\cap [V] &= \nu_{2*} (c(T\P^r)\cap [\P^r]) \tag*{ } \\
		&= \nu_{2*} ((1+h)^{r+1})  \tag*{ }  \\
		&= \nu_{2*} \left(\displaystyle \sum_{j=0}^{r} {\binom{r+1}{j}}[\P^{r-j}]\right)  \tag*{ }  \\
		&= \displaystyle \sum_{j=0}^{r} {\binom{r+1}{j}}2^{r-j}[\P^{r-j}]  \tag*{ }  \\
		&= \displaystyle \sum_{j=0}^{r} {\binom{r+1}{r-j}}2^{j}[\P^j]  
		\quad\quad\in A_*(\P^{2r})^*. 
		\label{eq-cTV}
	\end{align}
	onde na penúltima igualdade usamos 
	$\deg\nu_2(\P^i)=2^i$, o grau da imagem de $\P^i$ pelo
	mergulho de Veronese
	(\cite[Proposition~2.8, p.\thinspace{}49]{EH16}).
	Reescrevendo \eqref{eq-cTV} em termos da classe hiperplana,
	temos portanto que
	\[
	q(h) := \sum_{i=r}^{2r} {\binom{r+1}{i-r}}2^{2r-i}h^i = 
	2^r h^r \left( \big(1+\frac{h}{2}\big)^{r+1} - \big(\frac{h}{2}\big)^{r+1} \right)
	\qquad \in A_*(\P^{2r})^* 
	\]
	é a classe de Mather de  $(\Sec_rC)^*$. Aplicando a fórmula
	de involução de Aluffi em \cite[Theorem~1.3]{Alu18}, obtemos que
	\[
	\cma(\Sec_r C)=
	(-1)^{r+1}
	\Big(
	q(-1-h)-q(-1)\big((1+h)^{2r+1}-h^{2r+1}\big)
	\Big) 
	\qquad \in A_*\P^{2r}
	\] 
	é a classe de Mather da $r$-secante da curva $C$: o ajuste do sinal vem do fato de que $\dim (\Sec_r C)^*=\dim  \nu_2(\P^r)=r$.
	\medskip
	
	Se $r$ é ímpar, então $q(-1)=0$ e  $q(-1-h)=-\frac{1}{2}(1+h)^r((1-h)^{r+1}-(1+h)^{r+1})$,
	donde obtemos a identidade \eqref{eq:cmaSecr}. Por outro lado,
	se $r$ é par, então $q(-1)=1$ e 
	\(
	q(-1-h)=\frac{1}{2}(1+h)^r((1-h)^{r+1}+(1+h)^{r+1})
	\) e logo
	\begin{align*}
		\cma(\Sec_r C) & =   -\frac{1}{2}(1+h)^r\big((1-h)^{r+1}+(1+h)^{r+1}\big)+ \big((1+h)^{2r+1}-h^{2r+1}\big) \\
		&=\frac{1}{2}(1+h)^r\big(-(1-h)^{r+1}+(1+h)^{r+1}\big) - h^{2r+1}
	\end{align*}
	obtendo novamente \eqref{eq:cmaSecr}. A última afirmação do enunciado agora é imediata.
\end{proof}

Expandindo a fórmula~\eqref{eq:cmaSecr} obtemos explicitamente
o coeficiente $\cma(\Sec_r C)_j$ da componente de dimensão $j$ da classe $\cma(\Sec_r C)\in A_*\P^{2r}$:
\begin{equation}\label{cma_j_par}
	\cma(\Sec_r C)_j = \sum_{i=0}^{r-\frac{j+2}{2}} \binom{r}{2i+1}\binom{r+1}{2(r-i)-j-1}, \text{ se $j$ é par}
\end{equation}
e
\begin{equation}\label{cma_j_impar}
	\cma(\Sec_r C)_j = \sum_{i=0}^{r-\frac{j+1}{2}} \binom{r}{2i}\binom{r+1}{2(r-i)-j}, \text{ se $j$ é ímpar}.
\end{equation}

\begin{rmk}
	Para variedades determinantais genéricas, um resultado de positividade análogo ao do Teorema~\ref{teo-Mather_sec} é esperado:  
	X. Zhang conjecturou (\cite[\S 7.6]{Zhang18})
	que os coeficientes da classe de Mather $\cma(M_k(m,n)) \in A_*\P^{mn-1}$ são todos não-negativos.
	De fato, X. Zhang obteve fórmulas para a classe de Mather destas variedades, e consequentemente os respectivos graus polares e a distância euclidiana genérica \cite[Theorem~4.3, Proposition~5.4, Proposition~5.5]{Zhang18}. 
	As fórmulas ali obtidas dependem das classes de
	Chern de certos fibrados sobre Grassmaniannas para
	as quais não existem fórmulas explícitas,
	embora possam ser calculadas em cada caso específico.
	\hfill{$\Box$}
\end{rmk}

\section{Distância euclidiana e graus polares}

Seja $X \subsetneq \P^n$ uma subvariedade fechada própria.
Para um ponto geral $y\in\P^n$, defina $d_y(x) =\sum_{i=0}^n (y_i-x_i)^2$
como o quadrado da distância euclidiana entre $x$ e $y$. 
 O \textbf{grau distância euclidiana de $X$}, denotado por $\eddeg(X)$, 
é definido como o número de pontos críticos de $d_y$ restrita à $X_{\text{suave}}$.
Este número depende da geometria do lugar de interseção
de $X$ com a \emph{quádrica isotrópica} $Q=Z(x_0^2+\dotsb+x_n^2)$.

Consideremos a variedade de incidência
\[
\Phi_X = \overline{\{(p, H);\  p \text{ é um ponto não-singular e } H\supset T_p X\}} \subset  \P^n\times (\P^n)^*
\]
chamada a \textbf{variedade conormal} de $X$ e cuja classe denotamos
\[
[\Phi_X] = \delta_0(X)H^nh+\dots +\delta_{n-1}(X)Hh^n \quad \in  A_*(\P^n\times (\P^n)^*)
\]
onde $H$ e $h$ são, respectivamente, o pull-back das classes hiperplanas de $\P^n$ e de ($\P^n)^*$. Os números $\delta_i(X)$ são chamados os \textbf{graus polares} de $X$.
Aos interessados em aprender mais, indicamos as referências \cite{Pie78} e \cite{Pie15}.

O \textbf{grau distância euclidiana genérica} de $X$ é definido em \cite{AH17} como a soma
dos seus graus polares, isto é,
\begin{equation*}
\label{eq:somadeltas}
\geddeg(X) := \delta_0(X) + \dotsb + \delta_{n-1}(X).
\end{equation*}
A nomenclatura se justifica pois $\eddeg(X)=\geddeg(X)$ se $X$ está em 
posição geral; isto se segue da discussão após Lemma~2.3
em \cite{AH17} e como consequência 
de \cite[Theorem~5.4]{DH+16}:
\emph{se a variedade conormal $\Phi_X$ não intersecta a diagonal 
$\Delta(\P^n) \subset \P^n\times \P^n$, então $\eddeg(X)$
é a soma dos graus polares de $X$}.

Mais ainda, do Teorema de Reflexividade (\cite[Theorem~10.20, p.\thinspace381]{EH16}) decorre que
\[
\Phi_X = \Phi_{X^*} 
\]
desde que tomemos o cuidado de 
trocar os fatores para $\Phi_{X^*}\subset (\P^n)^*\times \P^n$. Logo a soma dos graus polares de $X$ coincide com a soma dos graus da sua variedade dual, ou seja,
\begin{equation}
\label{eq:geddual}
\geddeg(X)=\geddeg(X^*).
\end{equation}

A relação entre os graus polares e a classe de Mather é dada como segue.

\begin{prop}[{\cite[Proposition~2.9]{Alu18}}] 
\label{prop-aluffi-somadeltas}
Seja $X$ uma subvariedade própria de $\P^n$ de dimensão $m$. Então $\delta_i(X)=0$ para $i>m$, e
\[
\geddeg(X)=\delta_0(X)+\dotsb + \delta_m(X) = \displaystyle \sum_{j=0}^{m} (-1)^{m+j} \cma(X)_j (2^{j+1}-1),
\]
onde $\cma(X)_j$ é o coeficiente da componente de dimensão $j$ na classe $\cma(X)\in A_* \P^n$, isto é, $\cma(X)=\sum_j \cma(X)_j[\P^j]$.
\end{prop}

O caso em que a variedade $X$ é não-singular foi provado em \cite[Theorem~5.8]{DH+16}.

	
\begin{cor}
\label{cor-geddeg}
Seja $C\subset\P^{2r}$ uma curva racional normal. Então 
\[
\geddeg(\Sec_r C) = \dfrac{3^{r+1}-1}{2}.
\]\end{cor}

\begin{proof}
É uma aplicação direta da 
Proposição~\ref{prop-aluffi-somadeltas} e dos cálculos
que fizemos no Teorema~\ref{teo-Mather_sec}.
Por \eqref{eq:geddual} podemos considerar a dual $(\Sec_r C)^*$, o que se torna mais conveniente. Como vimos em \eqref{eq-cTV}:
\[
\cma((\Sec_rC)^*) = 
\displaystyle \sum_{j=0}^{r} {\binom{r+1}{r-j}}2^{j}[\P^j]  
\quad\quad\in A_*(\P^{2r})^*. 
\]
Portanto, pela Proposição~\ref{prop-aluffi-somadeltas},
\[
\begin{split}
\geddeg(\Sec_r C) &= \displaystyle \sum_{j=0}^{r} (-1)^{r+j} {\binom{r+1}{r-j}}2^j(2^{j+1}-1)  \\
&= \dfrac{(-1)^r}{2}\left[\displaystyle \sum_{j=0}^{r} (-1)^j {\binom{r+1}{j+1}}4^{j+1} - \displaystyle \sum_{j=0}^{r} (-1)^j {\binom{r+1}{j+1}}2^{j+1}\right] \\
&= \dfrac{(-1)^r}{2}\left[- \displaystyle \sum_{j=0}^{r+1} (-1)^j {\binom{r+1}{j}}4^j +1 + \displaystyle \sum_{j=0}^{r+1} (-1)^j {\binom{r+1}{j}}2^j - 1\right] \\
 &= \dfrac{(-1)^{r+1}}{2}\left[ (1-4)^{r+1} - (1-2)^{r+1}\right] \\
&= \dfrac{3^{r+1}-1}{2}.
\end{split}
\]
\end{proof}

\medskip

R. Piene obteve, em \cite[Théorèm~3]{Pie88}, uma relação precisa entre os graus polares e as classes de Mather para uma variedade projetiva. Esta relação foi recuperada por P.~Aluffi em \cite[Corollary~2.3]{Alu18}, utilizando outros métodos.
Para uma subvariedade fechada própria $X\subset \P^n$, tem-se:
\[
[\Phi_X] = \displaystyle \sum_{j=0}^{n-1} (-1)^{\dim X +j} \cma(X)_j (H+h)^{j+1}H^{n-j} \quad \in A_*(\P^n\times (\P^n)^*),
\]
onde $\cma(X)_j$ é o coeficiente da componente de dimensão $j$ na classe $\cma(X)\in A_*\P^n$ (aqui $H$ e $h$ são como na definição da variedade conormal).
Portanto os graus polares da hipersuperfície $\Sec_r C \subset \P^{2r}$ são obtidos a partir das equações \eqref{cma_j_par} e \eqref{cma_j_impar}:
\begin{equation}
\label{eq-graus-polares}
\delta_i(\Sec_r C) = \sum_{j=i}^{2r-1} (-1)^{j+1} \binom{j+1}{i+1} \cma(\Sec_r C)_j 
\end{equation}
para $i=0, \dotsc , 2r-1$. Eis alguns exemplos calculados explicitamente:

\medskip


\begin{exem}
\label{cma_secmax_rnc} 
Sejam $C\subset \P^{2r}$ a curva racional normal e $X:=\Sec_r C\subset \P^{2r}$ a hipersuperfície $r$-secante. Denotemos os seus graus polares $\delta_i(X)$ simplesmente por $\delta_i$.
\begin{enumerate}
\item
Para $r=2$ temos
\[
\cma(X) = 3h+6h^2+4h^3+2h^4 \quad \in A_*\P^4,
\]
\[
\delta_0 = 0, \delta_1 = 4, \delta_2 = 6, \delta_3 = 3.
\]
\item
Para $r=3$ temos
\[
\cma(X) = 4h+12h^2+16h^3+16h^4+12h^5+4h^6 \quad \in A_*\P^6,
\]
\[
\delta_0 = 0, \delta_1 = 0, \delta_2 = 8, \delta_3 = 16, \delta_4 = 12, \delta_5 = 4.
\]
\item
Para $r=4$ temos
\[
\cma(X) = 5h+20h^2+40h^3+60h^4+66h^5+44h^6+16h^7+4h^8 \quad \in A_*\P^8,
\]
\[
\delta_0 = 0, \delta_1 = 0, \delta_2 = 0, \delta_3 = 16, \delta_4 = 40, \delta_5 = 40, \delta_6 = 20, \delta_7 = 5.
\]
\end{enumerate}
\ \hfill{$\Box$}
\end{exem}

%
%

%
%
\chapter{Mapa gradiente da secante maximal} 
\label{cap-grau}

Denote por $C\subset \P^{2r}$ uma curva racional normal.
Buscamos calcular a classe de Schwartz-MacPherson da hipersuperfície $\Sec_r C$. Como vimos no Capítulo 1, essa classe codifica informações importantes
sobre a geometria e topologia da variedade mergulhada, como a característica de Euler
de seções hiperplanas e a classe de Milnor. Nossa abordagem,
usando o Teorema \ref{Aluffi2.1} de P. Aluffi, é usar a geometria do mapa gradiente associado, o que na prática significa calcular os graus projetivos do mapa. 

Acontece aqui algo peculiar: o mapa gradiente se fatora via uma projeção por uma certa Grassmanianna. Via este caminho somos capazes de calcular os graus projetivos apenas em situações particulares, porém 
suficientes para inferir uma fórmula conjectural para 
o caso geral, como veremos no 
capítulo seguinte.

\section{Fatorando o mapa gradiente}
Seja $\Sec_r C \subset\P^{2r}$ a $r$-secante de $C$. Da Proposição~\ref{secant_x_hankel} vem que esta  hipersuperfície
é dada  pelo determinante 
de uma matriz de Hankel, a saber 
$f=\det(\Hcal_{r+1,r+1})$, que é portanto  um polinômio de grau $r+1$. Seja $\phi_r\colon \P^{2r}\dashrightarrow\P^{2r}$ 
o mapa gradiente correspondente, cujas coordenadas são dadas pelas derivadas parciais $f_{x_0},\dotsc,f_{x_{2r}}$ 
que geram o ideal jacobiano $J_r$ de $f$. Observe 
que o lugar de base do mapa gradiente coincide com 
o lugar singular de $\Sec_rC$, como esquemas.

O caso em que $r=1$ não é emocionante: aqui o mapa gradiente
define um isomorfismo linear $\P^2\rightarrow \P^2$. Na 
discussão que se segue assumimos implicitamente $r\geq 2$.

Em geral $\sing(\Sec_r C)$ não é reduzido, 
mas está suportado na secante
de $r-1$ pontos de $C$. 
Mais ainda, pela Proposição~\ref{props_rnc2} (d)
temos a igualdade de ideais
\(
\sqrt{J_r}=I_r
\)
e em particular vale a inclusão  
\begin{equation}
\label{eq-inclusao}
J_r \subset I_r 
\end{equation}
onde $I_r$ é o ideal gerado pelos  
menores maximais ${p_I}$ da matriz
de Hankel 
\begin{equation}
\label{eq:Hr,r+2}
\Hcal_{r,r+2} = 
\begin{pmatrix}
x_0 & x_1 & \dots & x_{r+1} \\
x_1 & x_2 & \dots & x_{r+2} \\
\vdots & & & \vdots \\
x_{r-1} & x_{r} & \dots & x_{2r}
\end{pmatrix}
\end{equation}
onde $I$ percorre todos $r$-subconjuntos 
de 
$\{1,2,\dotsc,r+2\}$. Cada um destes menores tem grau $r$, o mesmo grau das derivadas parciais $f_{x_i}$; segue-se
por \eqref{eq-inclusao} que cada derivada parcial pode ser escrita 
como uma combinação linear dos $p_I$'s, de maneira única, uma vez que estes
menores são linearmente independentes%
\footnote{Isto é um fato geral: considere uma matriz de Hankel qualquer $\Hcal=\Hcal_{m,n}(x)$
como em \eqref{eq-matrizhankel}. Então, na ordem lexicográfica, os termos líderes
dos menores maximais de $\Hcal$ são todos distintos entre si e portanto
os menores maximais são linearmente independentes sobre $\C$.
}.
Isto nos leva naturalmente a  considerar o mapa 
\begin{align}
\label{lala}
\psi_r\colon \P^{2r} &\dashrightarrow \P^N\\
q &\mapsto (p_I(q)) \tag*{}
\end{align}
cujas coordenadas são dadas pelos menores maximais
de \eqref{eq:Hr,r+2}, sendo $N=\binom{r+2}{r}-1$. 
Note que o lugar de base deste mapa é exatamente a
secante $\Sec_{r-1} C$ e portanto um esquema reduzido.

Uma vez que é dado por menores maximais, 
a imagem do mapa $\psi_r$ está contida em $\G(r-1,r+1)\hookrightarrow\P^N$, a imagem da Grassmanianna
de $(r-1)$-planos em $\P^{r+1}$ via o mergulho de Pl\"ucker. Como pode-se esperar, tal mapa 
foi considerado classicamente. A propriedade
que necessitamos já aparece em um artigo de J. Semple
de 1931 (\cite{Sem31}); veja também \cite[Proposition~4.4]{RS01}. 
\begin{teo}
\label{teo-Semple}
Para todo $r\geq 1$, o mapa 
$\psi_r\colon\P^{2r}\dashrightarrow \G(r-1,r+1)$ é birracional.
\end{teo} 

Voltando à nossa construção: escrevemos 
as derivadas parciais $f_{x_i}=\sum_{I}a_{i,I}p_I$ em função
dos menores maximais de \eqref{eq:Hr,r+2}  e consideramos o mapa $\P^N\dashrightarrow \P^{2r}$
dado por estas equações: ou seja, a projeção $\pi_{L_r}$ com centro no
subespaço $L_r\subset\P^N$ dado pelos zeros dos $2r+1$ hiperplanos
\begin{equation}
\label{eq-lineares}
\sum_I a_{i,I}x_I=0 \qquad\qquad i=0,1,\dotsc,2r
\end{equation}
onde $x_I$ denotam as coordenadas de $\P^N$.
\begin{obs}
\label{obs-dimLr}
É oportuno observar que $\codim L_r=2r+1$: de fato,
conhecemos os coeficientes $a_{i,I}$ (veja-os em
\eqref{eq:Lr}; é conveniente adiar sua apresentação)
e um momento de reflexão mostra que os hiperplanos
em \eqref{eq-lineares}
são independentes. Em particular, 
a projeção $\pi_{L_r}$ é sobrejetora sobre $\P^{2r}$.
\hfill{$\Box$}
\end{obs}

Fixemos a notação $\G_r := \G(r-1,r+1)$. Temos então um diagrama comutativo
\begin{equation}
\label{eq-fatoracao}
\xymatrix{ 
& \G_r \ar@{-->}[dr]^{\pi_r} \\ \P^{2r}\ar@{-->}[ru]^{\psi_r} \ar@{-->}[rr]^{\phi_r}& & \mathbb{P}^{2r} }
\end{equation}
onde $\pi_r :=\pi_{L_r}|_{\G_r}$ é a restrição da projeção. Em outras palavras, o mapa gradiente se fatora pela Grassmanniana. 

Do Lema~\ref{intersecao-vazia} abaixo segue-se que $\pi_r$ é dominante e em particular que
o mapa $\phi_r$ é sempre dominante (M. Mostafazadehfard 
\cite[Proposition~3.3.11]{Mo14} provou diretamente que $\phi_r$ é dominante via
um argumento sobre a Hessiana da equação que define $\Sec_rC$).
Aqui é natural indagar: \emph{o mapa gradiente
 $\phi_r$ é birracional?} Para $r=1$ é um isomorfismo linear e logo a resposta é sim. Para $r=2$,
M. Mostafazadehfard e A. Simis mostraram, algebricamente,
que $\phi_2$
\underline{não é} birracional (\cite{MS}). Mais ainda, 
propuseram a seguinte:

\begin{conjectura}[{\cite[Conjecture~3.18]{MS}}]
\label{conj-Maral-Aron}
Para $r\geq 2$, a hipersuperfície $\Sec_r C$ não
é homaloidal. Em outras palavras, o mapa  $\phi_r\colon \P^{2r}\dashrightarrow\P^{2r}$ não é birracional.
\end{conjectura}

Interpretando em termos do grau do mapa gradiente:
o fato de que $\phi_r$ é dominante nos diz que
$\deg(\phi_r)>0$; vale que $\deg(\phi_1)=1$; e a 
Conjectura~\ref{conj-Maral-Aron} afirma que $\deg(\phi_r)>1$ para $r\geq 2$.

\medskip
Mostraremos adiante que esta conjectura é verdadeira.
Nossa estratégia consiste em determinar
$\deg(\phi_r)$ para todo $r$. Por sua vez isto
seguirá como consequência de um resultado ainda mais compreensivo (Teorema~\ref{thm2.9}),
surpreendente em um primeiro momento: para qualquer $r$, \emph{todos} 
os graus projetivos do mapa $\phi_r$
e do mapa $\psi_r$ são os mesmos! 
O passo fundamental,
assunto da próxima seção,
consiste do 
feito de que a projeção $\pi_r$ é de fato
um morfismo. Antes, ilustramos nosso método quando $r=2$.

\begin{exem} 
	\label{exemplo-P4}
	Tome $C \subset \P^4$ a quártica racional normal e considere sua secante $\Sec_2 C$,
	dada pelos zeros do polinômio
	\begin{equation*}
	f = \det\Hcal(x) =
	\det \begin{pmatrix}
	x_0 & x_1 & x_2 \\
	x_1 & x_2 & x_3 \\
	x_2 & x_3 & x_4
	\end{pmatrix}.
	\end{equation*}
	Para simplificar o cálculo das derivadas parciais aplicamos a regra 
	da cadeia à composição
	\(
	\A^5\stackrel{\Hcal}{\longrightarrow} \A^9 \stackrel{\det}{\longrightarrow} \C
	\)
	obtendo
	\[
	f_{x_k} = \sum_{i+j=k} m_{ij}
	\]
	onde 
	$m_{ij}$ é $(-1)^{i+j}$ vezes o  $(i,j)$-menor of $\Hcal(x)$.
	Assim:
	\[
	f_{x_0} = 
	\begin{vmatrix} x_2 & x_3\\x_3 & x_4\end{vmatrix},
	\quad
	f_{x_1} = -2
	\begin{vmatrix} x_1 & x_3\\x_2 & x_4\end{vmatrix},
	\quad
	f_{x_2} = 2
	\begin{vmatrix} x_1 & x_2\\x_2 & x_3\end{vmatrix} +
	\begin{vmatrix} x_0 & x_2\\x_2 & x_4\end{vmatrix}, 
	\]
	\[
	f_{x_3} = -2
	\begin{vmatrix} x_0 & x_2\\x_1 & x_3\end{vmatrix},
	\quad
	f_{x_4} = 
	\begin{vmatrix} x_0 & x_1\\x_1 & x_2\end{vmatrix}.
	\]
	Da inclusão em \eqref{eq-inclusao}, cada derivada parcial é uma combinação linear dos menores maximais 
	$p_{ij}$ ($1\leq i<j \leq 4$) da matriz
	\[
	\begin{pmatrix}
	x_0 & x_1 & x_2 & x_3 \\
	x_1 & x_2 & x_3 & x_4
	\end{pmatrix}
	\]
	a saber
	\[
	f_{x_0} = p_{34},
	\quad
	f_{x_1} = -2p_{24},
	\quad
	f_{x_2} = 3p_{23}+p_{14},
	\]
	\[
	f_{x_3} = -2p_{13},
	\quad
	f_{x_4} = p_{12}.
	\]
	Os hiperplanos correspondentes em $\P^5$ são independentes e logo o subespaço $L_2\subset\P^5$ dado pela interseção destes hiperplanos  é um ponto, a saber $q=(0: 0: -3: 1: 0: 0)$. Recorde que
	a Grassmanianna $\G_2=\G(1,3)$ das retas de $\P^3$ é dada pela equação
	\[
	x_{12}x_{34}-x_{13}x_{24}+x_{14}x_{23}=0
	\]
	e portanto $q\not\in \G_2$. Assim, na fatoração 
	 \eqref{eq-fatoracao} do mapa gradiente 
	\begin{equation}
	\xymatrix{ 
		& \G_2 \ar@{->}[dr]^{\pi_2} \\ \P^4\ar@{-->}[ru]^{\psi_2} \ar@{-->}[rr]^{\phi_2}& & \mathbb{P}^4 } 
	\end{equation}
	a projeção
	$\pi_2=\pi_q|_{\G_2}\colon\G_2\to\P^4$ é de fato um morfismo e daí
	$\deg(\pi_2)=\deg\G_2=2$. Finalmente, pelo Teorema~\ref{teo-Semple} o mapa $\psi_2$ é birracional sobre $\G_2$ e portanto
	\[
	\deg(\phi_2) = \deg (\psi_2)\cdot\deg(\pi_2) = 1\cdot 2=2
	\]
	e, em particular, $\phi_2$ não é birracional.
	\hfill{$\Box$}
\end{exem}

\section{Lema sobre o centro de projeção}

Mostramos agora que o  subespaço linear $L_r\subset\P^N$ não intersecta a Grassmanniana $\G_r = \G(r-1,r+1)$. 
Por simplicidade assumimos $r\geq 2$.
Ponha $N={\binom{r+2}{r}-1}$. 

Tomamos em $\P^N$ coordenadas $x_I$ indexadas
por $r$-sequências ordenadas $I=(i_1<i_2<\dotsb<i_r)$ de inteiros  
entre $1$ e $r+2$. Em seguida, estendemos para $r$-sequências
$J$ quaisquer: se $J$ é uma permutação de $I$, então
tomamos $x_J:=\pm x_I$ com o devido sinal; e se $J$ possui elementos repetidos,
fazemos $x_J:=0$. Com estas convenções, as equações
que definem o ideal da Grassmanniana $\G_r\hookrightarrow\P^N$ 
(as \emph{relações de Plücker}) são (\cite[Theorem~1.3, p.\thinspace{}94]{GKZ}):
\begin{equation}
\label{eq:G1}
\sum_{a=1}^{r+1} (-1)^a x_{i_1\dots i_{r-1}j_a} x_{j_1\dots \hat{j_a} \dots j_{r+1}}=0
\end{equation}
para quaisquer duas sequências 
$1\leq i_1 < \dots < i_{r-1}\leq r+2$ e 
$1\leq j_1 < \dots < j_{r+1}\leq r+2$.

Estas expressões podem ser dramaticamente simplificadas 
indexando-se as coordenadas de uma outra forma:
tomando-se o complementar. Dados $i,j\in\{1,\dotsc,r+2\}$ com $i<j$,
definimos
\[
q_{ij}:=x_I \qquad\qquad \text{onde } I=\{i_1<i_2<\dotsb<i_r\}
=\{1,2,\dotsc,r+2\}\setminus\{i,j\}.
\]
Com esta escolha as relações de Plücker para $\G_r$
são simplesmente
\begin{equation}
\label{eqgrass1}
q_{ij}q_{kl}-q_{ik}q_{jl}+q_{il}q_{jk} = 0, 
\qquad\qquad
\text{onde $1\leq i < j < k < l \leq r+2$.}
\end{equation}
e não é surpresa que estas sejam as equações da 
Grassmanianna de retas de $\P^{r+1}$.

Como vimos na discussão em \eqref{eq-lineares}, as equações do centro de projeção
são dadas pelas expressões das derivadas parciais
da equação de $\Sec_r C$ em termos
dos menores que definem o ideal de $\Sec_{r-1} C$. Felizmente
o (árduo) cálculo explícito destas expressões 
foi realizado por M. Mostafazadehfard em sua tese
(\cite[Lemma~3.3.4]{Mo14}) e encontram-se também em \cite[(11)]{MS}; com a nossa 
escolha de índices, as equações para $L_r$ são
\begin{equation}
\label{eq:Lr}
  \sum_{\substack{i+j=p\\ 1\leq i<j}} (j-i)q_{ij} = 0, \qquad\qquad  p=3,4,\dotsc,2r+3.
\end{equation}
Estamos em posição de provar o ponto crucial da nossa construção:
o centro de projeção não intersecta a Grassmanianna.

\begin{lem}
\label{intersecao-vazia}
$L_r \cap \G_r = \emptyset$. Como consequência, o mapa $\pi_r$ em \eqref{eq-fatoracao} é um morfismo.
\end{lem}
\begin{proof}
Suponha que exista um ponto $(q_{ij}) \in \P^N$ na interseção entre
$L_r$ e $\G_r$.
Mostraremos que cada uma das coordenadas $q_{ij}$ deve ser nula
e logo tal ponto não existe.
Argumentaremos por indução na soma $p$ dos índices das coordenadas.
 
Para $p=3,4$, segue diretamente de \eqref{eq:Lr}
que $q_{12} = q_{13} = 0$.
Tomemos $p>4$ e suponhamos que $q_{ij}=0$ sempre que $i+j<p$. 
Considere o conjunto dos pares cuja soma das entradas é $p$:
\[
P = \{(i, j)\mid  1 \leq i < j \text{ e } 
i+j=p\}.
\]
Repare que $P$ tem pelo menos dois elementos. Afirmamos que
\begin{equation}
\label{eq:ij.kl}
q_{ij}\cdot q_{kl}=0  \qquad\qquad \text{para quaisquer\ } (i,j)\neq(k,l)\text{ em } P.
\end{equation}
De fato, tome pares $(i,j)\neq(k,l)$ em $P$. Como $i+j=k+l=p$,
segue-se que $i, j, k, l$ 
são distintos entre si. Permutando os pares se necessário,
podemos supor $i<k$ e portanto $i<k<l<j$.
Considere a equação correspondente em 
(\ref{eqgrass1}):
\begin{equation*}
q_{ik}q_{lj}-q_{il}q_{kj}+q_{ij}q_{kl} = 0.
\end{equation*}
Como $i+k<i+j=p$ e $i+l<k+l=p$, segue da hipótese de indução que 
$q_{ik}=q_{il}=0$ e daí, pela equação acima,  deduzimos 
$q_{ij}q_{kl} = 0$, o que prova a afirmação em \eqref{eq:ij.kl}.

\smallskip
Como nosso ponto também pertence ao centro de projeção $L_r$,
suas coordenadas satisfazem as equações 
\eqref{eq:Lr}; temos portanto um sistema
\begin{equation*}
\begin{cases}
q_{ij}q_{kl}=0, \qquad \forall (i, j)\neq(k, l)\in P\\
 \sum_{(i,j)\in P} (j-i)q_{ij}=0 
\end{cases}
\end{equation*}
donde decorre $q_{ij}=0$ para todo $(i,j)\in P$:
basta observar que dadas duas sequências 
$t_1,\dotsc, t_s$ e $a_1, \dotsc, a_s$,
 de números complexos tais que $a_i\neq 0$ para todo $i$, 
as relações
\[
\begin{cases}
t_it_j=0, \qquad\text{para todo $i\neq j$}\\
\sum_{i=1}^s a_it_i=0 
\end{cases}
\]
implicam em $t_1 = \dotsb = t_s = 0$.
\end{proof}
De imediato  obtemos um resultado de interesse:

\begin{cor}
	\label{cor-grad-dominante}
	Para todo $r\geq 1$, o mapa gradiente $\phi_r$ é dominante.
\end{cor}
\begin{proof}
Dado um ponto $p\in \P^{2r}$ qualquer, segue da Observação~\ref{obs-dimLr} que a pré-imagem
$\pi_{L_r}^{-1}(p)$ tem codimensão $2r$. Como a Grassmanianna  $\G_r$ tem dimensão $2r$ e é disjunta do centro de projeção $L_r$ pelo Lema~\ref{intersecao-vazia}, vem que a fibra $\pi_r^{-1}(p)$ é não-vazia. O resultado segue então da
fatoração~\eqref{eq-fatoracao} do mapa gradiente e do
Teorema~\ref{teo-Semple}.
\end{proof}

\section{Graus projetivos}

Estamos prontos para o resultado principal deste capítulo. Depois da nossa longa preparação, a demonstração é curta.

\begin{thm}
	\label{thm2.9}
	Os graus projetivos dos mapas $\phi_r$ e $\psi_r$ coincidem.
\end{thm}

\begin{proof} Como observamos após a Definição~\ref{def-graus-projetivos}, o $j$-ésimo grau projetivo 
$d_j(\phi_r)$ é obtido tomando-se o  grau da
pré-imagem de um  $\P^{2r-j}\subset\P^{2r}$ geral via $\phi_r$. Analogamente, obtemos $d_j(\psi_r)$ 
tomando o grau de $\psi_r^{-1}(\P^{N-j}\cap\G_r)$,
uma vez que, pelo Teorema~\ref{teo-Semple}, a Grassmanianna $\G_r$ é o fecho da imagem deste mapa.

Tome um subespaço linear $\P^{2r-j}\subset\P^{2r}$ qualquer. 
Da Observação~\ref{obs-dimLr} vem que a pré-imagem pela projeção $\pi_{L_r}$ tem codimensão $j$ em $\P^N$. Como a grassmanianna $\G_r$ é disjunta do centro de projeção $L_r$ pelo Lema~\ref{intersecao-vazia}, a pré-imagem pelo mapa $\pi_r$ é da forma 
$\pi_r^{-1}(\P^{2r-j})=\P^{N-j}\cap \G_r$. Por outro lado, da fatoração 
em \eqref{eq-fatoracao}
\[
\phi_r^{-1}(\P^{2r-j}) = 
(\pi_r\circ\psi_r)^{-1}(\P^{2r-j}) = 
\psi_r^{-1}(\pi_r^{-1}(\P^{2r-j})) = 
\psi_r^{-1}(\P^{N-j}\cap\G_r).
\]
Agora, se tomamos $\P^{2r-j}$ geral, segue de sucessivas aplicações do Teorema de Bertini 
que a interseção $\P^{N-j}\cap\G_r$ é transversal, uma vez que $\pi_r$ não possui pontos de base. 
Isto é suficiente para concluir que os graus projetivos dos mapas $\phi_r$ e $\psi_r$ são os mesmos.
\end{proof}

\begin{rmk}
Na demonstração do Teorema~\ref{thm2.9} há uma alternativa para a utilização do Teorema de Bertini. O Teorema de Transversalidade de Kleiman \cite[Corollary~4]{Kle74} trata, sob certa hipóteses, sobre a transversalidade da interseção de dois subesquemas de um esquema algébrico integral. Fixe $j \in \{0,\dotsc,2r\}$. Considere a variedade ambiente
$\P^N$ e os subesquemas $\G_r$ e $Y$ um elemento qualquer de
\begin{align*}
\Lcal_j & :=
 \{\text{espaços lineares de codimensão $j$ em $\P^N$ contendo $L_r$}\} \\
& = \{\text{fecho das pré-imagens por $\pi_{L_r}$ dos espaços lineares de codimensão $j$ em $\P^{2r}$}\}.
\end{align*}
Note que o grupo $G:= \{g\in \text{Aut}(\P^N);\  g(L_r)=L_r\}$ age em $\Lcal_j$. Mais ainda, como $G$ age transitivamente em $\P^N$, 
 segue do Teorema de Transversalidade que para um elemento geral $g\in G$, a interseção $(gY)\cap \G_r$ é transversal.
 \hfill{$\Box$}
\end{rmk}

Apresentamos agora algumas consequências importantes do Teorema~\ref{thm2.9}.

\begin{cor}\label{thm3}
	Sejam $C\subset \P^{2r}$ a curva racional normal e $\phi_r\colon \P^{2r} \dashrightarrow \P^{2r}$ o mapa gradiente da hipersuperfície $\Sec_r C$. Então:
	\[
	\deg(\phi_r) = \deg\G_r = \frac{1}{r+1}\binom{2r}{r}.
	\]
	Em particular $\phi_r$ não é birracional para $r\geq2$ e por conseguinte a Conjectura~\ref{conj-Maral-Aron} é verdadeira.
\end{cor}

\begin{proof} Argumentando como  no final do Exemplo~\ref{exemplo-P4} obtemos 
\[
\deg(\phi_r) = \deg(\pi_r) = \deg\G_r = \deg(\G(1,r+1))
\]
cujo grau é exatamente $\frac{1}{r+1}\binom{2r}{r}$.
\smallskip

Alternativamente, recorde que o último grau projetivo
é o produto do grau do mapa com o grau da sua imagem.
Assim, temos que $d_{2r}(\phi_r)=\deg(\phi_r)\cdot 1$ pois
o mapa gradiente é dominante e também que $d_{2r}(\psi_r)=\deg(\psi_r)\cdot\deg\G_r=\deg\G_r$, pois pelo Teorema~\ref{teo-Semple} o mapa $\psi_r$ é birracional. O resultado agora segue do Teorema~\ref{thm2.9}.
%
\end{proof}

\begin{cor}\label{cor-thm2.9}
A classe de Segre do lugar singular da hipersuperfície $\Sec_r C$ coincide com a classe de Segre da secante $\Sec_{r-1} C$. De maneira precisa, 
\[
s(\sing(\Sec_rC),\P^{2r}) = s(\Sec_{r-1}C,\P^{2r})
\quad\in A_*\P^{2r}.
\]
\end{cor}

\begin{proof}
Vimos que $\sing(\Sec_rC)$ e $\Sec_{r-1}C$ são, respectivamente, os lugares de base dos mapas $\phi_r$ e $\psi_r$. Tendo em mente que os coeficientes das classes
de Segre do lugar de base de um mapa racional
são dados pelos graus projetivos  (Proposição~\ref{Aluffisegre}), o corolário decorre
imediatamente do Teorema~\ref{thm2.9}.
\end{proof}

Isto é surpreendente uma vez que $\Sec_{r-1}C$ é reduzida mas em geral o lugar singular de $\Sec_r C$ não é. Isto
nos será bastante útil na próxima seção.

\smallskip

Por fim, uma fórmula para a característica de Euler topológica de uma seção hiperplana genérica de $\Sec_r C$.

\begin{cor}\label{euler_genericsection_sec}
Para um hiperplano genérico $H\subset \P^{2r}$, temos:
\[
\euler((\Sec_r C)\cap H) = \dfrac{1}{r+1}\binom{2r}{r} - 1+2r .
\]
\end{cor}

\begin{proof}
Segue da combinação do Teorema~\ref{teo-Euler_Char_Sec_k}, do Corolário~\ref{thm3} e da fórmula de Dimca e Papadima 
na Observação~\ref{obs-DimcaPapadima},
que relaciona o grau do mapa gradiente de $\Sec_r C$ com a sua característica de Euler e a de uma seção hiperplana genérica:
\begin{align*}
\frac{1}{r+1}\binom{2r}{r} & = (-1)^{2r}\left(1 - \euler(\Sec_r C) + \euler((\Sec_r C) \cap H)\right) \\
&= 1 - 2r + \euler((\Sec_r C) \cap H).
\end{align*}
\end{proof}

\subsection{Algumas fórmulas}
\label{algumas-formulas}

Sejam $d_0(\phi_r),\dotsc,d_{2r}(\phi_r)$ os graus projetivos do mapa gradiente $\phi_r\colon \P^{2r} \dashrightarrow \P^{2r}$. 
Finalizamos este capítulo com uma relação de fórmulas para
alguns destes graus. Em princípio estas
fórmulas são calculadas via seções lineares do lugar singular da hipersuperfície
secante; porém via o Corolário~\ref{cor-thm2.9}
reduzimos o estudo para seções lineares
da secante $\Sec_{r-1}C\subset\P^{2r}$, que é reduzida e de dimensão menor, o que 
torna o cálculo possível.

\begin{prop}
\label{prop-graus}
As seguintes fórmulas são válidas:
\begin{enumerate}[{\rm(a)}]
\item $d_0(\phi_r)=1$, $d_1(\phi_r) = r$, $d_2(\phi_r)=r^2$.
\item $d_3(\phi_r) = \frac{1}{6}r(r-1)(5r+2)$.
\item $d_4(\phi_r) = \frac{1}{12}r(r-1)(7r^2-5r-6)$.
\end{enumerate}
\end{prop}

\begin{proof}
%
Para calcular os graus projetivos do mapa $\phi_r$, segue do Teorema~\ref{thm2.9} que podemos usar o mapa $\psi_r$; mais ainda, como
 observamos depois da Definição~\ref{def-graus-projetivos}, basta 
 tomar a restrição a subespaços lineares genéricos. Em resumo,
\[
d_i(\phi_r) = d_i(\psi_r) = d_i(\psi_r |_{\P^i})
\]
onde $\P^i \subset \P^{2r}$ é genérico de dimensão $i$.

\smallskip

(a) Recorde que o mapa $\psi_r$ tem como lugar de base a
secante $\Sec_{r-1}C$, de codimensão 3 no ambiente $\P^{2r}$. Logo, para
$i=0, 1, 2$, a restrição a um $\P^i$ geral é um morfismo e portanto
\[
d_0(\psi_r) = 1, \quad d_1(\psi_r) = r \quad \text{e} \quad d_2(\psi_r)=r^2.
\]

\smallskip

Quando $i\geq 3$, calculamos a classe de Segre do lugar de base do mapa $\psi_r|_{\P^i}$, 
a saber $(\Sec_{r-1}C)\cap \P^i$, e utilizamos
a Proposição~\ref{Aluffisegre}. Observe que,
como vimos na Proposição~\ref{props_rnc2}, a secante
$\Sec_{r-1}C$ tem grau $\binom{r+2}{3}$.
\smallskip

(b) Para $i=3$,
$(\Sec_{r-1}C)\cap \P^3$ é um conjunto de $\binom{r+2}{3}$ pontos, cada um com multiplicidade $1$, 
uma vez que o lugar singular de $\Sec_{r-1}C$ tem suporte em $\Sec_{r-2}C$ e portanto vive em codimensão $5$ em $\P^{2r}$, pela Proposição~\ref{props_rnc2}.
Logo, $s((\Sec_{r-1}C)\cap \P^3,\P^3) = \binom{r+2}{3} [\P^0]$, e de \eqref{eq-djs-segre} obtemos
\[
d_3(\psi_r) = r^3 - \binom{r+2}{3} = \dfrac{1}{6}r(r-1)(5r+2).
\]

\smallskip

(c) Tratamos agora o caso $i=4$, mais elaborado. Consideramos a curva 
\[
Y_r := (\Sec_{r-1}C)\cap \P^4,
\]
o lugar de base da restrição $\psi_r|_{\P^4}$.
Como o lugar singular 
de $\Sec_{r-1} C$ vive em codimensão 5 
em $\P^{2r}$, segue do Teorema de Bertini que esta curva é suave;
calculamos seu gênero no Exemplo~\ref{ex_seccional_genus} a partir da série de Hilbert de $\Sec_{r-1} C$:
\[
g(Y_r)=\frac{1}{24}(r-1)(r-2)(3r^2+11r+12).
\]
Como $Y_r$ é suave, sua classe de Segre é dada pelo inverso da classe de Chern
do normal $N_{Y_r/\P^4}$, que por sua vez pode ser calculada 
via a sequência exata (a fórmula de adjunção)
\[
0 \to TY_r \to T\P^4|_{Y_r} \to N_{Y_r/\P^4} \to 0\quad .
\]
O resultado é:
\begin{align*}
s(Y_r,\P^4) & =c(N_{Y_r/\P^4})^{-1} \\
& = \left((2-2g(Y_r))h^4+\binom{r+2}{3}h^3\right)((1+h)^5)^{-1} \\
& = \left(2-2g(Y_r)-5\binom{r+2}{3}\right)h^4+\binom{r+2}{3}h^3 \quad \in A_*\P^4.
\end{align*}
Portanto, das fórmulas em \eqref{eq-djs-segre},
\begin{align*}
d_4(\psi_r) & = r^4-4rs_3(Y_r)-s_4(Y_r) \\
& = r^4 -4r\binom{r+2}{3}-2 + 2g(Y_r)+5\binom{r+2}{3} \\
& = r^4 +(5-4r)\binom{r+2}{3}+2g(Y_r)-2\\
& = \frac{1}{12}r(r-1)(7r^2-5r-6).
\end{align*}
\end{proof}

%
%
%


\chapter{Conjectura}
\label{cap-conjecturas}
Como no capítulo anterior, seja $\phi_r\colon \P^{2r}\dashrightarrow\P^{2r}$
o mapa gradiente associado à hipersuperfície 
$\Sec_r C \subset\P^{2r}$, a $r$-secante
de uma curva
racional normal $C\subset \P^{2r}$. 
Buscamos expressões para os graus projetivos de todos estes mapas.

\smallskip
Em todo este capítulo utilizaremos a notação seguinte. 

Para $i=0,1,\dotsc,2r,$ indicamos por $d_i(\phi_r)$ o $i$-ésimo grau projetivo do mapa $\phi_r$;
em outras palavras, 
\[
d_i(\phi_r)=\deg(\phi_r|_{\P^i})
\]
é o grau topológico 
da restrição do mapa a um $\P^i$ geral. Em particular, $d_{2r}(\phi_r)=\deg(\phi_r)$.

Denotamos por $U_r=\P^{2r}\setminus \Sec_r C$ o \emph{complementar 
da hipersuperfície $r$-secante}. Para $i=0,1,\dotsc,2r$, denote por 
$c_i(U_r)$ o coeficiente de $[\P^i]$ na classe $\csm({U_r})$, ou seja,
\[
\csm(U_r) = c_0(U_r)[\P^0] + c_1(U_r)[\P^1] + \dotsb + c_{2r}(U_r)[\P^{2r}] \qquad \in A_*\P^{2r}.
\] 
Assim o grau da classe é $c_0(U_r)$, que é portanto a 
característica de Euler topológica do aberto $U_r$.

Para $k\geq 0$, denotamos por $\cat_k=\frac{1}{k+1}\binom{2k}{k}$ os números de Catalan.
Recorde nossa convenção: para inteiros $a,b$, vale $\binom{a}{b}=0$ sempre que $0<a<b$ ou $b<0$. 
Isso nos permite desconsiderar índices em somas infinitas
envolvendo binomiais, o que simplifica as fórmulas e a argumentação.

\section{Conjectura}

Começamos revisitando o Teorema~\ref{Aluffi2.1} de P. Aluffi,
que relaciona os graus projetivos do mapa gradiente $\phi_r$ e a classe $\csm(\Sec_rC)$. 
Utilizando o príncipio de inclusão-exclusão e o fato de que $\csm(\P^n)=c(T\P^n)\cap[\P^n]= (1+h)^{n+1}\in A_*\P^n$, segue da fórmula ali apresentada uma expressão para a classe de 
Schwartz-MacPherson do aberto $U_r$:
\begin{equation}
\label{eq-aluffi-paper}
\csm({U_r})= \sum_{j=0}^{2r} d_j(\phi_r)(-h)^j(1+h)^{2r-j} \qquad \in A_*\P^{2r}.
\end{equation}
Em particular, lendo o coeficiente de $[\P^0]=h^{2r}$:
\begin{equation}
\label{eq-c_0}
c_0(U_r) = \sum_{j=0}^{2r} (-1)^jd_j(\phi_r) = d_0(\phi_r) - d_1(\phi_r) + \dotsb + d_{2r}(\phi_r)
\end{equation}
e portanto a característica de Euler topológica de $U_r$ 
é a soma alternada
dos graus projetivos. O caso geral é similar: os $c_i$'s são dados por uma soma alternada
dos graus projetivos com certos pesos, que dependem da dimensão do espaço ambiente.
De fato, expandindo \eqref{eq-aluffi-paper} obtemos 
\begin{equation}
\label{eq-aluffi}
c_i(U_r) = \sum_{j=0}^{2r-i} (-1)^j \textstyle \binom{2r-j}{i} d_j(\phi_r) \qquad.
\end{equation}
Em outras palavras, para cada $r\geq 0$ fixado, temos uma relação linear
$\vec{c}=A_r\vec{d}$, onde $A_r$ é uma matriz invertível
cujas entradas, frisamos, dependem de $r$:
\begin{equation}
\label{eq-d2c}
\begin{pmatrix}
 c_0 \\c_1 \\ c_2 \\ \vdots \\ c_{2r-1} \\ c_{2r}\\
\end{pmatrix}
=
\begin{pmatrix}
 1 & -1 & \cdots  & 1 & -1 & 1 \\
 \binom{2r}{1} & -\binom{2r-1}{1} & \cdots &\binom{2}{1} & -\binom{1}{1} &  \\
 \binom{2r}{2} & -\binom{2r-1}{2} & \cdots & \binom{2}{2} &   &  \\
 \vdots & & \iddots \\
 \binom{2r}{2r-1} & - \binom{2r-1}{2r-1} \\
 \binom{2r}{2r}
\end{pmatrix}
\begin{pmatrix}
 d_0 \\d_1 \\ d_2 \\ \vdots \\ d_{2r-1} \\ d_{2r}\\
\end{pmatrix}.
\end{equation}
Reciprocamente,
\begin{equation}
 \label{ci-di}
 d_i(\phi_r) = \sum_{j=2r-i}^{2r}
 \textstyle
 (-1)^j\binom{j}{2r-i}c_j(U_r)\qquad
\end{equation}
e em particular o grau topológico do mapa gradiente é a soma 
alternada dos $c_i$'s:
\begin{equation}
\label{eq-d2r}
 d_{2r}(\phi_r) = c_0(U_r) - c_1(U_r) + \dotsb - c_{2r-1}(U_r) + c_{2r}(U_r).
\end{equation}

\bigskip

Na Proposição~\ref{prop-graus} calculamos 
explicitamente as funções $d_i(\phi_r)$ para
$i\leq 4$.
Em cada um destes casos
vimos que $d_i(\phi_r)$
é polinomial em $r$, de grau $i$.


E sobre as funções $c_i(U_r)$, o que podemos dizer? 
Vejamos o caso em que $i=0$: no Teorema~\ref{teo-Euler_Char_Sec_k}
  calculamos a característica de Euler de $\Sec_r C$ via pontos fixos 
 de uma ação de $\C^*$, obtendo $2r$. Logo:
 \[
 	c_0(U_r) = \euler(\P^{2r}\setminus\Sec_rC) = 
	\euler(\P^{2r})-\euler(\Sec_r C) = 
	2r+1-2r=1.
\]
Ou seja, a função $c_0(U_r)$ é de fato constante. 

Se $i>0$, temos fórmulas para os coeficientes
$c_i(U_r)$ quando $i\geq 2r-4$:  por \eqref{eq-aluffi} estes se reduzem a calcular $d_i(\phi_r)$ para $i\leq 4$ e logo podemos aplicar a Proposição~\ref{prop-graus}; o resultado é
\begin{equation}
\label{eq-ci's}
\begin{split}
c_{2r}(U_r)=1, \qquad
c_{2r-1}(U_r)=r, \qquad
c_{2r-2}(U_r)=r^2, \qquad
\\
c_{2r-3}(U_r)=\frac{1}{2}r^2(r-1) \qquad\text{e}\qquad
c_{2r-4}(U_r)=\frac{1}{4}r^2(r-1)^2.
\end{split}
\end{equation}
Não temos fórmulas para valores pequenos de $i$.

Entretanto quando a dimensão do espaço ambiente é pequena podemos calcular estes coeficientes efetivamente, com o auxílio de um computador. Com os recursos que dispomos, 
conseguimos 
fazê-lo apenas para $r\leq 5$; veja o código para 
o \Mdois{} no Apêndice~\ref{subsecao-r<=5}. Reproduzimos os resultados obtidos para 
as classes $\csm$ na Tabela~\ref{tab0}.
\begin{table}[ht]
\begin{center}
\small
\begin{tabular}{c||l|l|l|l|l|l|l|l|l|l|l}
$r$ & $c_0$ & $c_1$ & $c_2$ & $c_3$ & $c_4$ & $c_5$ & $c_6$ & $c_7$ & $c_8$ & $c_9$ & $c_{10}$\\ 
\hline 
\hline 
0 & 1 & & & & & & & & & & \\
\hline
1 & 1 & 1 & 1 &   &   &   &  & & & &\\ 
\hline 
2 & 1 & 2 & 4 & 2 & 1 &   & & & & & \\ 
\hline 
3 & 1 & 3 & 9 & 9 & 9 & 3 & 1 & & & &\\
\hline
4 & 1 & 4 & 16 & 24 & 36 & 24 & 16 & 4 & 1 & & \\
\hline
5 & 1 & 5 & 25 & 50 & 100 & 100 & 100 & 50 & 25 & 5 & 1\\ 
\end{tabular} 
\label{tab0}
\caption{Coeficientes $c_i(U_r)$ calculados via \Mdois.}
\end{center}
\end{table}

Analisando a tabela, somos levados a acreditar que $c_1(U_r)=r$ e $c_2(U_r)=r^2$, novamente polinômios
de graus $1$ e $2$, respectivamente. 

Mais ainda, suponha por um momento que
$c_3(U_r)$ seja polinomial em $r$, de grau $3$: como já conhecemos
seus valores para $r=0,1,2,3$ (a saber $0,0,2,9$, respectivamente)
podemos interpolar, obtendo
\begin{equation}
\label{palpite-c_3}
c_3(U_r)=\frac{1}{2}r^2(r-1)
\end{equation}
que seria então nosso palpite. Este polinômio avaliado em $r=4$ e $r=5$
resulta em $24$ e $50$, em concordância com os valores da Tabela~\ref{tab0}. 
Algo similar acontece com $c_4(U_r)$. Outro aspecto notável 
é a simetria. 

\medskip
Todas estas considerações nos levam à seguinte:

\newcommand{\tc}{c}
\newcommand{\td}{d}

\begin{conjectura}
 \label{conj-A}
Para cada $i\geq 0$ fixado tem-se:
\begin{enumerate}[{\rm(a)}]
\item A função $c_i(U_r)$ é polinomial em $r$, de grau $\leq i$.
\item A função $d_i(\phi_r)$ é polinomial em $r$, de grau $\leq i$.
\end{enumerate}
\end{conjectura}

Uma vez feita a hipótese, é natural explorar.
Que consequências decorrem da conjectura?
A resposta, em sua formulação última, é que a sua validade implica em 
\emph{fórmulas explícitas} para as funções  $c_i(U_r)$ e $d_i(\phi_r)$. Isto, concordemos, não é nada claro aqui. Os mais apressados as encontrarão
no Teorema~\ref{teo-equivalencias}. E
destas fórmulas decorrem certas propriedades qualitativas
interessantes para as classes $\csm(U_r)$ e $\csm(\Sec_r C)$;
consulte o Corolário~\ref{cor-conjectura}.

Finalmente, a Conjectura~\ref{conj-A} pode ser reformulada de uma maneira 
ainda mais simples; veja a Proposição~\ref{conj-A'}.

\section{Algoritmo}
\label{secao-algoritmo}

Mostraremos que
se estamos dispostos a assumir a validade da Conjectura~\ref{conj-A},
então os graus projetivos do mapa gradiente $\phi_r$ e a classe 
$\csm(U_r)$ podem ser efetivamente calculados para \emph{todo}
$r$.
Os ingre\-dientes-chave são:
o cálculo do grau topológico do mapa gradiente,
$d_{2r}(\phi_r)=\cat_r$, realizado no Corolário~\ref{thm3} e as relações \eqref{eq-aluffi} de Aluffi.

O resultado abaixo é de caráter numérico,
sem conexão alguma
com a geometria. Para enfatizar isso, mudamos um pouco a notação:
$\tc_i(r)$ e $\td_i(r)$ denotarão os números inteiros produzidos,
que a priori não possuem relação com os invariantes geométricos
$c_i(U_r)$ e $d_i(\phi_r)$ que desejamos determinar.

\begin{teo}[Algoritmo]
\label{algoritmo}
Existe um único par de sequências duplas $(\tc_i(r))$ e 
$(\td_i(r))$ com $i,r\geq 0$ que satisfaz as condições seguintes:
\begin{enumerate}[{\rm(a)}]
 \item $\td_i(r)=0$ para todo $i\geq 2r+1$.
 \item $\td_{2r}(r)=\cat_r$.
 \item Fixado $i$, a função $\tc_i(r)$ é polinomial em $r$, de grau $\leq i$.
 \item Fixado $i$, a função $\td_i(r)$ é polinomial em $r$, de grau $\leq i$.
 \item Valem as relações de Aluffi para $\tc_i(r)$ e $\td_i(r)$, isto é
 \begin{equation}
 \label{eq-cd}
 \tc_i(r) = \sum_{j=0}^{2r-i} (-1)^j \textstyle \binom{2r-j}{i} \td_j(r) \qquad.
 \end{equation}
\end{enumerate}
Mais ainda, estas sequências são efetivamente calculáveis.
\end{teo}
\begin{proof} 
 Olhamos para $\tc=(\tc_i(r))$ e $\td=(\td_i(r))$ como matrizes infinitas, 
 onde $r$ indexa as linhas e $i$ indexa as colunas. Por (a),
 apenas as primeiras $2r+1$ entradas da linha $r$ na matriz $\td$ 
 são possivelmente não-nulas e por (e) o mesmo vale para
 a matriz $\tc$.

 Procedemos por indução em $r$. Descrevemos a seguir como obter uma nova coluna em ambas as matrizes a partir das colunas anteriores. A coluna $r$ (e também toda a linha $r$) é gerada no passo $r$.
 
 \smallskip
 \underline{Passo 0}: para $r=0$, segue de (b) que $\td_0(0)=1$. 
 De (d) segue 
 que $\td_0(r)$ é constante e logo $\td_0(r)=1$ para todo $r$. 
 De (e) vem que $\tc_0(0)=1$ e daí segue por (c) que $\tc_0(r)=1$
 para todo $r$. Assim, em ambas as matrizes, todas as 
 entradas na primeira coluna são iguais a $1$. 
 
 \smallskip
  \underline{Passo $r$}:  neste passo 
já conhecemos todas as entradas das $r$ primeiras colunas. 
Em particular, conhecemos os valores de 
$\tc_i(r)$ e $\td_i(r)$ para $i=0, \dotsc, r-1$.
Decorre daí que também conhecemos 
$\tc_{r+1}(r),\dotsc,\tc_{2r}(r)$, pois por $\eqref{eq-cd}$
estes valores dependem apenas dos 
já conhecidos 
$\td_0(r),\dotsc,\td_{r-1}(r)$.

Assim, na linha $r$ da matriz $\tc$, 
já conhecemos todos os valores de $\tc_i(r)$ para todo $i$,
exceto quando $i=r$, na diagonal. Entretanto a relação dada em (e)
entre os $c_i$'s e $d_i$'s é invertível para $r$ fixado:
\begin{equation}
\label{eq-dc}
d_i(r) = \sum_{j=2r-i}^{2r}
\textstyle
(-1)^j\binom{j}{2r-i}c_j(r) \qquad .
\end{equation}
Em particular, para $i=2r$, 
\begin{equation}
\label{eq-cr}
\tc_0(r) - \tc_1(r) + \dotsb + (-1)^r \tc_r(r) +\dotsb - \tc_{2r-1}(r) + \tc_{2r}(r) = \td_{2r}(r)
\end{equation}
e por (b) temos $d_{2r}(r)=\cat_r$, obtendo assim $\tc_r(r)$. 
Conhecemos portanto $\tc_0(r),\dotsc,\tc_{2r}(r)$
e, como já observamos, todas as outras entradas na linha $r$ da matriz $\tc$
são nulas. Em suma, conhecemos toda a linha $r$ na matriz $\tc$.
Usando as equações \eqref{eq-dc} oriundas de (e), 
obtemos também a linha $r$ da matriz $\td$.
Finalmente, vem de (c) e (d) que $\tc_r(\cdot)$ e $\td_r(\cdot)$ são polinomiais
de grau $\leq r$; como agora conhecemos os $r+1$ valores $\tc_r(0),\dotsc,\tc_r(r)$ e $\td_r(0),\dotsc,d_r(r)$ na coluna $r$, podemos interpolar, obtendo
assim toda a coluna $r$ em ambas as matrizes, o que termina esse passo.
\end{proof}

É um algoritmo muito simples. Oferecemos uma 
implementação no Apêndice~\ref{secao-alg1}.
\smallskip

\begin{exemplo} 
\label{ex-alg}
Acompanhemos o funcionamento do algoritmo no começo do passo 3. 
Neste ponto já conhecemos as 3 primeiras linhas e colunas das matrizes 
$c_i(r)$ e $d_i(r)$. Em particular, já conhecemos
\[
c_0(r)=1, c_1(r)=r, c_2(r)=r^2 
\qquad\text{e}\qquad
d_0(r)=1, d_1(r)=r, d_2(r)=r^2.
\]
Por (b), conhecemos também $d_6(3)=\frac{1}{4}\binom{6}{3}=5$.
As entradas desconhecidas estão indicadas simbolicamente,
e, por simplicidade, omitimos 
o índice da linha.
\begin{center}
\begin{tabular}{c||l|l|l|l|l|l|l}
$r$ & $c_0$ & $c_1$ & $c_2$ & $c_3$ & $c_4$ & $c_5$ & $c_6$ \\ 
\hline 
\hline
0 & 1 &   &   &   &   &   &   \\ 
\hline 
1 & 1 & 1 & 1 &   &   &   &   \\ 
\hline 
2 & 1 & 2 & 4 & 2 & 1 &   &   \\ 
\hline 
3 & 1 & 3 & 9 & \verm{$c_3$} & \verm{$c_4$} & \verm{$c_5$} & \verm{$c_6$} \\ 
\end{tabular} 
\qquad\quad
\begin{tabular}{c||l|l|l|l|l|l|l}
$r$ & $d_0$ & $d_1$ & $d_2$ & $d_3$ & $d_4$ & $d_5$ & $d_6$ \\ 
\hline 
\hline
0 & 1 &   &   &   &   &   &   \\ 
\hline 
1 & 1 & 1 & 1 &   &   &   &   \\ 
\hline 
2 & 1 & 2 & 4 & 4 & 2 &   &   \\ 
\hline 
3 & 1 & 3 & 9 & \verm{$d_3$} & \verm{$d_4$} & \verm{$d_5$} & 5 \\ 
\end{tabular} 
\end{center}
Usando \eqref{eq-cd}, obtemos $c_4=9, c_5=3$ e $c_6=1$
pois estes só dependem de $d_0,d_1$ e $d_2$; daí, usando 
\eqref{eq-cr}, vem que 
\[
 1-3+9-c_3+9-3+1 = 5
\]
e portanto $c_3=9$. Aplicando \eqref{eq-dc} obtemos
$d_3=17$, $d_4=21$, $d_5=15$. Finalmente, como
a coluna $3$ é polinomial de grau $\leq 3$ e
conhecemos $4$ valores, interpolamos 
e obtemos 
\[
c_3(r)=\frac{1}{2}r^2(r-1)
\quad\text{ e } \quad
d_3(r)=\frac{1}{6}r(r-1)(5r+2).
\]
Incidentalmente, observe que
$c_3(r)$ coincide com o nosso palpite para $c_3(U_r)$ em \eqref{palpite-c_3}
e que $d_3(r)$ coincide com
$d_3(\phi_r)$ calculado na Proposição~\ref{prop-graus}.
Temos agora todas
as entradas das 4 primeiras linhas e colunas:
\begin{center}
\small
\begin{tabular}{c||l|l|l|l|l|l|l|l|l}
$r$ & $c_0$ & $c_1$ & $c_2$ & $c_3$ & $c_4$ & $c_5$ & $c_6$ & $c_7$ & $c_8$ \\ 
\hline 
\hline
0 & 1 &   &   &   &   &   &  & & \\ 
\hline 
1 & 1 & 1 & 1 &   &   &   &  & & \\ 
\hline 
2 & 1 & 2 & 4 & 2 & 1 &   & & &  \\ 
\hline 
3 & 1 & 3 & 9 & 9 & 9 & 3 & 1 & & \\
\hline
4 & 1 & 4 & 16 & 24 & \verm{$c_4$} & \verm{$c_5$} & \verm{$c_6$} & \verm{$c_7$} & \verm{$c_8$} \\  
\end{tabular} 
\quad
\begin{tabular}{c||l|l|l|l|l|l|l|l|l}
$r$ & $d_0$ & $d_1$ & $d_2$ & $d_3$ & $d_4$ & $d_5$ & $d_6$ & $d_7$ & $d_8$ \\ 
\hline 
\hline
0 & 1 &   &   &   &   &   & & & \\ 
\hline 
1 & 1 & 1 & 1 &   &   &   & & & \\ 
\hline 
2 & 1 & 2 & 4 & 2 & 1 &   & &  & \\ 
\hline 
3 & 1 & 3 & 9 & 17 & 21 & 15 & 5 & & \\ 
\hline
4 & 1 & 4 & 16 & 44 & \verm{$d_4$} & \verm{$d_5$} & \verm{$d_6$} & \verm{$d_7$} & 14 \\ 
\end{tabular} 
\end{center}
e estamos prontos para prosseguir para o próximo passo.
\hfill{$\Box$}
\end{exemplo}

Evidentemente as condições (a)--(e) do Teorema~\ref{algoritmo} foram
designadas tendo em vista o cálculo efetivo 
dos invariantes $c_i(U_r)$ e $d_i(\phi_r)$
oriundos das secantes $\Sec_r C\subset\P^{2r}$.
Deixamos registrado que 
se nossa conjectura vale, então o 
algoritmo no Teorema~\ref{algoritmo} de fato 
produz esses invariantes.

\begin{prop}
\label{prop-num=geom}
Se vale a Conjectura~\ref{conj-A}, então $c_i(U_r)=c_i(r)$
e $d_i(\phi_r)=d_i(r)$ 
para todo $i,r\geq 0$.
\end{prop} 
\begin{proof}
Pela unicidade, 
basta mostrar que as sequências $(c_i(U_r))$ e 
$(d_i(\phi_r))$ satisfazem as hipóteses (a)--(e) do 
Teorema~\ref{algoritmo}, o que faremos a seguir.

Como $\Sec_r C\subset\P^{2r}$, 
é claro que $d_i(\phi_r)=0$
sempre que $i>2r$, cumprindo (a); 
a condição (b) vem do Corolário~\ref{thm3}; 
assumir a Conjectura~\ref{conj-A} é exatamente afirmar que (c) e (d) valem;
por fim, a condição (e) vem das relações~\eqref{eq-aluffi} de Aluffi
entre os graus projetivos do mapa gradiente
e os coeficientes da classe $\csm$ do complementar 
da hipersuperfície.
\end{proof}

\begin{rmk}
Nas condições (c) e (d) do Teorema~\ref{algoritmo}
exigimos que $c_i(r)$ e $d_i(r)$ sejam polinomiais em $r$ para cada $i$ fixado.
E por (e) estas sequências estão relacionadas por uma matriz invertível. É natural perguntar: não bastaria
exigir que apenas uma delas seja polinomial
e deduzir que a outra também o seja como
consequência de (e)?

A resposta é \underline{não}.
Para um exemplo, tome $d_i(r)=r^i$. Fixado $i$,
a sequência $c_i(r)$ obtida via as equações \eqref{eq-cd} é {exponencial} em
$r$, já para $i=0$:
\begin{equation*}
c_0(r) = \sum_{j=0}^{2r}(-1)^j\binom{2r-j}{0}r^j = r^{2r}-r^{2r-1}+\dotsb-r+1.
\end{equation*}
\end{rmk}
\vspace{-1.1cm}
\hfill{$\Box$}
\ \\

Executando o algoritmo calculamos tantos valores para as sequências
quanto desejarmos. Os números produzidos, \emph{per si}, 
não dizem muita coisa. 
Em outras palavras, gostaríamos de responder: existem \emph{fórmulas} 
que codificam
os resultados do algoritmo?
A resposta é afirmativa, como veremos
no Teorema~\ref{teo-formulas}. O caminho não é curto, 
mas tem surpresas agradáveis.

\section{Funções geradoras}
\label{secao-funcaogeradora}

\begin{quote}
\emph{In combinatorics it is a standard procedure to associate with
a sequence of numbers $a_0, a_1, a_2, \dots$ the generating function $\sum a_it^i$.}

\hfill{
\begin{small}
H. Matsumura, Commutative ring theory.
\end{small}
}
\end{quote}

\begin{quote}
\emph{A generating function is a device somewhat similar to a bag. Instead of carrying many little objects detachedly, which could be embarrassing, we put them all in a bag, and then we have only one object to carry, the bag.}

\hfill{
\begin{small}
George P\'{o}lya, Mathematics and plausible reasoning.
\end{small}
}
\end{quote}
\medskip

Em toda esta seção denotaremos por 
$(c_i(r))$ e $(d_i(r))$ as sequências produzidas pelo algoritmo descrito no Teorema~\ref{algoritmo}. 

Os resultados para $r\leq 7$ encontram-se
no Apêndice, Seção~\ref{algoritmo-resultados}. Reproduzimos
os valores de $d_i(r)$ na Tabela~\ref{tab1}. Cabe
notar que estes resultados coincidem com os graus projetivos
$d_i(\phi_r)$ que fomos capazes de calcular 
diretamente via \Mdois{}, a saber $r\leq 5$,
ou seja, até $\P^{10}$.
Para $r\geq 6$ os valores apresentados são conjecturais.
Nosso objetivo é encontrar fórmulas para $c_i(r)$ e $d_i(r)$.
\smallskip

\begin{table}[ht]
\begin{center}
\begin{footnotesize}
\begin{tabular}{l||l|l|l|l|l|l|l|l|l|l|l|l|l|l|l}
$r$ & $d_0$ & $d_1$ & $d_2$ & $d_3$ & $d_4$ & $d_5$ & $d_6$ & $d_7$ & $d_8$ & $d_9$ & $d_{10}$ & $d_{11}$ & $d_{12}$  & $d_{13}$ & $d_{14}$  \\
\hline
\hline
0  & 1 &&&&&&&&&&&& && \\
1  & 1 & 1 & 1 &&&&&&&&&& &&\\
2  & 1 & 2 & 4 & 4 & 2 &&&&&&&& &&\\
3  & 1 & 3 & 9 & 17 & 21 & 15 & 5  &&&&&& &&\\
4  & 1 & 4 & 16 & 44 & 86 & 116 & 104 & 56 & 14 &&&& &&\\
5  & 1 & 5 & 25 & 90 & 240 & 472 & 680 & 700 & 490 & 210 & 42 && &&\\
\hline
6  & \verm{1} & \verm{6} & \verm{36} & \verm{160} & \verm{540} & \verm{1392} & \verm{2752} & \verm{4152} & \verm{4710} & \verm{3900} & \verm{2232} & \verm{792} & \verm{132} &&\\
7 & \verm{ 1} & \verm{ 7} & \verm{ 49} & \verm{ 259} & \verm{ 1057} & \verm{ 3367} & \verm{ 8449} & \verm{ 16753} & \verm{ 26173} & \verm{ 31899} & \verm{29757} & \verm{ 20559} & \verm{ 9933} & \verm{ 3003} & \verm{ 429} \\
\end{tabular}
\end{footnotesize}
\end{center}
\caption{Graus projetivos (via Conjectura~\ref{conj-A} para $r\geq6$) do mapa $\phi_r\colon\P^{2r}\dashrightarrow\P^{2r}$.}
\label{tab1}
\end{table}%

Nossa estratégia consiste em buscar a função geradora para cada
função $d_i(r)$, ou seja, para cada
uma das colunas da Tabela~\ref{tab1}. Concretamente, fixado $i\geq 0$, consideramos a
série formal
\[
p_i(x) :=\sum_{r=0}^{\infty} d_i(r)x^r = d_i(0) + d_i(1)x + d_i(2)x^2 + \dotsb
\qquad \in \Z[[x]].
\]
Por exemplo, $p_3(x) = 4x^2+ 17x^3+44x^4+90x^5+160x^6+\dotsb$. Geometricamente,
a interpretação é que se a Conjectura~\ref{conj-A} é válida, então
$p_i(x)$ enumera os $i$-ésimos graus projetivos de $\phi_r$, obtidos tomando a restrição 
a um $\P^i$ geral em $\P^{2r}$ com $i$ fixado, deixando $r$ variar.

\smallskip

Quando $d_i(r)$ é polinomial em $r$ de grau $\leq i$, a sua função geradora 
admite uma representação por uma função racional da forma $P(x)/(1-x)^{i+1}$,
onde $P(x)$ é um polinômio de grau $\leq i$. Isto é padrão e apresentamos
uma prova sucinta no Lema~\ref{lema-A1}. A partir daí fica fácil obter fórmulas
para as funções geradoras.

\begin{exemplo} 
\label{ex-A1}
Considere o polinômio $d_3(r)=\frac{1}{6}r(r-1)(5r+2)$ obtido no 
Exemplo~\ref{ex-alg}. Sua função geradora é 
$p_3(x)=4x^2+17x^3+44x^4+\dotsb$, que pode então ser escrita como 
\[
p_3(x)=P(x)/(1-x)^4
\]
com $P(x)=a_0+a_1x+a_2x^2+a_3x^3$. Expandindo a igualdade $(1-x)^4 p_3(x) = P(x)$:
\[
(1-4x+6x^2-4x^3+x^4)(4x^2+17x^3+\dotsb) = a_0+a_1x+a_2x^2+a_3x^3
\]
deduzimos $a_0=a_1=0, a_2=4$ e $a_3=1$. Portanto,
\[
p_3(x) = \frac{x^2}{(1-x)^4}\cdot({x+4})
\]
Esta é a função racional que aparece na Tabela~\ref{tab2} abaixo.
\hfill{$\Box$}
\end{exemplo}

\newcommand{\tq}{\tilde{q}}

Listamos algumas das funções polinomiais $d_i(r)$ 
produzidas pelo Algoritmo~\ref{algoritmo} na primeira coluna
da Tabela~\ref{tab2}. Na segunda coluna
aparecem as funções racionais asso\-cia\-das.
A inusitada fatoração se justifica: primeiro, das considerações acima, 
segue-se que $(1-x)^{i+1}p_i(x)$ é um polinômio, de grau $\leq i$; mais ainda,
$x^{[(i+1)/2]}$ divide este polinômio, uma vez que $d_i(r)=0$ sempre que $i\geq 2r+1$,
pela condição (a) do Teorema~\ref{algoritmo}. Logo, para todo $i\geq 0$, temos que
\begin{equation}
\label{eq-qn}
q_i(x) := \frac{(1-x)^{i+1}p_i(x)}{x^{[(i+1)/2]}} 
\end{equation}
é um polinômio. Seu grau é $\leq [i/2]$, uma vez que $i=[i/2]+[(i+1)/2]$. 
Na Tabela~\ref{tab2},
estes são os polinômios que aparecem à direita. Por exemplo, 
$q_4(x)=x^2+11x+2$.

\newcommand{\vvv}[3]{ \frac{{#1}}{(1-x)^{#2}}\cdot{#3}  }
\begin{table}[thp]
\begin{center}
\begin{tabular}{l||l|l}
$i$ & $d_i(r)$ & $p_i(x)$ \\
\hline
\hline
0 & $1$
             &   $ \vvv{1}{ }   {\verm{1}} $\\
\hline
1 & $r$ 
             &   $\vvv{x}{2}   {\verm{1}} $\\
\hline
2 & $r^2$ 
             &   $\vvv{x}{3}   {(\verm{1}x+\verm{1})}$\\
\hline
3 & $ \frac{1}{6} r (r-1) (5 r+2) $
             &   $\vvv{x^2}{4} {(\verm{1}x+\verm{4})}$\\
\hline
4 & $\frac{1}{12} r (r-1) (7 r^{2}-5 r-6) $
             &    $\textstyle\vvv{x^2}{5} {(\verm{1}x^2+\verm{11}x+\verm{2})}$\\
\hline
5 & $\frac{1}{60}  r (r-1) (r-2) (21 r^{2}-7 r-18) $
             &   $\vvv{x^3}{6} {(\verm{1}x^2+\verm{26}x+\verm{15})}$\\
\hline
6 & $\frac{1}{60}  r (r-1) (r-2) (11 r^{3}-27 r^{2}-8 r+20) $
             &   $\vvv{x^3}{7} {(\verm{1}x^3+\verm{57}x^2+\verm{69}x+\verm{5})}$\\
\hline
\end{tabular}
\end{center}
\caption{Funções geradoras $p_i(x)=\sum_r d_i(r)x^r$ e os polinômios $q_i(x)$.}
\label{tab2}
\end{table}%

\medskip
Um cidadão sensato perguntaria pelo motivo de toda essa ginástica.

\subsection{Conexões inesperadas}
\label{subsecao-inesperadas}

Calculamos os polinômios $q_n(x)$ para diversos valores de $n$.
Observamos então que os seus coeficientes formam uma sequência
bem conhecida em combinatória, a saber \emph{o número de caminhos de Dyck de semicomprimento
$n$ com $k$ subidas longas}, a qual denotaremos por $T(n,k)$. 
Isto não era esperado.
Você está convidado a comparar os primeiros termos desta sequência,
exibidas na Tabela~\ref{tab-T}, com os coeficientes dos polinômios 
$q_n(x)$ que obtivemos na Tabela~\ref{tab2}. A coincidência persiste para todo $n\leq 20$.

Para outras informações sobre a sequência $T(n,k)$
(referências,
outras identidades
e interpretações do ponto de vista da combinatória)
recomendamos \cite[A091156]{OEIS},
no famoso \emph{On-Line Encyclopaedia of Integer Sequences}.
Indicamos também \cite{STT}.
Apresentemos uma fórmula: tem-se $T(n,0)=1$ e
\begin{equation}
\label{eq-Tnk}
 T(n,k)=\frac{1}{n+1} \binom{n+1}{k}
  \sum_{j=2k}^{n} \binom{j-k-1}{k-1} \binom{n+1-k}{n-j}
\end{equation}
para $0\leq n$ e $0< k\leq [n/2]$. 

\begin{table}[btp]
\begin{center}
\begin{tabular}{c|lllll}
$n\diagdown{}k$ & 0 & 1 & 2 & 3 \\
\hline
\hline
{0} & \verm{1}\\
{1} & \verm{1}\\
{2} & \verm{1} & \verm{1} \\
{3} & \verm{1} & \verm{4}\\
{4} & \verm{1} & \verm{11} & \verm{2}\\
{5} & \verm{1} & \verm{26} & \verm{15}\\
{6} & \verm{1} & \verm{57} & \verm{69} & \verm{5}\\
\end{tabular}
\end{center}
\caption{Termos iniciais da sequência $T(n,k)$.}
\label{tab-T}
\end{table}%

%


\medskip
Há uma outra conexão inesperada, mais interessante. K. Gedeon em
um artigo recente (\cite{Gedeon}) mostrou
que os termos da sequência $T(n,k)$ também aparecem como
coeficientes
de outros polinômios, a saber os polinômios de \emph{Kazhdan-Lusztig} de 
(matróides associados a) grafos completos tripartidos $K_{1,1,n}$%
.
Traduzindo para o nosso contexto, se $P_{M_n}(x)$ são os polinômios obtidos
por Gedeon, estes \emph{deveriam} coincidir com os nossos
polinômios $q_n(x)$, desde que tomemos o cuidado de inverter a ordem dos coeficientes;
ou seja,
deveríamos ter
\begin{equation}
\label{eq-gedeon}
P_{M_n}(x) = x^{[n/2]}q_n(1/x).
\end{equation}
Enfatizamos que até aqui esta igualdade foi verificada 
apenas para pequenos valores de $n$, mas é 
tentador imaginar que isso ocorra sempre.
Seguirá do Teorema~\ref{teo-formulas} que isso de fato acontece!
Tal coincidência pede por uma explicação conceitual
(de algum modo há uma
conexão entre secantes de curvas racionais normais e grafos 
completos tripartidos) mas no momento
não temos nenhuma para oferecer.

Em resumo, nosso objetivo é  \emph{demonstrar} que 
os coefi\-cien\-tes dos polinômios $q_n(x)$ 
em \eqref{eq-qn}, produzidos pelo algoritmo do
Teorema~\ref{algoritmo}, são dados pela sequência
$T(n,k)$ para todo $n\geq 0$. 
 
\subsection{Função geradora para a sequência $(d_i(r))$}

Reformulamos o problema de maneira adequada.
Consideramos a função geradora, agora em duas variáveis,
\begin{equation}
\label{eq-defg}
 g(x,y) := \sum_{i,r\geq 0} d_i(r)x^ry^i = \sum_{i\geq 0} p_i(x)y^i
 \qquad \in \Z[[x,y]]
\end{equation}
para a sequência $(d_i(r))$ produzida pelo algoritmo do Teorema~\ref{algoritmo}.

Nossa estratégia: uma fórmula para a função geradora
da sequência $T(n,k)$ é conhecida;
somos levados a uma expressão-candidata para $g(x,y)$; 
mostramos
que nosso candidato funciona.
\medskip

Mãos à obra. Seja
\[
w(t,u) := \sum_{n=0}^\infty \sum_{k=0}^{[n/2]} T(n,k)t^k u^n.
\]
a função geradora da sequência $T(n,k)$.
Então de \cite[p.\thinspace{}9]{Wang} (veja também \cite[Section~1]{STT})
\begin{align}
\label{eq-Wang}
w(t,u)  & =   \frac{1-\sqrt{1-4u(1-u+ut)}} {2u(1-u+ut)}\\
\medskip\tag*{}
& =  1 + u + (1+t)u^2 + (1+4t)u^3 + (1+11t+2t^2)u^4 +  (1+26t+15t^2)u^5 + \dotsb 
\end{align}
Repare que os coeficientes dos $u$'s 
são os mesmos que aparecem nos numeradores do lado direto da Tabela~\ref{tab2}, exceto pelo grau e pela ordem,
que está invertida.
Precisamos portanto ajustá-los.
Mediante a mudança de variáveis
\(
t\mapsto 1/x, \   u \mapsto xy
\)
obtemos
\begin{align*}
 w(1/x,xy) &  =  \frac{1-\sqrt{1-4xy(1+y(1-x))}} {2xy(1+y(1-x))} \\
\medskip
& = 
1 + xy + x(x+1)y^2   +  x^2(x+4)y^3  + x^2 (x^2+11x+2)y^4 \\
& \hspace{4.15cm} + x^3(x^2+26x+15)y^5 + \dotsb.
\end{align*}
Perfazendo $y\mapsto y/(1-x)$ e multiplicando por $1/(1-x)$ obtemos
\begin{align*}
\frac{1}{1-x}\;w\Big(\frac{1}{x},\frac{xy}{1-x}\Big)
 = 
\frac{1}{1-x}  + \frac{x}{(1-x)^2}y  + \frac{x(x+1)}{(1-x)^3}y^2 & + \frac{x^2(x+4)}{(1-x)^4}y^3 \\
& +\frac{x^2(x^2+11x+2)}{(1-x)^5} y^4 +\dotsb 
\end{align*}
ou seja, os coeficientes de $y^i$ são as expressões do lado direito da Tabela~\ref{tab2}.
Em outras palavras,
\emph{conjecturalmente},
\[
\frac{1}{1-x}\;w\Big(\frac{1}{x},\frac{xy}{1-x}\Big)
=
\sum_{i=0}^\infty p_i(x)y^i = 
g(x,y)
 .
\]
Substituindo em \eqref{eq-Wang} e simplificando,
o resumo do que fizemos é: demonstrar que 
os coefi\-cien\-tes dos polinômios $q_n(x)$ 
em \eqref{eq-qn} são dados
pela sequência
$T(n,k)$ é \emph{equivalente} a demonstrar que
\begin{equation}
\label{gf2}
g(x,y)  = \frac{1-\sqrt{\displaystyle1-\frac{4xy(1+y)}{1-x}}}  {2xy(1+y)}\quad .
\end{equation}
Expandimos a expressão \eqref{gf2} em série
de potências. Não é difícil. Começamos com a bem
conhecida expansão da função geradora dos números de 
Catalan 
\[
\frac{1}{2z}(1-\sqrt{1-4z}) = 
\sum _{k=0}^{\infty }\cat_k z^{k} \quad .
\]
Substituindo $z\mapsto xy(1+y)/(1-x)$ e  dividindo por
$1-x$, obtemos \eqref{gf2}; portanto demonstrar~\eqref{gf2}
é equivalente a mostrar que
\[
g(x,y) = \sum _{k=0}^{\infty } \frac{\cat_k x^k y^k (1+y)^k}{(1-x)^{k+1}}  \quad .
\]
Das identidades
\[
\frac{1}{(1-x)^{k+1}} = \sum_{r=k}^{\infty} \binom{r}{k} x^{r-k} 
\qquad \text{e}\qquad
(1+y)^k=\sum_{\ell=0}^k \binom{k}{\ell} y^\ell
\]
o que queremos provar é
\begin{equation*}
 \label{gf3}
g(x,y) = \sum_{k=0}^\infty\sum_{r=k}^\infty\sum_{\ell=0}^{k} 
\binom{k}{\ell}\binom{r}{k} \cat_k \;x^{r} y^{k+\ell}  \quad .
\end{equation*}
Recorde que $d_j(r)$ é o coeficiente de $x^ry^j$ em $g(x,y)$
e que convencionamos $\binom{a}{b}=0$ sempre que $b<0$ ou $a<b$.
Lendo o coeficiente correspondente, (isto é, fazendo $k+\ell=j$), 
concluímos que \eqref{gf2} é equivalente a
\begin{equation}
\label{gf4}
d_j(r) = \sum_{k=0}^{j} \binom{k}{j-k}\binom{r}{k} \cat_k
\end{equation}
o que demonstraremos 
no Teorema~\ref{teo-formulas} que de fato vale! Entretanto
ainda temos mais uma etapa a percorrer.

\subsection{Função geradora para a sequência $(c_i(r))$}

Consideramos agora a função geradora 
\begin{equation*}
f(x,y) := \sum_{i,r\geq0} c_i(r)x^ry^i
 \qquad \in \Z[[x,y]]
\end{equation*}
associada à sequência $(c_i(r))$ produzida pelo algoritmo no 
Teorema~\ref{algoritmo}.
Como as sequências $(c_i(r))$ e $(d_i(r))$ estão conectadas
pelas relações lineares \eqref{eq-cd}, que são invertíveis, 
podemos obter a função geradora $f(x,y)$ a partir de $g(x,y)$
e vice-versa. A manipulação é elementar:
\begin{align*}
f(x,y)  &= \sum_r \sum_i c_i(r)x^ry^i \\
&=\sum_r\sum_i\sum_j (-1)^j \textstyle\binom{2r-j}{i}d_j(r)x^ry^i \\
&= \sum_r \sum_j (-1)^jd_j(r)\left[ \sum_i \textstyle\binom{2r-j}{i}y^i\right]x^r \\
&= \sum_r \sum_j d_j(r) (-1)^j(1+y)^{2r-j} x^r.
\end{align*}
Faça $u=-1-y$, $X=u^2x$ e $Y=u^{-1}$. Então:
\begin{align*}
f(x,y)  &= \sum_r \sum_j d_j(r) u^{2r-j} x^r \\
&= \sum_r \sum_j d_j(r) (u^2x)^ru^{-j} \\
&= \sum_r \sum_j d_j(r) X^rY^j \\
&= g(X,Y)
\end{align*}
Assim a relação procurada entre as funções geradoras é 
\begin{equation}
\label{eq-fg}
f(x,y)=g(x(1+y)^2,-(1+y)^{-1}).
\end{equation}
Havíamos obtido uma expressão conjectural para $g(x,y)$ 
em \eqref{gf2}; perfazendo a substituição em \eqref{eq-fg} e simplificando,
vem que a fórmula para a função geradora $g(x,y)$ em \eqref{gf2} é válida
\emph{se, e somente se,} 
\begin{equation}
\label{eq-f1}
f(x,y) = 
\frac{1-\sqrt{1-\frac{-4xy}{1-x(1+y)^2}}}{-2xy}\ \ .
\end{equation}
Recorde que $c_i(r)$ é o coeficiente de $x^ry^i$ em $f(x,y)$.
Expandindo como na passagem de \eqref{gf2} para
\eqref{gf4}, deduzimos que assumir a igualdade 
em \eqref{eq-f1} é equivalente a afirmar
\begin{equation}
\label{eq-f2}
c_i(r) =
\sum_{k=0}^{r} (-1)^k\binom{2r-2k}{i-k} \binom{r}{k}\cat_k.
\end{equation}

\begin{obs}
\label{obs47}
Para uso futuro, destacamos uma consequência importante do malabarismo
que acabamos de realizar: se $d_i(r)=\sum_{k=0}^{i} \binom{k}{i-k}\binom{r}{k} \cat_k
$ (ou seja, se vale \eqref{gf4}), então decorre das
relações \eqref{eq-cd} que $c_i(r)$ é dado pela
fórmula em \eqref{eq-f2}.
\hfill{$\Box$}
\end{obs}

\subsection{Resumo}
Listamos as fórmulas, todas equivalentes entre si, que obtivemos até aqui.
\begin{prop} 
\label{prop-resumo}
Mantenha a notação desta seção. São equivalentes:
\begin{enumerate}[{\rm(i)}]
\item Para todo $n\geq 0$, os coeficientes dos polinômios $q_n(x)$ 
são dados pelos termos da sequência $T(n,k)$.
\item 
$g(x,y) =\Big(1-\sqrt{\displaystyle1-\frac{4xy(1+y)}{1-x}}\Big)  (2xy(1+y))^{-1}$.
\item
 $d_i(r) = \sum_{k=0}^{i} \binom{k}{i-k}\binom{r}{k} \cat_k$.
 \item
$f(x,y) = 
\Big({1-\sqrt{1-\frac{-4xy}{1-x(1+y)^2}}}\Big)(-2xy)^{-1}$.
\item 
$c_i(r)= \sum_{k=0}^i (-1)^k\binom{2r-2k}{i-k} \binom{r}{k}\cat_k$.\ \hfill{$\Box$}
\end{enumerate}
\end{prop}


\section{Fórmulas}

Estamos prontos para a conclusão desta aventura. Vamos provar que
as fórmulas apresentadas na Proposição~\ref{prop-resumo} são válidas.
De fato, visando explicitar uma consequência importante
da Conjectura~\ref{conj-A}, a saber, a positividade 
dos coeficientes das classes de Schwartz-MacPherson 
$\csm(\Sec_r C) \in A_*\P^{2r}$ (Corolário~\ref{cor-conjectura}),
faremos um esforço extra e apresentaremos uma fórmula ainda mais
simples para as funções $c_i(r)$.

Recorde nossa
notação $\cat_k = \frac{1}{k+1}\binom{2k}{k}$ 
para os números de Catalan. O teorema a seguir é o resultado principal 
deste capítulo.

\begin{teo}
\label{teo-formulas}
As sequências $(c_i(r))$ e $(d_i(r))$ produzidas
pelo algoritmo no Teorema~\ref{algoritmo} são
dadas pelas fórmulas
\begin{equation}\label{eq-formula1}
c_i(r)=\binom{r}{\left[\frac{i}{2}\right]}
\binom{r}{\left[\frac{i+1}{2}\right]}
\end{equation}
e
\begin{equation}\label{eq-formula2}
 d_i(r) = \sum_{k=0}^{i} \binom{k}{i-k}\binom{r}{k}\cat_k.
\end{equation}
\end{teo}
\begin{proof}
Tendo em vista a unicidade, 
nos é suficiente mostrar que 
estas sequências
satisfazem as condições (a)-(e) do Teorema~\ref{algoritmo}.
Para (a), observe que:
\begin{itemize}
\item se $k>r$, então $\binom{r}{k}=0$; e
\item se $i\geq 2r+1$ e $0\leq k\leq r$, então $\binom{k}{i-k}=0$.
\end{itemize}
Portanto $d_i(r)=0$ sempre que $i\geq 2r+1$, como desejado. 
E para $i=2r$, um argumento análogo nos fornece
$d_{2r}(r)=\cat_r$, como exigido em (b).
É evidente que ambas as funções em \eqref{eq-formula1} e \eqref{eq-formula2}
são polinomiais em $r$, de grau exatamente $i$
(para $c_i(r)$, repare que $[i/2]+[(i+1)/2]=i$), 
contemplando as condições (c) e (d).

Resta a condição (e), a mais delicada: 
provar que as fórmulas em \eqref{eq-formula1} e \eqref{eq-formula2}
estão conectadas pelas
relações~\eqref{eq-cd}.
Da nossa oportuna Observação~\ref{obs47} 
vem que devemos então provar que 
%
%
\[
\sum_{k=0}^{r} (-1)^k \binom{2r-2k}{i-k} \binom{r}{k} \cat_k
=
\binomci.
\] 
Isto não é fácil, requer bastante trabalho. 
Nossa ideia é verificar que as funções geradoras
associadas são iguais, o que demonstramos no Lema~\ref{lema-A3}.
\end{proof}

%

\subsection{Relações recursivas}
\label{secao-recursivas}

\begin{prop}
 A função geradora $g(x,y)=\sum_{i,r\geq0} d_i(r)x^ry^i$ satisfaz 
\begin{equation}
\label{gf-rec}
	xy(1+y)g(x,y)^2 -g(x,y) + \frac{1}{1-x}=0 \qquad.
\end{equation}
Valem as relações recursivas
\begin{equation}
 \label{gp-rec}
 d_i(r) = 
 \sum_{\substack{a+u=i-1 \\ b+v=r-1}} d_a(b)d_u(v) + 
 \sum_{\substack{a+u=i-2 \\ b+v=r-1}} d_a(b)d_u(v)
 \qquad\qquad
 (i,r\geq 1, \ a,b,u,v\geq 0)
\end{equation}
com condições iniciais $d_0(r)=1$ para $r\geq 0$ e $d_i(0)=0$ para $i\geq 1$. 
\end{prop}
\begin{proof}
A equação \eqref{gf-rec} segue diretamente de \eqref{gf2} (Bhaskara\dots). Para as relações 
recursivas, basta substituir $g(x,y)=\sum d_i(r)x^ry^i$ em \eqref{gf-rec}, expandir e comparar os respectivos coeficientes.
\end{proof}

Repare que as relações \eqref{gp-rec}
e as condições iniciais indicadas determinam 
completamente a sequência $(d_i(r))$. Por exemplo,
\begin{align*}
d_1(r) & = \sum_{b+v=r-1} d_0(b)d_0(v) = \sum_{j=0}^{r-1} 1 =  r, \\
d_2(r) & = 2\sum_{b+v=r-1}d_0(b)d_1(v) + \sum_{b+v=r-1}d_0(b)d_0(v) = 2\sum_{j=0}^{r-1} j\  + \ r = r^2, \\ 
d_3(2) & =  
2\cdot \big( d_0(0)d_2(1) + d_1(0)d_1(1) +  d_0(1)d_2(0) +  d_0(0)d_1(1) +d_0(1)d_1(0)   \big) \\
 & =  2\cdot (1\cdot 1 \qquad \ \ \,
 + 0\cdot 1 \qquad \ \ 
 + 1\cdot 0 \qquad \ \ 
 + 1\cdot 1 \qquad \ \ 
 + 1\cdot 0) \\
 & = 4  
\end{align*}
e assim por diante.


\section{Conclusão}

Como vimos no decorrer das construções deste capítulo,
a Conjectura~\ref{conj-A} pode ser reformulada de várias maneiras equivalentes.
Para referência colecionamos todas elas a seguir, em uma lista
que resume todo o trabalho realizado neste capítulo.

\begin{teo} 
\label{teo-equivalencias}
Sejam
$d_i(\phi_r)$ o $i$-ésimo grau projetivo do mapa gradiente $\phi_r\colon\P^{2r}\dashrightarrow\P^{2r}$ e $c_i(U_r)$ o coefi\-ciente de $[\P^i]$ na 
classe de Schwartz-MacPherson do aberto $U_r=\P^{2r}\setminus \Sec_r C$,
como descritos no início deste capítulo. São equivalentes:
\begin{enumerate}[{\rm(a)}]
\item Vale a Conjectura~\ref{conj-A}: fixado $i\geq0$,
as funções 
 $c_i(U_r)$ e $d_i(\phi_r)$ são polinomiais em $r$, de grau $\leq i$.
\item Vale a expressão \eqref{eq-formula1} no Teorema~\ref{teo-formulas}:
\[
c_i(U_r)=\binom{r}{\left[\frac{i}{2}\right]}\binom{r}{\left[\frac{i+1}{2}\right]}.
\]
\item Vale a expressão \eqref{eq-formula2} no Teorema~\ref{teo-formulas}:
\[
d_i(\phi_r) = \sum_{k=0}^{i} \frac{1}{k+1}\binom{k}{i-k}\binom{r}{k}\binom{2k}{k}.
\]
\item A função geradora dos $c_i(U_r)$ é:
\[
\sum_{i,r\geq 0} c_i(U_r)x^ry^i =
\frac{-1+\sqrt{\textstyle{}1+\frac{4xy}{1-x(1+y)^2}}} {2xy}.
\]
\item A função geradora dos $d_i(\phi_r)$ é:
\[
\sum_{i,r\geq 0} d_i(\phi_r)x^ry^i =
\frac{1-\sqrt{\textstyle{}1-\frac{4xy(1+y)}{1-x}}} {2xy(1+y)}.
\]
\item Para todo $n\geq 0$, os coeficientes dos polinômios $q_n(x)$ em \eqref{eq-qn}
são dados pelos termos da sequência $T(n,k)$ em~\eqref{eq-Tnk}, ou seja,
\begin{equation*}
\sum_{r=0}^{\infty} d_n(\phi_r)x^r = 
\frac{1}{(1-x)^{n+1}}\cdot
\sum_{k=0}^{[n/2]} T(n,k)x^{n-k}.
\end{equation*}
Em particular, vale
a identidade
\(
P_{M_n}(x) = x^{[n/2]}q_n(1/x),
\)
onde $P_{M_n}(x)$ são os polinômios
de Kazhdan-Lusztig de 
 grafos completos tripartidos $K(1,1,n)$ obtidos em \cite{Gedeon}.
 
\end{enumerate}
\end{teo}
\begin{proof}
Consequência direta da Proposição~\ref{prop-num=geom}, das fórmulas apresentadas na Proposição~\ref{prop-resumo} e do Teorema~\ref{teo-formulas}.
\end{proof}

Para variedades determinantais genéricas, uma
conjectura de X.Zhang (\cite[\S 7.6]{Zhang18})
afirma que os coeficientes da classe $\csm(M_k(m,n)) \in A_*\P^{mn-1}$ são todos não-negativos.
Veremos a seguir que nossa conjectura tem como 
consequência uma propriedade de positividade similar.

\begin{cor}
\label{cor-conjectura}
Assumindo a Conjectura~\ref{conj-A}:
\begin{enumerate}[{\rm(a)}]
\item {\rm{}(Simetria)} Os coeficientes $c_i(U_r)$ são \emph{simétricos}, isto é, 
$c_i(U_r)=c_{2r-i}(U_r)$. 
\item {\rm{}(Positividade)} Seja $\sum_i {\csm}_i({\Sec_r C})[\P^i] \in A_*\P^{2r}$ a classe de Schwartz-MacPherson
da hipersuperfície $\Sec_r C$. Então 
\[
c_i(U_r) > 0 
\qquad\text{e}\qquad
{\csm}_i(\Sec_r C) > 0
\]
para todo $r\geq 1$ e $i=0,\dotsc,2r$.
\end{enumerate}
\end{cor}
\begin{proof} Da expressão em (b) no Teorema~\ref{teo-equivalencias}
decorrem a simetria e a positividade dos coeficientes $c_i(U_r)$.
Provemos agora as desigualdades ${\csm}_i(\Sec_rC)>0$:
usando a inclusão-exclusão para a classe 
de Schwartz-MacPherson, temos que
$\csm({\Sec_r C}) = \csm({\P^{2r}}) - \csm({U_r})$.
Como  
\[
\csm({\P^{2r}}) = c(T\P^{2r})\cap [\P^{2r}] 
= (1+h)^{2r+1} = \sum_{i=0}^{2r}\textstyle\binom{2r+1}{2r-i}[\P^{i}]
 \in A_*\P^{2r}
\]
basta verificar que para $0\leq i\leq 2r$ tem-se
\[
\binom{2r+1}{2r-i} -
\binom{r}{\left[\frac{i}{2}\right]}\binom{r}{\left[\frac{i+1}{2}\right]}
 > 0
\]
que é elementar, e pode ser encontrada no singelo Lema~\ref{lema-A4}.
\end{proof}

Agora mostramos como as fórmulas obtidas no Teorema~\ref{teo-formulas}
nos levam a uma maneira mais simples e elegante de reformular a Conjectura~\ref{conj-A}.
De maneira breve,
basta pedir que $d_i(\phi_r)$ seja polinomial em $r$, sem condições sobre seu grau;
e inspirados pelo Corolário~\ref{cor-conjectura}, trocamos a exigência de que $c_i(U_r)$ seja polinomial por uma condição 
de simetria nos coeficientes, para cada $r$ fixado. 
Poderíamos ter incluído esta reformulação na lista do Teorema~\ref{teo-equivalencias},
mas preferimos destacá-la em um resultado a parte.

\begin{prop}
\label{conj-A'}
Vale a Conjectura~\ref{conj-A} se e somente se as seguintes condições são válidas:
\begin{enumerate}[{\rm(a')}]
\item $c_i(U_r)=c_{2r-i}(U_r)$ para todo $i=0,\dotsc,2r$.
\item Para cada $i\geq0$ fixado, $d_i(\phi_r)$ é uma função polinomial em $r$.
\end{enumerate}
\end{prop}
\begin{proof}
Assumindo a Conjectura~\ref{conj-A}, temos que (b') é imediata;
e 
a condição (a') segue da fórmula para $c_i(U_r)$ no item (c)
da lista de equivalências do Teorema~\ref{teo-equivalencias}.
Para a recíproca: 
\medskip

\noindent\textbf{Afirmação}. 
\emph{Se vale (b'), então $\deg(d_i(\phi_r))\leq i$ para todo $i$
(ou seja, vale a condição (b) da Conjectura~\ref{conj-A}).}
\smallskip

De fato, é uma 
consequência do teorema do índice de Hodge que a sequência
de graus projetivos de um mapa racional  $\P^m\dashrightarrow\P^n$ qualquer
satisfaz duas propriedades
(\cite[Theorem 21]{Huh12}, \cite[Section 1.6]{Lazarsfeld}, \cite{Teissier}):
\begin{enumerate}[(i)]
\item é \emph{log-côncava}, o que significa
\(
d_{i-1}(\phi_r)d_{i+1}(\phi_r) \leq d_i(\phi_r)^2 
 \text{\quad para $i=1,\dotsc,2r-1$}.
\)
\item não possui \emph{zeros internos}.
\end{enumerate}
Recorde que o último grau projetivo $d_{2r}(\phi_r)$ é não-nulo (Corolário~\ref{cor-grad-dominante}) e como $d_0(\phi_r)=1$
segue de (ii) que todos os outros graus projetivos também são não-nulos;
temos daí por (i) uma cadeia de 
desigualdades
\[
\frac{d_0(\phi_r)}{d_1(\phi_r)} \leq \frac{d_1(\phi_r)}{d_2(\phi_r)} \leq \frac{d_2(\phi_r)}{d_3(\phi_r)} \leq \cdots
\leq \frac{d_{2r-1}(\phi_r)}{d_{2r}(\phi_r)}.
\]
Como $d_0(\phi_r)=1$ e $d_1(\phi_r)=r$ temos
indutivamente que $d_i(\phi_r)\leq d_1(\phi_r)^i = r^i$, o que
demonstra nossa afirmação.
\medskip

Agora segue das relações \eqref{eq-aluffi} de Aluffi que
\[
c_{2r-i}(U_r) = \sum_{j=0}^i (-1)^j \binom{2r-j}{i-j} d_j(\phi_r)
\]
que é portanto, pela nossa afirmação, polinomial em $r$ de grau $\leq i-j+j=i$.
Daí a parte (a) da Conjectura~\ref{conj-A} segue imediatamente da nossa condição
de simetria (a'), o que termina a demonstração.
\end{proof}

\subsection{Evidências}\label{secao-evidencias}

Para finalizar, apresentamos algumas evidências que apontam para a validade da Conjectura~\ref{conj-A}. Para $r\geq 1$ fixado:

\begin{enumerate}
\item Na Proposição~\ref{prop-graus} calculamos os graus projetivos $d_i(\phi_r)$, para $i=0, 1, 2, 3, 4$
e as expressões coincidem com as respectivas fórmulas obtidas no item (b) do Teorema~\ref{teo-equivalencias}.

\item Os resultados obtidos computacionalmente via \Mdois{} para $r\leq 5$ (veja ~\ref{subsecao-r<=5} no Apêndice) coincidem com os
valores obtidos a partir das expressões no Teorema~\ref{teo-equivalencias}.

\item Como $\euler(\Sec_r C)=2r$ (Teorema~\ref{teo-Euler_Char_Sec_k}), que como sabemos é o grau da classe $\csm(U_r)$, tem-se
\[
c_0(U_r) = \euler(\P^{2r}\setminus \Sec_r C) = 1
\]
em acordo com a fórmula no item (c) do Teorema~\ref{teo-equivalencias}.

\item No Corolário~\ref{euler_genericsection_sec} calculamos
a característica de Euler de uma seção hiperplana genérica
de $\Sec_rC \subset \P^{2r}$ e, usando o princípio de inclusão-exclusão, obtemos
\[
\euler(U_r\cap H)=1-\cat_r. 
\]
Se vale a Conjectura~\ref{conj-A}, então os números $c_i(r)$ da fórmula~\eqref{eq-formula1}
fornecem os coeficientes da classe $\csm(U_r)$, como vimos no Teorema~\ref{teo-equivalencias}. 
Decorre então Observação~\ref{CSMsecao} que a soma alternada 
de $c_1(r),\dotsc,c_{2r}(r)$ deveria ser igual a 
$\euler(U_r\cap H)$; e como $c_0(r)=1$, \emph{deveríamos} ter
\[
c_0(r)-c_1(r)+c_2(r)-\dots -c_{2r-1}(r) + c_{2r}(r) = \cat_r
\]
ou seja,
\[
\sum_{i=0}^{2r}(-1)^i \binom{r}{\left[\frac{i}{2}\right]}\binom{r}{\left[\frac{i+1}{2}\right]} = \cat_r.
\]
E de fato esta identidade é verdadeira:
\begin{align*}
\sum_{i=0}^{2r}(-1)^i \binom{r}{\left[\frac{i}{2}\right]}\binom{r}{\left[\frac{i+1}{2}\right]}  &= \sum_{i=0}^{r} \binom{r}{i}^2 - \sum_{i=0}^{r-1} \binom{r}{i}\binom{r}{i+1} \\
&= \binom{2r}{r}-\binom{2r}{r-1} \\
&= \dfrac{1}{r+1}\binom{2r}{r},
\end{align*}
onde a penúltima igualdade é obtida por indução.
\end{enumerate}

\appendix
%
%
%

\chapter{Combinatória}
\label{cap-combinatoria}
\label{apA}

\section{Funções racionais}
\label{secao-funcoesracionais}
Dada uma função polinomial $f\colon \N\to\C$ de grau $\leq d$, então 
a sua função geradora 
\[
F = \sum_{n=0}^\infty f(n) x^n \qquad \in \C[[x]]
\]
admite uma representação racional $F=P(x)/Q(x)$, onde $P,Q$ 
são polinômios. Isto é bem conhecido (uma boa referência é
\cite{Stanley}) e apresentamos aqui 
uma demonstração curta.
É conveniente considerar a base
\(
P_i(x)=\binom{x+i}{i} (\text{$i=0,1,2,\dots$}) 
\)
do $\C$-espaço vetorial $\C[x]$.

\begin{lema}
\label{lema-A1} Mantenha a notação acima. Então $F = P(x)/(1-x)^{d+1}$
onde $P(x)\in\C[x]$ é um polinômio de grau $\leq d$. Mais precisamente,
se $f(x)=b_0P_0(x) + \dotsb + b_dP_d(x)$, então
\[
F = \frac{b_0}{1-x} + \frac{b_1}{(1-x)^2} + \dotsb + \frac{b_d}{(1-x)^{d+1}}\qquad.
\]
\end{lema}
\begin{proof}
Para $k\geq 0$, considere a identidade
\[
\frac{1}{(1-x)^{k+1}} = \sum_{n=0}^\infty \binom{n+k}{k}x^n 
= \sum_{n=0}^\infty P_k(n)\,x^n
 \]
(basta tomar a $k$-ésima derivada de $(1-x)^{-1}$). Daí
\[
F = \sum_{n=0}^\infty f(n)\,x^n 
= \sum_{n=0}^\infty \,\sum_{k=0}^d b_kP_k(n)\,x^n
= \sum_{k=0}^d b_k\,\sum_{n=0}^\infty P_k(n)\,x^n
= \sum_{k=0}^d \frac{b_k}{(1-x)^{k+1}}. 
\]
\end{proof}

\section{Identidades binomiais}
%

\begin{lema}
\label{lema-A4}
Seja $r\geq 1$ e suponha $0\leq i \leq 2r$.
Então:
\[
\binom{2r+1}{2r-i} >
\binom{r}{\left[\frac{i}{2}\right]}\binom{r}{\left[\frac{i+1}{2}\right]}.
\]
\end{lema}
\begin{proof}
Se $i=2r$, não há nada a fazer. Suponha portanto $0\leq i<2r$.
Daí, como
\[
\binom{2r+1}{2r-i} = \binom{2r+1}{i+1}= \frac{2r+1}{i+1}\binom{2r}{i} > \binom{2r}{i},
\]
é suficiente mostrar que 
\[
\binom{2r}{i} \geq 
\binom{r}{\left[\frac{i}{2}\right]}\binom{r}{\left[\frac{i+1}{2}\right]}.
\]
E isto se segue observando-se que $[i/2]+[(i+1)/2]=i$ e do fato geral de que 
para quaisquer $a,b$ com $0\leq a,b \leq r$ tem-se
\[
\binom{2r}{a+b}\geq \binom{r}{a}\binom{r}{b}
\]
obtida comparando-se os coeficientes de $x^{a+b}$ 
na expansão de $(1+x)^{2r}$
e de $((1+x)^r)^2$.
\end{proof}
\begin{lema} 
\label{lema-A3}
 Sejam $\cat_k=\frac{1}{k+1}\binom{2k}{k}$ os  números
de Catalan. Então, para quaisquer $i,r\geq 0$,
  \begin{equation*}
 \label{eq-c_i}
 \sum_{k=0}^{r} (-1)^k\binom{2r-2k}{i-k}\binom{r}{k} \cat_k
 =
 \binom{r}{\left[\frac{i}{2}\right]}
\binom{r}{\left[\frac{i+1}{2}\right]}.
\end{equation*}
\end{lema}
\begin{proof}
Nossa prova consiste em mostrar que estas duas sequências
possuem a mesma função geradora, isto é, que as séries
\[
f(x,y):=
\sum_{i,r\geq0} 
\left[
\sum_{k=0}^{r} (-1)^k\binom{2r-2k}{i-k} \binom{r}{k}\cat_k 
\right]
x^ry^i
\quad\text{ e }\quad
h(x,y):=\sum_{i,r\geq0} \binomci x^ry^i
\]
de fato coincidem. Uma vez que já calculamos  
$f(x,y)$ (compare as fórmulas \eqref{eq-f1} e \eqref{eq-f2}), 
a saber
\begin{equation}
\label{eq-f1a}
f(x,y) = 
\frac{1-\sqrt{1-\frac{-4xy}{1-x(1+y)^2}}}{-2xy}
\end{equation}
passamos ao cálculo de $h(x,y)$.

Visando simplificar os binômios envolvidos e tornar
o problema mais acessível, dividimos a série $h(x,y)$ em duas, de acordo com 
a paridade dos expoentes em $y$. 
Definimos
\[
h_1(x,z):=\sum_{r,a=0}^\infty \binom{r}{a}^2 x^rz^a
\qquad\text{e}\qquad
h_2(x,z):=\sum_{\substack{r=0\\ a=1}}^\infty 
\binom{r}{a-1}\binom{r}{a} x^rz^a
\]
Fazendo $z=y^2$, obtemos $h(x,y)=h_1(x,y^2)+y^{-1}h_2(x,y^2)$.

A segunda série está relacionada aos célebres
\emph{números de Narayana}:
\[
N_{r,a}=\frac{1}{r}\binom{r}{a-1}\binom{r}{a} = 
        \frac{1}{a}\binom{r-1}{a-1}\binom{r-1}{a},
        \quad\qquad r\geq 1,\ a=1,\dotsc,r.
\]
Tome 
\(
N(x,z)=\sum_{r,a=1}^\infty N_{r,a} x^rz^a
\)
a função geradora associada. Levando em consideração que $rN_{r,a}=\tbinom{r}{a-1}\tbinom{r}{a}$, obtemos
\[
h_2(x,z)=\sum_{\substack{r=1\\ a=1}}^\infty 
rN_{r,a} x^rz^a = x\frac{\partial N}{\partial x}(x,z)
\]
o que nos conduz a
\begin{equation}
\label{eq-h}
h(x,y) = h_1(x,y^2) + \frac{x}{y}\frac{\partial N}{\partial x}(x,y^2).
\end{equation}
Felizmente as funções geradoras para as séries do lado direito são conhecidas! Para
a primeira (\cite[A008459]{OEIS}):
\[
h_1(x,z)=\frac{1}
{\sqrt{x^2z^2-2(x^2+x)z+x^2-2x+1}}
\]
e para os números de Narayana, (\cite{Wang} ou \cite[A001263]{OEIS}):
\[
N(x,z) = \frac{1-x(1+z)-\sqrt{(1-x(1+z))^2-4x^2z}}{2x}.
\]
Substituindo em \eqref{eq-h} e após uma longa sessão 
de simplificações, obtemos
%
%
%
\[
h(x,y)=
\frac{1-x(1+y)^2-\sqrt{(1-x(1+y)^2)(1-x(1-y)^2)}}
{-2xy(1-x(1+y)^2)}.
\]
Desenvolvendo a fórmula para $f(x,y)$ 
em \eqref{eq-f1a}, verifica-se que ela coincide
com a expressão para $h(x,y)$ que acabamos 
de encontrar, o que conclui a demonstração.
\end{proof}
%
%
%

\chapter{Códigos}
\label{apB}

Apresentamos aqui códigos para alguns dos cálculos necessários
na preparação do texto. Utilizamos o programa \Mdois\  \cite{M2}, versão~1.15.
Os códigos podem ser executados online, basta acessar o endereço
\url{web.macaulay2.com}
e copiar e colar. 

\section{Códigos auxiliares}
\label{secao-codigos}
\subsection{Graus projetivos e classe $\csm$ para $r\leq 5$}
\label{subsecao-r<=5}
Seja $\Sec_r C\subset\P^{2r}$ a secante de $r$ pontos 
da curva racional normal $C\subset\P^{2r}$.
Tome $\phi_r\colon\P^{2r}\dashrightarrow\P^{2r}$ o mapa gradiente
associado a esta hipersuperfície. Apresentamos o código para calcular os graus projetivos
destes mapas e, usando as relações de Aluffi~\eqref{eq-aluffi-paper},
a classe $\csm$ do aberto $U_r=\P^{2r}\setminus \Sec_rC$.
Para o cálculo dos graus projetivos há diversas implementações;
escolhemos o pacote ``Cremona'' 
desenvolvido por G. Staglianò \cite{Stag18}.
\begin{verbatim}
Macaulay2, version 1.15
with packages: ConwayPolynomials, Elimination, IntegralClosure, InverseSystems, LLLBases,
               PrimaryDecomposition, ReesAlgebra, TangentCone

restart
loadPackage "Cremona"
graus = {{1}};
rmax = 5;
for r from 1 to rmax do (
P = QQ[x_0..x_(2*r)];
m = matrix for i from 0 to r list for j from 0 to r list x_(i+j); -- Hankel
f = det m;
J = ideal jacobian ideal f;
phi = map(P,P,gens J);
graus = append(graus, projectiveDegrees(phi));
);
<< "graus projetivos phi_r:" << endl;
for r from 1 to rmax do (
<< "r=" << r << ": "<< graus_r << endl;
);
-- tendo os graus projetivos, usamos as relacoes de Aluffi para
-- obter a classe CSM
<< "classe CSM(U_r):" << endl;
for r from 1 to rmax do (
c = for i from 0 to 2*r list
     sum(0..2*r-i,j->(-1)^j*binomial(2*r-j,i)*graus_r_j); 
<< "r=" << r << ": "<< c << endl;
);

graus projetivos phi_r:
r=1: {1, 1, 1}
r=2: {1, 2, 4, 4, 2}
r=3: {1, 3, 9, 17, 21, 15, 5}
r=4: {1, 4, 16, 44, 86, 116, 104, 56, 14}
r=5: {1, 5, 25, 90, 240, 472, 680, 700, 490, 210, 42}
classe CSM(U_r):
r=1: {1, 1, 1}
r=2: {1, 2, 4, 2, 1}
r=3: {1, 3, 9, 9, 9, 3, 1}
r=4: {1, 4, 16, 24, 36, 24, 16, 4, 1}
r=5: {1, 5, 25, 50, 100, 100, 100, 50, 25, 5, 1}
\end{verbatim}

\subsection{Lugar singular, $\P^4$}
Sejam $C\subset\P^4$ a curva racional normal e $\Sec_2 C$ a sua secante.
Verificamos abaixo que, embora o ideal de $C$ não coincida 
com o ideal jacobiano de $\Sec_2 C$, estes dois ideais tem a mesma saturação,
ou seja, vale a igualdade $\sing(\Sec_2 C) = C$ como esquemas.
Omitimos trechos desnecessários da saída.

\begin{verbatim}
restart
i1 : P4 = QQ[x_0..x_4]

i2 : H2 = matrix{{x_0..x_3},{x_1..x_4}}

o2 = | x_0 x_1 x_2 x_3 |
     | x_1 x_2 x_3 x_4 |

i3 : C = minors(2,H2)

o3 = ideal of P4

i4 : H3 = matrix{{x_0..x_2},{x_1..x_3},{x_2..x_4}}

o4 = | x_0 x_1 x_2 |
     | x_1 x_2 x_3 |
     | x_2 x_3 x_4 |

i5 : Sec2 = ideal det(H3)

              3                2    2
o5 = ideal(- x  + 2x x x  - x x  - x x  + x x x )
              2     1 2 3    0 3    1 4    0 2 4

i6 : J = ideal jacobian Sec2

               2                            2 
o6 = ideal (- x  + x x , 2x x  - 2x x , - 3x  + 2x x  + x x ,
               3    2 4    2 3     1 4      2     1 3    0 4 
                           2  
         2x x  - 2x x , - x  + x x )
           1 2     0 3     1    0 2

i7 : J == C             -- ideais diferentes

o7 = false

i8 : saturate J == C    -- esquemas iguais

o8 = true
\end{verbatim}

\subsection{Lugar singular, $\P^6$}
\label{exemplo-nao-radical}
Entretanto, já para a curva racional normal $C\subset\P^6$,
os esquemas $\Sec_2 C$ e $\sing(\Sec_3 C)$ não
coincidem embora tenham a mesma parte reduzida,
isto é, sejam iguais como conjuntos. De fato,
acreditamos que para $n\geq 5$ e para $2\leq k\leq [n/2]$,
tem-se que $\Sec_{k-1} C$ e $\sing(\Sec_k C)\subset\P^n$ não coincidem como esquemas
embora definam o mesmo conjunto; veja a Proposição~\ref{props_rnc2}.
\begin{verbatim}
i9 : restart

i1 : P6 = QQ[x_0..x_6]

o1 = P6

o1 : PolynomialRing

i2 : H3 = matrix{{x_0..x_4},{x_1..x_5},{x_2..x_6}}

o2 = | x_0 x_1 x_2 x_3 x_4 |
     | x_1 x_2 x_3 x_4 x_5 |
     | x_2 x_3 x_4 x_5 x_6 |

i3 : Sec2 = minors(3,H3)

o3 : Ideal of P6

i4 : Sec3 = ideal det matrix{{x_0..x_3},{x_1..x_4},{x_2..x_5},{x_3..x_6}}

o4 : Ideal of P6

i5 : J = ideal jacobian Sec3

o5 : Ideal of P6

i6 : saturate J == saturate Sec2  -- esquemas diferentes

o6 = false

i7 : radical J == J -- ideal jacobiano nao e' radical

o7 : false

i8 : radical J == Sec2    -- mesmo lugar reduzido

o8 = true
\end{verbatim}

\section{Algoritmo principal}
\label{secao-alg1}
Apresentamos aqui uma implementação do algoritmo descrito no Teorema~\ref{algoritmo}.
\begin{small}
\begin{verbatim}
--
-- interpolacao de Lagrange
--
interpola = (lista,x) -> (
--
-- recebe lista da forma { {a_0,b_0},..,{a_n,b_n} }
-- retorna polinomio p de grau <=n na variavel x tal que p(a_i)=b_i
--
	I := 0..#lista-1;
	return sum( I, i->lista_i_1*
	   product(toList(set(I)-{i}), j->(x-lista_j_0)/(lista_i_0-lista_j_0))
	);
) 
--
-- numeros de Catalan
--
catalan = n -> substitute(binomial(2*n,n)/(n+1),ZZ);
--
-- Algoritmo principal:
-- calcula as sequencias c_i(r) e d_i(r) caracterizadas pelo Teorema 4.2, 
-- para r percorrendo os valores 0,1,...,N
--
calcula = N -> (
t:= symbol t;
R := QQ[t];
c=mutableMatrix(QQ,2*N+1,N+1);
d=mutableMatrix(QQ,2*N+1,N+1);
polc=mutableMatrix(R,N+1,1); -- vetor com os polinomios c_i(t)
pold=mutableMatrix(R,N+1,1); -- vetor com os polinomios d_i(t)
c_(0,0)=1;    d_(0,0)=1;    
polc_(0,0)=1; pold_(0,0)=1;
--
-- laco principal: r indexa a linha, espaco ambiente P^{2r}
--
for r from 1 to N do (
--
-- atribui valores das colunas 0..r-1, herdados do passo anterior
--
for i from 0 to r-1 do (
d_(i,r)=substitute(pold_(i,0)+0*t,t=>r);
c_(i,r)=substitute(polc_(i,0)+0*t,t=>r);
);
--
-- determina colunas r+1,...,2r na matriz c, usando as relacoes de Aluffi
--
for i from 0 to r-1 do (
c_(2*r-i,r)=sum( 0..i, j->(-1)^j*binomial(2*r-j,2*r-i)*d_(j,r) );
);
--
-- determina c_{r,r}
--
s:=sum(0..r-1, j->(-1)^j*c_(j,r))  +  sum(r+1..2*r, j->(-1)^j*c_(j,r));
c_(r,r)=(-1)^r*(catalan(r) - s);
--
-- agora temos toda a linha r na matriz c 
-- calculamos a linha r na matriz d, usando as relacoes de Aluffi
--
for i from r to 2*r do (
d_(i,r) = sum( 2*r-i..2*r, j->(-1)^j*binomial(j,2*r-i)*c_(j,r) );
);
--
-- interpola e obtem os polinomios da coluna r
-- em ambas as matrizes
--
polc_(r,0) = interpola( for linha in 0..r list {linha,c_(r,linha)}, t );
pold_(r,0) = interpola( for linha in 0..r list {linha,d_(r,linha)}, t );
); -- for r
--
--
c=transpose c;
d=transpose d;
<< "N = " << N << ". Respostas em c, d, polc, pold" << endl;
) -- calcula
\end{verbatim}
\end{small}

\subsection{Resultados}
\label{algoritmo-resultados}
\begin{verbatim}
i4 : calcula 7
N = 7. Respostas em c, d, polc, pold

i5 : c

o5 = | 1 . .  .   .   .   .    .    .    .   .   .   .  . . |
     | 1 1 1  .   .   .   .    .    .    .   .   .   .  . . |
     | 1 2 4  2   1   .   .    .    .    .   .   .   .  . . |
     | 1 3 9  9   9   3   1    .    .    .   .   .   .  . . |
     | 1 4 16 24  36  24  16   4    1    .   .   .   .  . . |
     | 1 5 25 50  100 100 100  50   25   5   1   .   .  . . |
     | 1 6 36 90  225 300 400  300  225  90  36  6   1  . . |
     | 1 7 49 147 441 735 1225 1225 1225 735 441 147 49 7 1 |

o5 : MutableMatrix

i6 : d

o6 = | 1 . .  .   .    .    .    .     .     .     .     .     .    .    .   |
     | 1 1 1  .   .    .    .    .     .     .     .     .     .    .    .   |
     | 1 2 4  4   2    .    .    .     .     .     .     .     .    .    .   |
     | 1 3 9  17  21   15   5    .     .     .     .     .     .    .    .   |
     | 1 4 16 44  86   116  104  56    14    .     .     .     .    .    .   |
     | 1 5 25 90  240  472  680  700   490   210   42    .     .    .    .   |
     | 1 6 36 160 540  1392 2752 4152  4710  3900  2232  792   132  .    .   |
     | 1 7 49 259 1057 3367 8449 16753 26173 31899 29757 20559 9933 3003 429 |

o6 : MutableMatrix

i7 : polc

o7 = | 1                                             |
     | t                                             |
     | t2                                            |
     | 1/2t3-1/2t2                                   |
     | 1/4t4-1/2t3+1/4t2                             |
     | 1/12t5-1/3t4+5/12t3-1/6t2                     |
     | 1/36t6-1/6t5+13/36t4-1/3t3+1/9t2              |
     | 1/144t7-1/16t6+31/144t5-17/48t4+5/18t3-1/12t2 |

o7 : MutableMatrix

i8 : pold

o8 = | 1                                                       |
     | t                                                       |
     | t2                                                      |
     | 5/6t3-1/2t2-1/3t                                        |
     | 7/12t4-t3-1/12t2+1/2t                                   |
     | 7/20t5-7/6t4+3/4t3+2/3t2-3/5t                           |
     | 11/60t6-t5+19/12t4-1/6t3-19/15t2+2/3t                   |
     | 143/1680t7-11/16t6+91/48t5-83/48t4-23/30t3+23/12t2-5/7t |

o8 : MutableMatrix

i9 : time calcula 50          -- ate' P^{100}...
N = 50. Respostas em c, d, polc, pold
     -- used 10.5821 seconds


\end{verbatim}



\end{document}